\documentclass[reqno, 10pt]{amsart}

\usepackage[usenames,dvipsnames]{color}
\usepackage[pdfauthor={Andrea Appel, Valerio Toledano Laredo}, 
pdftitle={A 2--categorical extension of Etingof--Kazhdan quantization}]{hyperref}
\hypersetup{
    colorlinks=true,       			
    linkcolor=blue,          			
    citecolor=blue,		 		
    filecolor=blue,      				
    urlcolor=blue           			
}
\usepackage{amsmath, amssymb, empheq, stmaryrd}
\usepackage{tikz}
\usetikzlibrary{shapes,backgrounds}
\usepackage[all,cmtip,knot]{xy}
\xyoption{arc}

\newcommand{\Omit}[1]{}

\newtheorem*{theorem}{Theorem}
\newtheorem*{proposition}{Proposition}
\newtheorem*{corollary}{Corollary}
\newtheorem*{lemma}{Lemma}

\theoremstyle{definition}
\newtheorem*{definition}{Definition}

\numberwithin{equation}{section}

\newenvironment{pf}{\paragraph{{\sc Proof}}}{\qed\par\medskip}

\newcommand{\remark}{{\bf Remark.\ }}
\newcommand {\example}{{\sc Example.\ }}

\newcommand {\IN}{\mathbb{N}}

\newcommand {\C}{\mathcal C}
\newcommand {\D}{\mathcal D}
\newcommand {\E}{\mathcal E}
\newcommand {\F}{\mathcal F}
\newcommand {\G}{\mathcal G}

\newcommand {\N}{\mathcal N}
\renewcommand {\O}{\mathcal O}

\newcommand {\Q}{\mathcal Q}
\newcommand {\R}{\mathcal R}

\newcommand {\V}{\mathcal V}

\renewcommand {\a}{\mathfrak a}
\renewcommand {\b}{\mathfrak b}
\renewcommand {\c}{\mathfrak c}
\renewcommand {\d}{\mathfrak d}

\newcommand {\g}{\mathfrak{g}}

\newcommand {\m}{\mathfrak m}

\renewcommand {\SS}{\mathfrak{S}}

\newcommand {\sfA}{\mathsf{A}}

\newcommand {\sfD}{\mathsf{D}}

\newcommand {\ie}{{\it i.e., }}

\newcommand {\fd}{finite--dimensional }

\newcommand {\rhs}{right--hand side }

\newcommand {\wrt}{with respect to }

\newcommand {\ol}{\overline}
\newcommand {\wt}{\widetilde}
\newcommand {\wh}{\widehat}

\newcommand {\sDJ}{^{\scriptscriptstyle{\operatorname{DJ}}}}

\newcommand {\DYt}{Drinfeld--Yetter }

\newcommand {\nEK}{Etingof--Kazhdan }

\newcommand {\KM}{Kac--Moody }

\newcommand{\Lp}{\mathfrak{p}}

\def\iip#1#2{\langle{#1},{#2}\rangle}

\def\slantfrac#1#2{\kern.2em\raise.2em\hbox{$#1$}\kern-.2em\left/\lower.4em\hbox{$#2$}\kern-.2em\right.}
\def\backslantfrac#1#2{\kern.2em\lower.4em\hbox{$#1$}\kern-.3em\left\backslash\raise.4em\hbox{$#2$}\right.}

\DeclareMathOperator{\Hom}{Hom}
\DeclareMathOperator{\End}{End}
\DeclareMathOperator{\Ker}{Ker}
\DeclareMathOperator{\colim}{colim}
\DeclareMathOperator{\Ind}{Ind}
\DeclareMathOperator{\Res}{Res}
\DeclareMathOperator{\id}{id}

\DeclareMathOperator{\Rep}{Rep}
\DeclareMathOperator{\obj}{Obj}
\DeclareMathOperator{\vect}{\operatorname{Vect}}
\DeclareMathOperator{\vectk}{\operatorname{Vect}_\sfk}

\DeclareMathOperator{\ad}{ad}

\newcommand {\fml}{[\negthinspace[\hbar]\negthinspace]}

\newcommand {\half}[1]{\frac{#1}{2}}

\newcommand{\ctp}{\wh{\otimes}}
\newcommand{\ten}{\otimes}

\newcommand {\ul}[1]{\underline{#1}}

\newcommand{\ind}{\operatorname{Ind}}

\newcommand{\ima}{\operatorname{Im}}

\newcommand {\Ug}{U\g}

\newcommand {\Uhg}{U_{\hbar}\g}

\newcommand {\aand}{\qquad\text{and}\qquad}

\newcommand {\qc}{quasi--Coxeter }

\newcommand {\qtqbas}{quasitriangular quasibialgebras }

\newcommand {\iD}{i}
\newcommand{\PhiU}{\Phi_{\b}}
\newcommand{\PhiD}{\Phi_{\a}}

\newcommand{\OmegaD}{\Omega_{\a}}

\newcommand{\Lpp}{\mathfrak{p_+}}
\newcommand{\Lpm}{\mathfrak{p_-}}
\newcommand{\Lppm}{\mathfrak{p}_{\pm}}

\newcommand{\Lmm}{\mathfrak{m_-}}
\newcommand{\Lmpm}{\mathfrak{m_{\pm}}}

\renewcommand{\1}{\mathbf{1}}
\newcommand{\Mp}{M_+}
\newcommand{\Mm}{M_-}
\newcommand{\Mpv}{\Mp^\vee}
\newcommand{\Mpd}{\Mp^*}
\newcommand{\Lm}{L_-}
\newcommand{\Np}{N_+}
\newcommand{\Npd}{\Np^*}
\newcommand{\Npv}{\Np^\vee}

\newcommand{\um}{1_-}
\newcommand{\up}{1_+}
\newcommand{\coup}{\epsilon_+}
\newcommand{\cou}{\epsilon}

\newcommand{\ipd}{i_+^{\vee}}
\newcommand{\im}{i_-}
\newcommand{\ips}{i_+^*}
\newcommand{\ip}{i_+}

\newcommand{\Lmh}{\Lm^\hbar}
\newcommand{\Nph}{\Np^{\hbar}}

\newcommand{\Npvh}{(\Np^\vee)^{\hbar}}

\def\Ue#1{U#1}
\def\UE#1{U(#1)}

\newcommand{\Eq}[1]{\E_{#1}}
\newcommand{\hEq}[2]{\E_{#1}^{#2}}

\newcommand{\DrY}[1]{\mathsf{DY}_{#1}}
\newcommand{\aDrY}[1]{\mathsf{DY}\aadm_{#1}}
\newcommand{\aadm}{^{\scriptscriptstyle{\operatorname{adm}}}}
\newcommand{\hDrY}[2]{\mathsf{DY}_{#1}^{#2}}
\newcommand{\ff}{\mathsf{f}}

\newcommand{\gb}{\g_{\b}}
\newcommand{\ga}{\g_{\a}}

\newcommand {\op}{^{\scriptscriptstyle{\operatorname{op}}}}

\newcommand {\<}{\langle}
\renewcommand {\>}{\rangle}

\newcommand {\VecK}{\mathbf{\vectk}}

\newcommand {\sfk}{\mathsf{k}}
\newcommand {\sfK}{\mathsf{K}}

\newcommand {\Kvect}{\vectk}

\newcommand{\PROP}{{\sf PROP}}

\newcommand{\propic}{propic }

\newcommand{\Sym}{\operatorname{Sym}}

\newcommand{\LA}{{\sf LA}}

\newcommand{\Alg}{{\sf Alg}}

\newcommand{\BA}{\sf{BA}}
\newcommand{\preLA}{{\sf LA}}
\newcommand{\preAlg}{{\sf Alg}}

\newcommand{\LBA}{{\sf LBA}}
\newcommand{\PLBA}{{\sf{PLBA}}}
\newcommand{\DY}{{\sf DY}}

\newcommand{\sfEnd}[1]{\mathsf{End}\left(#1\right)}

\newcommand{\Cat}{\mathsf{Cat}}

\newcommand{\sfad}{\mathsf{ad}}

\renewcommand{\LBA}{{\sf LBA}}
\renewcommand{\PROP}{\sf PROP}
\renewcommand{\PLBA}{{\sf{PLBA}}}
\renewcommand{\DY}[1]{{\sf DY}_{#1}}

\renewcommand{\int}{^{\scriptscriptstyle{\operatorname{int}}}}

\newcommand {\sfP}{{\mathsf P}}

\newcommand {\rel}{^{\operatorname{\scriptstyle{rel}}}}

\newcommand{\gu}{\gb}
\newcommand{\gup}{\b_{+}}
\newcommand{\gum}{\b_{-}}
\newcommand{\gupm}{\b_{\pm}}
\newcommand{\gump}{\b_{\mp}}

\newcommand{\gd}{\ga}
\newcommand{\gdp}{\a_{+}}
\newcommand{\gdm}{\a_{-}}
\newcommand{\gdpm}{\a_{\pm}}
\newcommand{\gdmp}{\a_{\mp}}

\newcommand{\Rd}{\R}

\newcommand {\Uha}{U_\hbar\a}

\newcommand {\Uhb}{U_\hbar\b}
\newcommand {\Uhc}{U_\hbar\c}

\newcommand {\sLBA}{\mathsf{sLBA}}
\newcommand {\sfLBA}{\mathsf{LBA}}

\newcommand {\sQUE}{\mathsf{sQUE}}

\newcommand{\EK}{U_{\hbar}}
\newcommand{\EKeq}[1]{\wt{F}_{#1}}
\newcommand{\EKff}[1]{F_{#1}}

\newcommand{\noEKeq}[1]{\wt{F}_{#1}}
\newcommand{\noEKff}[1]{F_{#1}}

\newcommand{\hFF}[2]{F_{#1,#2}^{\hbar}}
\newcommand{\FF}[2]{F_{#1,#2}}
\newcommand{\wtFF}[2]{\wt{F}_{#1,#2}}
\newcommand{\resped}{_{\a,\b}}
\newcommand{\presped}{_{\pa,\pb}}

\newcommand{\HL}{U_{\hbar}\mathfrak{m}}
\newcommand{\HB}{U\rel_\hbar\gums}
\newcommand{\HBp}{U\rel_\hbar\pb}

\newcommand {\sff}{\mathsf{f}}

\newcommand {\qcc}{quasi--Coxeter category }

\newcommand {\KMA}{Kac--Moody algebra }

\newcommand{\Kar}[1]{{#1}^{\scriptscriptstyle{\mathsf{kar}}}}
\newcommand{\cKar}[1]{\ul{#1}}
\newcommand{\hext}[1]{{#1}{\fml}}

\newcommand{\pb}{[\b]}
\newcommand{\pa}{[\a]}
\newcommand{\ppm}{[\mathfrak{m}]}
\newcommand{\pp}{[\mathfrak{p}]}
\newcommand{\slbah}{\hext{\cKar{\LBA}}}

\newcommand{\uh}{\wt{\iota}}
\newcommand{\couh}{\wt{\epsilon}}
\newcommand{\mh}{\wt{m}}
\newcommand{\Dh}{\wt{\Delta}}
\newcommand{\Sh}{\wt{S}}

\newcommand{\funU}{\G_{\pb}}
\newcommand{\funL}{\G_{\pa}}

\newcommand{\Lmms}{\mathfrak{m}}

\newcommand{\pDrY}[3]{\mathsf{DY}_{#2}^{#3}}

\newcommand{\splbah}{\hext{\cKarPLBAD}}

\newcommand{\prelpsi}[1]{\psi_{#1}^{\pa,\pb}}
\newcommand{\preleta}[1]{\eta_{#1}^{\pa,\pb}}

\newcommand{\PLBAD}{\PLBA^{\negthinspace{+}}}

\newcommand{\pad}{[\a^\sharp]}
\newcommand{\coadd}{\vartriangleleft}
\newcommand{\coad}{\vartriangleright}

\newcommand{\cKarPLBAD}{\cKar{\PLBA}^{\negthinspace{+}}}

\newcommand {\PMm}{[\Mm]}
\newcommand {\PMpv}{[\Mpv]}
\newcommand {\PLm}{[\Lm]}
\newcommand {\PNpv}{[\Npv]}

\newcommand{\ev}{\operatorname{\mathsf{ev}}}

\newcommand{\ellh}{\ell_\hbar}

\newcommand{\Vm}{{[V]}}

\newcommand{\scs}{\scriptscriptstyle}

\newcommand{\pNpv}{[\Npv]}

\newcommand{\dcs}{\coad\negthinspace\coadd}

\newcommand{\Mpvh}{(\Mp^\vee)^{\hbar}}
\newcommand{\Mmh}{\Mm^{\hbar}}
\newcommand{\MmB}{M_B}
\newcommand{\MpB}{\wh{M^{\vee}_B}}
\newcommand{\MpBp}{{M}^{\vee}_B}

\newcommand{\SC}[1]{\mathsf{SC}(#1)}

\newcommand{\Mv}[1]{M_{#1}}
\newcommand{\Mdv}[1]{M_{#1}^{\vee}}

\newcommand {\BCH}{Baker--Campbell--Hausdorff }

\newcommand {\QQUE}{\operatorname{QUE}}
\newcommand {\QQFSH}{\operatorname{QFSH}}

\newcommand{\cP}{\mathsf{UE_{cP}}}

\newcommand{\Fun}{\mathsf{Fun}}

\newcommand{\bfb}{\mathbf{b}}

\newcommand{\aNpvh}{(N_+^{\vee})_{\a}^{\hbar}}
\newcommand{\gLm}[1]{L_{#1}}
\newcommand{\gNm}[1]{N_{#1}^{\vee}}

\newcommand{\QUE}{\mathsf{QUE}}
\newcommand{\qcBA}{\mathsf{qcBA}}

\newcommand{\dcp}[2]{#1\negthinspace\BonC\negthinspace\negthinspace\ConB\negthinspace#2}

\newcommand{\BonC}{\coad}
\newcommand{\ConB}{\coadd}

\title[A $2$--categorical extension of Etingof--Kazhdan quantisation]
{A $2$--categorical extension of Etingof--Kazhdan quantisation}
\author[A. Appel]{Andrea Appel}
\address{School of Mathematics,
University of Edinburgh,
James Clark Maxwell Building, 
Peter Guthrie Tait Road,
Edinburgh, EH9 3FD, UK}
\email{andrea.appel@ed.ac.uk}
\author[V. Toledano Laredo]{Valerio Toledano Laredo}
\address{Department of Mathematics,
Northeastern University,
360 Huntington Avenue,
Boston MA 02115}
\email{V.ToledanoLaredo@neu.edu}
\thanks{The first author was supported in part through the NSF
grant DMS--1255334, and the second through the NSF grant DMS--1505305.}

\begin{document}

\begin{abstract}
Let $\sfk$ be a field of characteristic zero.
In \cite{ek-1}, Etingof and Kazhdan construct a {quantisation} $\Uhb$
of any Lie bialgebra $\b$ over $\sfk$, which depends on the choice of
an associator $\Phi$. They prove moreover that this quantisation is
functorial in $\b$ \cite{ek-2}. Remarkably, the quantum group $\Uhb$
is endowed with a Tannakian equivalence $\noEKff{\b}$ from the braided
tensor category of Drinfeld--Yetter modules over $\b$, with deformed
associativity constraints given by $\Phi$, to that of Drinfeld--Yetter modules 
over $\Uhb$ \cite{ek-6}.
In this paper, we prove that the equivalence $\noEKff{\b}$ is
functorial in $\b$.
\end{abstract}

\maketitle
\setcounter{tocdepth}{1}
\tableofcontents

\newpage

\section{Introduction}

\subsection{} 

Let $\sfk$ be a field of characteristic zero. In \cite{ek-1}, Etingof and
Kazhdan construct a quantisation $\Uhb$ of any Lie bialgebra $\b$
over $\sfk$. The quantisation depends on the choice of an associator
$\Phi$, and has a number of remarkable properties: it is functorial in
$\b$ \cite{ek-2}, compatible with taking doubles and duals \cite{EG}
and, when $\b$ is a symmetrisable \KMA with its standard bialgebra
structure, coincides with the Drinfeld--Jimbo quantum group $U_\hbar
\sDJ\b$ associated to $\b$ \cite{ek-6}.

The quantum group $\Uhb$ is also compatible with another basic
operation, namely taking the tensor category of \DYt modules (see
\S\ref{ss:LBADY}--\S\ref{ss:HADY} below for definitions). Specifically,
it is endowed with a braided tensor equivalence
\[\EKeq\b: \hDrY{\b}{\Phi}\longrightarrow\aDrY{\Uhb}\]
where $\hDrY{\b}{\Phi}$ is the category of \DYt modules over the
Lie bialgebra $\b$, with deformed associativity constraints given
by $\Phi$, and $\aDrY{\Uhb}$ is the category of admissible
\DYt modules over $\Uhb$ \cite{ek-6}. If $\g$ is a symmetrisable
\KM algebra with negative Borel subalgebra $\b$, this implies in
particular the existence of an equivalence $E_\g$ between category
$\O$ representations of $\g$ and those of the quantum group
$U_\hbar\sDJ\g$.

Motivated by the theory of quasi--Coxeter algebras and categories
\cite{vtl-4, ATL1-2}, we prove in this paper that the equivalence
$\EKeq\b$ is itself functorial with respect to $\b$. This shows in
particular that the \nEK equivalence $E_\g$ is compatible \wrt
restriction to a standard Levi subalgebra.

\subsection{} \label{ss:LBADY}

A \DYt module over a Lie bialgebra $\b$ is a triple $(V,\pi,\pi^*)$ such
that $\pi:\b\otimes V\to V$ gives $V$ the structure of a left $\b$--module,
$\pi^*:V\to\b\otimes V$ that of a right $\b$--comodule, and $\pi,\pi^*$
satisfy a compatibility condition \cite{ek-2}. The latter is designed so
as to give rise to a representation of the Drinfeld double $\gb=\b\oplus
\b^*$ of $\b$, with $\phi\in\b^*$ acting on $V$ by $\phi\otimes\id_V\circ
\pi^*$. If $\b$ is finite--dimensional, the symmetric tensor category
$\hDrY{\b}{}$ of such modules coincides with that of representations
of $\gb$, with the coaction of $\b$ on an object $V$ of the latter category
given by $\pi^*(v)=\sum_i b_i\otimes b^i v$, where $\{b_i\},\{b^i\}$ are
dual bases of $\b$ and $\b^*$. For an arbitrary $\b$, $\hDrY{\b}{}$
coincides with the category of {\it equicontinuous}
representations of $\gb$, which are roughly those carrying a locally
finite action of $\b^*$ \cite{ek-1}.

\subsection{} \label{ss:HADY}

A \DYt module over a Hopf algebra $B$ is a triple $(V,\rho,\rho^*)$,
where $\rho:B\otimes V\to V$ is a left $B$--module, $\rho^*:V\to B
\otimes V$ a right $B$--comodule, and $\rho,\rho^*$ satisfy a suitable
compatibility relation \cite{Y,ek-2}. Such modules form a braided
tensor category $\hDrY{B}{}$, with commutativity constraints
$\beta_{U,V}:U\otimes V\to V\otimes U$ given by
\[\beta_{U,V}=
(1\,2)\circ(\rho_U\otimes \id_V)\circ(1\,2)\circ(\id_U\otimes\rho_V^*).\]
If $B$ is finite--dimensional, the category $\hDrY{B}{}$ coincides with
that of representations of the quantum double of $B$ \cite{drin-2}.

A similar statement holds if $B$ is a quantised universal enveloping
algebra, that is a topological Hopf algebra over $\sfk{\fml}$ such
that $B/\hbar B$ is a universal enveloping algebra $U\b$. If $\b$ is
finite--dimensional, representations of the quantum double of $B$
coincide, as a braided tensor category, with the category $\aDrY{B}$ of 
{\it admissible} \DYt modules over $B$, which are those for
which the coaction $\rho^*:V\to B\otimes V$ factors through $B'
\otimes V$, where $B'\subset B$ is the quantised formal group
corresponding to $B$ defined in \cite{drin-2,gav}.
\newpage

\subsection{} \label{ss:EK}

A crucial role in the quantisation of a Lie bialgebra $\b$ is played
by a deformation $\hDrY{\b}{\Phi}$ of the braided tensor category
$\hDrY{\b}{}$ over $\sfK=\sfk{\fml}$, where the commutativity
and associativity constraints are respectively given by
\[(1\,2)\circ\exp(\hbar\Omega/2)\aand\Phi(\hbar\Omega_{12},\hbar\Omega_{23})\]
with $\Omega\in\gb\wh{\otimes}\gb$ the canonical element representing
the inner product. Indeed, \nEK construct
a fiber functor $\EKff\b:\hDrY{\b}{\Phi}\to\VecK$, and obtain
$\Uhb$ as a sub Hopf algebra of $\End(\EKff{\b})$ \cite{ek-1}.
A remarkable feature of $\Uhb$ is that it is functorial \wrt $\b$ \cite{ek-2},
even though the intermediate steps in its construction, in particular
taking the double of $\b$ and considering the category $\hDrY{\b}{}$,
are not.

The quantum group $\Uhb$ possesses another remarkable
feature. Namely, in addition to acting on any $\EKff{\b}(V)$,
$V\in\hDrY{\b}{\Phi}$, it admits an admissible {\it co}action
on $\EKff{\b}(V)$ which is compatible with its action. This
gives rise to a Tannakian lift of $\EKff\b$ as a braided tensor functor
\[\EKeq\b: \hDrY{\b}{\Phi}\longrightarrow\aDrY{\Uhb}{}\]
where the \rhs are the admissible Drinfeld--Yetter modules
over $\Uhb$. Morever, $\EKeq\b$ is an equivalence \cite{ek-6}.

\subsection{} 

It is natural to ask whether the equivalence $\EKeq\b$ is functorial
\wrt $\b$. The goal of this paper is to prove that this is indeed the
case. 

A proper formulation of this statement requires considering a different
class of morphisms between Lie bialgebras, however, since taking
\DYt modules is not functorial in $\b$. Let for this purpose $\sfLBA(\sfk)$
be the (usual) category of Lie bialgebras over $\sfk$, and $\sLBA(\sfk)$
the category whose objects are Lie bialgebras, and morphisms are split
embeddings $\a\hookrightarrow\b$ in $\sfLBA(\sfk)$, that is
\[\Hom_{\sLBA}(\a,\b)=
\{(i,p)\in\Hom_{\mathsf{LBA}}(\a,\b)\times\Hom_{\mathsf{LBA}}(\b,\a)\;|\;
p\circ i=\id_{\a}\}.\]

A morphism in $\sLBA(\sfk)$ gives rise to a restriction functor
$\Res_{\a,\b}:\hDrY{\b}{}\longrightarrow\hDrY{\a}{}$ given by
$\Res_{\a,\b}(V,\pi,\pi^*)=(V,\pi\circ(i\otimes\id),(p\otimes\id)
\circ\pi^*)$.\footnote{In terms of the Drinfeld doubles $\ga,\gb$ of $\a,\b$, a split
embedding $\a\hookrightarrow\b$ corresponds to an isometric morphism
of Lie algebras $j:\ga\to\gb$ such that $j(\a)\subset\b,j(\a^*)\subset\b^*$,
and the transpose $j^t:\gb\to\ga$ restricts to morphisms of Lie algebras
$\b\to\a$ and $\b^*\to\a^*$. Moreover, under the identification of the
categories $\DrY{\a},\DrY{\b}$ with those of equicontinuous modules
over $\ga$ and $\gb$ respectively, the functor $\Res_{\a,\b}$ coincides
with the restriction functor $j^*$.}
This functor admits a natural tensor structure
\[J_{\a,\b}^0=\id:\Res_{\a,\b}(V_1)\otimes\Res_{\a,\b}(V_2)
\longrightarrow\Res_{\a,\b}(V_1\otimes V_2)\]
which clearly satisfies $(\Res_{\a,\b},J_{\a,\b}^0)\circ(\Res_{\b,\c},,J_{\b,\c}^0)
=(\Res_{\a,\c},J_{\a,\c}^0)$ for any chain of split embeddings
$\a\hookrightarrow\b\hookrightarrow\c$.
Thus, the assignment $\b\to\hDrY{\b}{}$ extends to a contravariant
functor from $\sLBA(\sfk)$ to the (1--)category $\Cat_{\sfk}^{\ten}$ of
$\sfk$--linear tensor categories.

\subsection{} 

In the presence of an associator $\Phi$, $(\Res_{\a,\b},J^0_{\a,\b})$
ceases to be a tensor functor, since the associativity constraints on
$\hDrY{\b}{\Phi},\hDrY{\a}{\Phi}$ are given by
$\Phi_\b=\Phi(\hbar\Omega_{12}^\b,\hbar\Omega_{23}^\b)$
and 
$\Phi_\a=
\Phi(\hbar\Omega_{12}^\a,\hbar\Omega_{23}^\a)$
respectively, and are therefore different. Our first main result asserts
that $\Res_{\a,\b}$ can be endowed with a tensor structure $J_{\a,\b}$
compatible with those on $\hDrY{\b}{\Phi}$ and $\hDrY{\a}{\Phi}$.
Moreover, with that tensor structure, the assignment $\b\to\hDrY{\b}{\Phi}$
extends to a {\it 2--functor}. The $2$--functoriality accounts for the fact
that, for a chain $\a\hookrightarrow\b\hookrightarrow\c$, the composition
$(\Res_{\a,\b},J_{\a,\b})\circ(\Res_{\b,\c},J_{\b,\c})$ is not equal
to $(\Res_{\a,\c},J_{\a,\c})$, but only isomorphic to it via a coherent
isomorphism.

Specifically, consider $\sLBA(\sfk)$ as a 2--category with 2--morphisms
given by equalities, and $\Cat_{\sfK}^{\ten}$ as a 2--category in the
usual way (1--morphisms are tensor functors, and 2--morphisms 
natural transformations). Then, the following holds for any associator $\Phi$

\begin{theorem}\label{th:first main}
There is a $2$--functor\footnote{Strictly
speaking $\mathsf{DY}^{\Phi}$ is a {\it pseudo} 2--functor in the terminology
of \cite{lack} since it preserves the composition of 1--morphisms in
$\sLBA(\sfk)$ only up to the coherent isomorphisms $u_{\a,\b,\c}$.}
\[
\mathsf{DY}^{\Phi}:\sLBA(\sfk)\longrightarrow\mathsf{Cat}_{\sfK}^{\ten}
\]
which assigns
\begin{itemize}
\item to any Lie bialgebra $\b$, the tensor category $\hDrY{\b}{\Phi}$, 
\item to any split embedding $\a\hookrightarrow\b$, a tensor structure
$J\resped$ on the restriction functor $\Res_{\a,\b}:\hDrY{\b}{\Phi}\to
\hDrY{\a}{\Phi}$,
\item to any chain $\a\hookrightarrow\b\hookrightarrow\c$, an isomorphism
of tensor functors
\[u_{\a,\b,\c}:(\Res_{\a,\b},J_{\a,\b})\circ (\Res_{\b,\c},J_{\b,\c})\longrightarrow
(\Res_{\a,\c},J_{\a,\c})\]
\end{itemize}
in such a way that, for any chain $\a\hookrightarrow\b\hookrightarrow\c
\hookrightarrow\d$, one has
\begin{equation}\label{eq:u assoc}
u_{\a,\b,\d}\circ u_{\b,\c,\d}=u_{\a,\c,\d}\circ u_{\a,\b,\c}
\end{equation}
as isomorphisms
\[ (\Res_{\a,\b},J_{\a,\b})\circ (\Res_{\b,\c},J_{\b,\c})\circ(\Res_{\c,\d},J_{\c,\d})
 \longrightarrow (\Res_{\a,\d},J_{\a,\d})\]
\end{theorem}

\subsection{}

Having established the correct functoriality of the assignment
$\b\to\DrY{\b}^\Phi$, we turn now to the \nEK equivalence
$\EKeq\b:\hDrY{\b}{\Phi}\to\aDrY{\Uhb}$. Our second main
result is that $\EKeq\b$ is functorial \wrt split embeddings and
that, moreover, it fits within an isomorphism of 2--functors
$\sLBA(\sfk)\to\mathsf{Cat}_\sfK^{\ten}$.

Specifically, let $\sQUE(\sfk)$ be the category of quantised
universal enveloping algebras over $\sfk{\fml}$, with morphisms
given by split embeddings. Taking admissible \DYt modules
yields a functor $\aDrY{}:\sQUE(\sfk)\longrightarrow\Cat_{\sfK}
^{\ten}$, which assigns to a split embedding $A\hookrightarrow B$
the restriction functor $\Res_{A,B}:\aDrY{B}\to\aDrY{A}$ given
by
\[\Res_{A,B}(V,\rho,\rho^*)=(V,\rho\circ(i\otimes\id),(p\otimes\id)\circ\rho^*)\]
and endowed with the trivial tensor structure. On
the other hand, \nEK quantisation gives rise to a functor
$Q^{\Phi}:\sLBA(\sfk)\to\sQUE(\sfk)$.

\begin{theorem}\label{th:second main}
There is an isomorphism of $2$--functors 
\[
\xymatrix@R=.2cm@C=.5cm{
\sLBA(\sfk) \ar[ddrr]_{\mathsf{DY}^{\Phi}} \ar[rrrr]^{Q^{\Phi}} && && \sQUE(\sfk) \ar[ddll]^{\aDrY{}}\ar@{=>}[dlll]\\
&&&&\\
&& \mathsf{Cat}_\sfK^{\ten} &&
}
\]
which assigns to a Lie bialgebra $\b\in\sLBA(\sfk)$ the tensor
equivalence $\EKeq\b:\hDrY{\b}{\Phi}\to\aDrY{\Uhb}$. In
particular,
\begin{itemize}
\item For any split embedding $\a\hookrightarrow\b$, there is a natural
isomorphism $v\resped$ making the following diagram commute
\begin{equation}\label{eq:nat transf}
\xymatrix@C=2cm{
\hDrY{\b}{\Phi} \ar[r]^{\EKeq\b} \ar[d]_{(\Res_{\a,\b}, J\resped)}& 
\aDrY{\Uhb} \ar[d]^{(\Res_{\Uha,\Uhb}, \id)}
\ar@{<=}[dl]_{v\resped}
\\
\hDrY{\a}{\Phi} \ar[r]_{\EKeq\a} & 
\aDrY{\Uha}
}
\end{equation}
where $(\Res_{\a,\b}, J\resped)$ is the tensor functor given by 
Theorem \ref{th:first main}, and the functor $\Res_{\Uha,\Uhb}$
is induced by the split embedding $\Uha\hookrightarrow\Uhb$.
\item For any chain of split embeddings $\a\hookrightarrow\b\hookrightarrow\c$,
the following diagram is commutative
\begin{equation}\label{eq:cylinder}
\xy
(0,20)*+{\hDrY{\c}{\Phi}}="C";
(15,0)*+{\hDrY{\b}{\Phi}}="B";
(0,-20)*+{\hDrY{\a}{\Phi}}="A";
(25,30)*+{\aDrY{\EK\c}}="HC";
(40,10)*+{\aDrY{\EK\b}}="HB";
(25,-10)*+{\aDrY{\EK\a}}="HA";
{\ar|(.4){\Res_{\a,\b}}@/^1pc/ "B";"A"};
{\ar|{\Res_{\b,\c}}@/^1pc/ "C";"B"};
{\ar@/_2pc/ "C";"A"_{\Res_{\a,\c}}};
{\ar|{\Res_{\Uha,\Uhb}}@/^1pc/ "HB";"HA"};
{\ar|{\Res_{\Uhb,\Uhc}}@/^1pc/ "HC";"HB"};
{\ar|(.3){\Res_{\Uha,\Uhc}}@{..>}@/_2pc/ "HC";"HA"};
{\ar "C";"HC"^{\noEKeq{\c}}};
{\ar "B";"HB"^{\noEKeq{\b}}};
{\ar "A";"HA"_{\noEKeq{\a}}};
{\ar@{=>}_{u_{\a,\b,\c}}"B";(-5,0)};
\endxy\end{equation}
where $u_{\a,\b,\c}$ is the isomorphism given by Theorem
\ref{th:first main}, the back 2--face is the identity, and the
lateral 2--faces are the isomorphisms $v_{\a,\c},v_{\b,\c},
v\resped$.\footnote{to alleviate the notation, tensor structures
are suppressed from the diagram \eqref{eq:cylinder}.}
\end{itemize}
\end{theorem}

\subsection{} 

We now outline the proofs of Theorems \ref{th:first main} and \ref{th:second main}.

The construction of the tensor structure $J_{\a,\b}$ on the functor
$\Res_{\a,\b}$ given by Theorem \ref{th:first main} is very much
inspired by that of the \nEK fiber functor $\EKff\b$ \cite{ek-1}, and
reproduces the latter if $\a=0$.
The principle adopted in \cite{ek-1} is the following. 
In a $\sfk$-linear monoidal category $\C$, a coalgebra structure on an
object $C\in\obj(\C)$ induces a tensor structure on the Yoneda functor
\footnote{strictly speaking, $C$ only induces a {\it lax} tensor structure
on $h_C$ since the morphism $\Delta_C^*:h_C(U)\otimes h_C(V)\to
h_C(U\otimes V)$ induced by the coproduct of $C$ may not be invertible.
We shall ignore this point, since all lax tensor structures we shall encounter
are easily seen to be invertible.}
\[h_C=\Hom_{\C}(C, -):\C\to\vectk.\]
If, moreover, $\C$ is braided and $C_1,C_2$ are coalgebra objects in
$\C$, then so is $C_1\ten C_2$, and there is therefore a canonical tensor
structure on $h_{C_1\ten C_2}$.

If $\b$ is a \fd Lie bialgebra, the forgetful functor $\hDrY{\b}{}=\Rep(\gu)
\to\vect_\sfk$ is represented by the enveloping algebra $\Ug_\b$ of the
Drinfeld double of $\b$. The object $\Ug_\b$ with its standard coproduct
is not a coalgebra in the deformed category $\hDrY{\b}{\Phi}$ due to the
non--triviality of the associativity constraints. However, the polarization
$\Ue{\gu}\simeq\Mm\ten\Mp$, where $M_\pm$ are the Verma modules
$\ind_{\gump}^{\gu}\sfk$, with $\gum=\b$ and $\gup=\b^*$, realizes $U
\gu$ as the tensor product of two coalgebra objects. This yields a
tensor structure on the functor $\EKff\b=h_{M_-\otimes M_+}:\hDrY{\b}
{\Phi}\to\vectk$, and therefore on the forgetful functor $h_{U\gu}$.

\subsection{}

Our starting point is to consider the restriction $\Res_{\a,\b}:\hDrY{\b}{}
\to\hDrY{\a}{}$ corresponding to a split pair of \fd Lie bialgebras $\a
\hookrightarrow\b$ as a relative forgetful functor, which is represented
by the $(\gu,\gd)$--bimodule $\Ue{\gu}$. We then factorise $\Ue{\gu}$
as the tensor product of two coalgebra objects $\Lm,\Np$ in the braided tensor
category of $(\gu,\gd)$--bimodules, with associativity constraints given
by $\PhiU\cdot\PhiD^{-1}$. Just as the Verma modules $\Mm,\Mp$ are related
to the decomposition $\gu=\gum\oplus\gup$, $\Lm$ and $\Np$ correspond
to the asymmetric decomposition
\[\gu=\Lmm\oplus\Lpp\]
where $\Lmm=\Ker(p)\subset\b_-$ and $\Lpp=i(\a_-)\oplus\b_+$.
The factorisation $\Ue{\gu}\cong\Lm\otimes\Np$ induces a tensor
structure on the functor $h_{\Lm\ten\Np}:\hDrY{\b}{\Phi}\to\hDrY{\a}
{\Phi}$, and therefore one on $\Res_{\a,\b}\cong h_{\Lm\ten\Np}$.

As in \cite[Part II]{ek-1}, this tensor structure can also be defined
when $\a$ or $\b$ are infinite--dimensional. This amounts to replacing
the Verma module $\Np$, which is not equicontinuous if $\b$ is infinite--dimensional,
with its appropriately topologised continuous dual $\Npv$, and the
Yoneda functor $h_{\Lm\ten\Np}$ with $\Hom_{\gb}(\Lm,\Npv\otimes V)$.

Having constructed the tensor structure $J_{\a,\b}$, the existence of the
natural transformation $u_{\a,\b,\c}$ satisfying the associativity constraint
\eqref{eq:u assoc} is readily obtained from that of the natural transformations
$v\resped,v_{\b,\c},v_{\a,\c}$ of Theorem \ref{th:second main} by requiring
the commutativity of the diagram \eqref{eq:cylinder}, and using the fact that
the functor $\noEKeq{\a}$ is an equivalence. It seems
an interesting problem to give an intrinsic construction of $u_{\a,\b,\c}$ which
does not rely on \nEK quantisation. 

\subsection{}

We now sketch the construction of a natural transformation $v\resped$
which makes the diagram \eqref{eq:nat transf} commute, assuming again
that $\a,\b$ are finite--dimensional.

As pointed out to us by Pavel Etingof, just as $(\Res_{\a,\b},J_{\a,\b})$
may be replaced by the isomorphic Yoneda functor $(h_{\Lm\otimes\Np},
\Delta_{\Lm\otimes\Np}^*)$, the tensor restriction functor $(\Res_{\Uha,
\Uhb},\id)$ can be replaced by $(h_{{\Lmh}\otimes\Np^\hbar},\id)$,
where $\Lmh,\Np^\hbar$ are quantum analogues of the modules $\Lm,
\Np$.
This reduces the problem to proving the commutativity of
\begin{equation}\label{eq:reformulate}
\xymatrixrowsep{0.4cm}
\xymatrixcolsep{0.4cm}
\xymatrix{
\hDrY{\b}{\Phi}\ar[rr]^{\EKeq{\b}}\ar[dd]_{(h_{\Lm\otimes\Np},\Delta_{\Lm\otimes\Np}^*)}
&&\aDrY{\Uhb}\ar[dd]^{(h_{{\Lmh}\otimes 
\Np^\hbar},\id)}\\
&&& \\
\hDrY{\a}{\Phi}\ar[rr]_{\EKeq{\a}}			&&\aDrY{\Uha}
}
\end{equation}

\subsection{}

The commutativity of \eqref{eq:reformulate} amounts to proving the
isomorphisms
\begin{equation}\label{eq:EK Verma}
{\EKeq{\a}}\circ {\EKeq{\b}}(\Lm)\simeq{\Lmh}\aand
{\EKeq{\a}}\circ {\EKeq{\b}}(\Np)\simeq\Np^{\hbar}
\end{equation}
as coalgebras in the category of $(\EK{\gu},\EK\gd)$--bimodules.
When $\a=0$, the Verma modules $\Lm,\Np$ coincide
with $M_-,M_+$, and it is easy to construct an isomorphism between
${\EKeq{\b}}(M_{\pm})$ and {the} quantum counterparts $M_\pm^\hbar$
of {$M_\pm$}. In general, however, the proof of \eqref{eq:EK Verma}
is more involved, and relies on the functoriality of the \DYt modules $\Lm,
\Np$, and of their quantisation via ${\EKeq{\a}}\circ {\EKeq{\b}}$, \wrt
morphisms of split pairs of Lie bialgebras. The latter is obtained from the description of the \nEK
quantisation functor in terms of $\PROP$s \cite{ek-2}, and the realization
of $\Lm,\Np$ as universal objects in a (new) colored $\PROP$ which
describes split inclusions of Lie bialgebras.  

\subsection{}

Our interest in Theorems \ref{th:first main} and \ref{th:second main} comes
from quasi--Coxeter \qtqbas \cite{vtl-4} and their categorical counterparts,
braided \qc categories \cite{ATL1-2}. If $(W,S)$ is a Coxeter group with
Coxeter graph $D$, a braided \qcc of type $W$ consists of 1) a family
of braided tensor categories $\Q_B$ labelled by the subgraphs $B$ of
$D$, 2) a tensor restriction functor $F_{B'B}:\Q_B\to\Q_{B'}$ for any inclusion
$B'\subseteq B$, and 3) an automorphism $S_i^\Q$ of the restriction
functor $F_{\emptyset i}:\Q_{\{i\}}\to\Q_\emptyset$, called {\it local
monodromy}, for any vertex $i$ of $D$. These data satisfy various
compatibilities which guarantee in particular that the generalised
braid group $B_W$ corresponding to $W$ acts on the restriction
functor $F_{\emptyset D}:\Q_D\to\Q_\emptyset$.

Such a structure arises in particular from a quantum \KMA $\Uhg$
with Weyl group $W$ \cite{ATL1-2}. The diagrammatic categories
$\Q_B$ are the integrable, highest weight categories of the standard
Levi subalgebras of $\Uhg$ with braiding given by their $R$--matrices,
the $F_{B'B}$ are the standard restriction functors, and the corresponding
braid group representations the quantum Weyl group representations.

We prove in \cite{ATL1-2} that this structure can be transferred to
the underformed enveloping algebra $\Ug\fml$, and put in a given
normal form. Specifically, the braided \qcc structure arising from
$\Uhg$ is equivalent to one where the diagrammatic categories
are the integrable, highest weight categories 
of the standard Levi subalgebras of $\Ug\fml$, with commutativity and
associativity constraints deformed by the associator $\Phi$, and
the $F_{B'B}$ are the standard restriction functors endowed
with appropriate tensor structures. The horizontal transport of structure 
is given by the collection of \nEK equivalences corresponding to the
Levi subalgebras of $\g$, while the vertical matching of the quantum
and classical restriction functors relies on Theorems \ref{th:first main}
and \ref{th:second main}.

We show in \cite{ATL2} that braided \qcc arising from $\Ug
\fml$ which have the same normal form as those transferred from
$\Uhg$ are rigid. We then use this in \cite{ATL3} to prove that the
monodromy of the rational Casimir connection of a symmetrisable
Kac--Moody algebra is described by the quantum Weyl group
operators of $\Uhg$, thus extending a result of the second
author valid when $\g$ is finite--dimensional \cite{vtl-4,vtl-6}.


\subsection{Outline of the paper}

In Section \ref{s:ek}, we review the construction of the \nEK fiber
functor, quantisation, and Tannakian equivalence. In Section \ref
{s:Gamma}, we generalise the first one by introducing the generalised
Verma modules $L_-,N_+$, and obtain a relative fiber functor $\FF
{\a}{\b}:\hDrY{\b}{\Phi}\to\hDrY{\a}{\Phi}$. In Section \ref{s:Gammah},
we define the quantum Verma modules $L_-^\hbar$ and $N_+^\hbar$,
use them to define a quantum fiber functor $\hFF{\a}{\b}:\aDrY{\Uhb}
\to\aDrY{\Uha}$, and construct an isomorphism $v_{\a,\b}:{\EKeq{\a}}
\circ\FF{\a}{\b}\simeq\hFF{\a}{\b}\circ{\EKeq{\b}}$ assuming the
quantisation isomorphisms \eqref{eq:EK Verma}. In Section \ref
{s:gammahopf}, we generalise the Etingof--Kazhdan quantisation to
the relative case, using the functor $\FF{\a}{\b}$ to construct a Hopf
algebra object in $\hDrY{\a}{\Phi}$ and obtain an alternative quantisation
of $\b$ via the Radford biproduct. In Section \ref{s:props}, we review
the description of \nEK quantisation by PROPs, and use it to give an
alternative proof that the Tannakian functor
$\EKeq{\b}$ is an equivalence. In Section \ref{se:rel prop}, we show
that the quantum Verma modules $L_-^\hbar$ and $N_+^\hbar$ are
isomorphic to the Etingof--Kazhdan quantisation of their classical
counterparts $\Lm,\Np$, thereby proving \eqref{eq:EK Verma}, by
using a suitably defined colored $\PROP$ which describes split
pairs of Lie bialgebras. We also show that the results of Sections
\ref{s:Gamma}, \ref{s:Gammah}, and \ref{s:gammahopf}
can be lifted to this $\PROP$, thus establishing their functoriality
with respect to morphisms of split inclusions of Lie bialgebras.
In Section \ref{se:Severa}, we review \v{S}evera's alternative
construction of a quantisation functor for Lie bialgebras \cite{sev},
and use it to obtain a stronger functoriality of the tensor functor
$\FF{\a}{\b}$. Finally, in Appendix \ref{s:app}, we review Majid's
description of the quantum double as a double crossed product
of Hopf algebras, and describe the braided tensor equivalence
between modules over a quantum double and \DYt modules.

\subsection{Acknowledgments}

The main results of this paper first appeared in more condensed form
in the preprint \cite{ATL0}. The latter is superseded by the present paper,
and its companion \cite{ATL1-2}. We are very grateful to Pavel Etingof
for his continuing interest throughout this project, and in particular for
correspondence and several enlightening discussions on foundational
and other aspects of quantisation.

\section{Etingof-Kazhdan quantisation}\label{s:ek}

With the exception of Sections \ref{ss:DY-qD}--\ref{ss:qRep},
which contain a detailed discussion of admissible \DYt modules
over a quantised universal enveloping algebra, and \ref{ss:quant-bimod},
this section follows \cite[Part II]{ek-1} and \cite[\S 4]{ek-6} closely. 


\subsection{Topological vector spaces}

Let $\sfk$ be a field of characteristic zero endowed with the discrete topology,
and  $V$ a topological vector space over $\sfk$. The topology on $V$ is {\em
linear} if the open subspaces in $V$ form a basis of neighborhoods of zero.

Let $V$ be endowed with a linear topology, and $p$ the natural map 
\[p:V\longrightarrow\lim_{\longleftarrow} V/U\]
where the inverse limit is taken over the open subspaces $U\subseteq V$.
$V$ is called \emph{separated} if $i$ is injective, and \emph{complete}
if $p$ is surjective. Note that 
\begin{itemize}
\item Since an open subspace of a topological vector space is also closed,
the quotient topology on each $V/U$ is the discrete one. The corresponding
product topology on $\lim V/U$ is linear, with a basis of neighborhoods of
zero given by the finite intersections $\bigcap p^{-1}_U(0)$, where $p_U:
\lim V/U'\to V/U$ are the projection maps. Moreover, $\lim V/U$ is separated
and complete.
\item The map $p$ is continuous. It need not be open in general, but it is if
$p$ is surjective. It follows that if $V$ is separated and complete, $p$ is a
homeomorphism.
\end{itemize}
Throughout this paper, we shall call  \emph{topological vector space} a linear,
complete, separated topological vector space.
Note that a \fd topological vector space is necessarily endowed with the
discrete topology.

\subsection{}

If $V,W$ are topological vector spaces, we let $\Hom_{\sfk}(V,W)$ be the
topological vector space of continuous linear maps from $V$ to $W$, equipped
with the weak topology. Namely, a basis of neighborhoods of zero in $\Hom_
{\sfk}(V,W)$ is given by the subspaces
\[Y(v_1,\dots,v_n;W_1,\dots,W_n)=\{T\in \Hom_{\sfk}(V,W)\;|\: T(v_i)\in W_i, i=1,\dots, n\}\]
where $n\in\IN, v_i\in V$ and $W_i$ are open subspaces in $W$ for all $i=1,\dots, n$.

In particular, if $W=\sfk$ with the discrete topology, the space $V^*=\Hom_{\sfk}(V,\sfk)$ 
has a basis of neighborhoods of zero given by orthogonal complements of 
finite--dimensional subspaces in $V$. When $V$ is finite--dimensional, $V^*$ 
is the full linear dual of $V$, and the weak topology the discrete topology. 

\subsection{}

Given two topological  vector spaces $V$ and $W$, define their topological
tensor product as
\[V\ctp W= \lim V/V' \ten W/W'=\lim V\otimes W/(V'\otimes W+V\otimes W')\]
where the limit is take over open subspaces of $V$ and $W$, and given the
product topology. Then, $V\ctp W$ is a topological vector space, and the
tensor product $\ctp$ endows the category $\vectk$ of topological vector
spaces over $\sfk$ with the structure of a symmetric monoidal category
with internal Hom's. Moreover, the duality functor satisfies $(V\ctp W)^*
\cong V^*\ctp W^*$, and is therefore a contravariant tensor endofunctor 
of $\vectk$.

\subsection{Formal power series}

Let $\hbar$ be a formal variable, and endow $\sfK=\sfk{\fml}$ with the
$\hbar$--adic topology given by the subspaces $\hbar^n\sfK$, $n\geq 0$.
Let $V$ be a
topological vector space. The space $V{\fml}=V\ctp\sfK$ of formal power
series in $\hbar$ with coefficients in $V$ is also a topological vector space
with the structure of a topological $\sfK$-module. A topological $\sfK$--module
is topologically free if it is isomorphic to  $V{\fml}$ for some topological
vector space $V$ as $\sfK$--module. 

The additive category $\vect_\sfK$ of topologically free $\sfK$--modules,
where morphisms are continuous $\sfK$--linear maps, has a natural 
symmetric monoidal structure with internal Hom's. 
The tensor product is defined as the quotient of the tensor product 
$V\ctp W$ by the closure of the image of the operator $\hbar\ten1-1\ten\hbar$,
and will be still denoted by $\ctp$.

There is an extension 
of scalar functor from the category of topological vector spaces to $\vect_\sfK$, mapping 
$V$ to $V{\fml}$. This functor respects the tensor product, 
\ie $(V\ctp W){\fml}$ is naturally isomorphic to $V{\fml}\ctp W{\fml}$. 

Henceforth, unless otherwise specified, we will denote by $\ten$ the complete tensor
products of topological $\sfk$--vector spaces and topologically free $\sfK$--modules.


\subsection{Lie bialgebras}\label{ss:lba}

A Lie bialgebra is a triple $(\b,[,]_{\b}, \delta_{\b})$ where
\begin{itemize}
\item $\b$ is a discrete vector space;
\item  $(\b,[,]_{\b})$ is a Lie algebra, \ie $[,]_{\b}:\b\ten\b\to\b$
is anti-symmetric and satisfies the Jacobi identity
\[
[,]_{\b}\circ\id_{\b}\ten[,]_{\b}\circ(\id_{\b^{\ten 3}}+(1\,2\,3)+(1\,3\,2))=0;
\]
\item $(\b,\delta_{\b})$ is a Lie coalgebra, \ie $\delta_{\b}:\b\to\b\ten\b$
is anti-symmetric and satisfies the co--Jacobi identity
\[
(\id_{\b^{\ten 3}}+(1\,2\,3)+(1\,3\,2))\circ\id_{\b}\ten\delta_{\b}\circ\delta_{\b}=0;
\]
\item the cobracket $\delta_{\b}$ satisfies the \emph{cocycle condition}
\begin{equation}\label{eq:cocycle}
\delta_{\b}\circ[,]_{\b}=\ad_{\b}\circ\id_{\b}\ten\delta_{\b}\circ(\id_{\b^{\ten 2}}-(1\,2)),
\end{equation}
as maps $\b\ten\b\to\b\ten\b$, 
where $\ad_{\b}:\b\ten\b\ten\b\to\b\ten\b$ denotes the left adjoint action of $\b$
on $\b\ten\b$.
\end{itemize}

\subsection{Manin triples}\label{ss:lbamt}

A Manin triple is the data of a Lie algebra $\g$ with 
\begin{itemize}
\item a nondegenerate invariant inner product $\iip{-}{-}$
\item isotropic Lie subalgebras $\gupm\subset\g$
\end{itemize}
such that 
\begin{itemize}
\item $\g=\gum\oplus\gup$ as vector spaces
\item the commutator of $\g$ is continuous with respect to the topology 
obtained by putting the discrete and the weak topologies on $\gum$
and $\gup$  respectively.
\item the inner product defines an isomorphism $\gup\to\gum^*$
\end{itemize}

Under these assumptions, the commutator on $\gup\simeq\gum^*$
induces a cobracket $\delta:\gum\to\gum\otimes\gum$ which satisfies
the cocycle condition \eqref{eq:cocycle}.
Therefore, $\gum$ is canonically endowed with a Lie bialgebra structure.
In general, however, $\gup$ is only a topological Lie bialgebra.

\subsection{Drinfeld double}\label{ss:drinf-double}

Every Lie bialgebra $(\b,[,]_{\b},\delta_{\b})$ gives rise to a Manin triple.
The Drinfeld double of $\b$ is the Lie algebra $\gb$ defined as follows.

As a vector space, $\gb=\b\oplus\b^*$. The pairing $\iip{-}{-}:\b\ten\b^*
\to\sfk$ extends uniquely to a symmetric non--degenerate bilinear form
on $\gb$, \wrt which $\b,\b^*$ are isotropic. The Lie bracket on $\gb$ is
defined as the unique bracket which coincides with $[,]_{\b}$ on $\b$,
with $\delta_{\b}^t$ on $\b^*$, and is compatible with $\iip{-}{-}$, \ie satisfies
\[
\iip{[x,y]}{z}=\iip{x}{[y,z]}
\]                                    
for all $x,y,z\in\gb$. The mixed bracket of $x\in\b$ and $\phi\in\b^*$ is then 
given by
\[
[x,\phi]=\sfad^*(x)(\phi)-\sfad^*(\phi)(x)
\]
where $\sfad^*$ is the coadjoint actions of $\b$ on $\b^*$ and of $\b^*$ on
$(\b^*)^*$\footnote{the latter preserves $\b\subset(\b^*)^*$ since it is easily
seen to be given by $\sfad^*(\phi)(x)=-\iip{\phi\ten\id}{\delta_\b(x)}$.}.

\subsection{Drinfeld--Yetter modules}\label{ss:DY LBA}

A Drinfeld--Yetter module over a Lie bialgebra $\b$ is a triple $(V,
\pi,\pi^*)$, where $\pi:\b\otimes V\to V$ gives $V$ the structure of a left
$\b$--module, that is
\begin{equation}\label{eq:action-lba}
\pi\circ[,]\ten\id=\pi\circ(\id\ten\pi)-\pi\circ(\id\ten\pi)\circ(21)
\end{equation}
as maps $\b\otimes\b\otimes V\to V$, $\pi^*:V\to\b\otimes V$ gives $V$
the structure of a right $\b$--comodule, that is
\begin{equation}\label{eq:coaction-lba}
\delta\ten\id\circ\pi^*=(21)\circ(\id\ten\pi^*)\circ\pi^*-(\id\ten\pi^*)\circ\pi^*
\end{equation}
as maps $V\to\b\otimes\b\otimes V$, and the maps $\pi,\pi^*$ satisfy the
following compatibility condition in $\End(\b\ten V)$
\begin{equation}\label{eq:act-coact-lba}
\pi^*\circ\pi-\id\ten\pi\circ(1\,2)\circ\id\ten\pi^*=[,]\ten\id\circ\id\ten\pi^*-
\id\ten\pi\circ\delta\ten\id.
\end{equation}

A \DYt module $V$ gives rise to an action of the Drinfeld double $\gu$
on $V$, with $\phi\in\b^*$ acting by $\phi\otimes\id\circ\pi^*$. The
corresponding map $\gu\to\End_\sfk(V)$ is easily seen to be continuous
if $\End_\sfk(V)$ is given the weak topology, and $\gu$ the product 
of the discrete and weak topologies on $\b,\b^*$, as in \ref {ss:lbamt}.
Conversely, a continuous Lie algebra homomorphism $\gu\to\End_\sfk(V)$
gives in particular rise to a locally finite action of $\b^*$, and therefore
to a \DYt module on $V$, with $\pi^*(v)=\sum_i b_i\otimes b^i v$, where
$\{b_i\},\{b^i\}$ are dual bases of $\b,\b^*$.


\subsection{Equicontinuous modules}\label{ss:eq}

Let $\g$ be a topological Lie algebra, and $V$ a topological vector space.
$V$ is an  \emph{equicontinuous} $\g$-module if it is endowed with a Lie
algebra homomorphism $\pi_V:\g\to\End_{\sfk}(V)$ such that
\begin{itemize}
\item $\pi_V$ is continuous
\item $\{\pi_V(X)\}_{X\in\g}$ is an equicontinuous family of linear
operators, \ie for any open subspace $U\subseteq V$, there exists
$U'$ such that $\pi_V(X)U'\subset U$ for all $X\in\g$.
\end{itemize}

The category $\Eq{\g}$ of equicontinuous $\g$-modules is a symmetric
monoidal category, \wrt the completed tensor product of topological vector
spaces and braiding defined by the permutation of components.
 
\subsection{Topological Drinfeld--Yetter modules}

The notion of \DYt module can be formulated in any symmetric
tensor category, in particular that of topological vector spaces, with
completed tensor product and continuous linear maps. The corresponding
category $\hDrY{\b}{}$ is then equivalent to that of equicontinuous
modules $\Eq{\gb}$ over $\gu$ defined in \S\ref{ss:eq}.

\subsection{$r$--matrix}

If $U,V\in\hDrY{\b}{}$, define $r_{U,V}\in\End_\sfk(U\ten V)$ as
the composition
\[r_{UV}=\pi_{U}\otimes\id\circ (1\,2)\circ \id\otimes\pi^*_V.\]
Then, $r$ satisfies the classical Yang--Baxter equations on $U\ten V\ten W$
\[[r_{UV},r_{UW}]+[r_{UV},r_{VW}]+[r_{UW},r_{VW}]=0\]
and is such that $\Omega_{UV}=r_{UV}+r^{21}_{VU}$ is
a morphism of Drinfeld--Yetter modules.

Under the identification $\DrY{\b}=\Eq{\gb}$, the action of $r_{UV}$ corresponds
to that of the canonical element $b_i\otimes b^i\in\b\ten\b^*\cong\End_k(\b)\ni\id_\b$,
and that of $\Omega_{UV}$ to that of the canonical element $b_i\otimes b^i+b^i\otimes
b_i$ in $\gu\ten\gu$ representing the bilinear form $\<-,-\>$.


\subsection{Drinfeld category}\label{ss:Drinfeld cat}

Following Drinfeld \cite{drin-3}, one can define a $\sfK$--linear deformation 
$\hDrY{\b}{\Phi}=\Eq{\gb}^\Phi$ of $\hDrY{\b}{}=\Eq{\gb}$ as a braided 
monoidal category using an associator $\Phi$ as follows. The objects of 
$\hDrY{\b}{\Phi}$ are Drinfeld--Yetter $\b$--modules in the category of 
topologically free $\sfk{\fml}$--modules with commutativity and associativity 
constraints given respectively by
\[\beta_{UV}=(1\,2)\circ \exp{\left(\frac{\hbar}{2}\Omega_{UV}\right)}
\aand
\Phi_{UVW}=\Phi(\hbar\Omega_{UV},\hbar\Omega_{VW}).
\]



\subsection{The Verma modules $M_-$ and $M_+$}\label{ss:vermaEK}
Set $\gum=\b$, $\gup=\b^*$, and consider the $\gu$--modules $\Mm,
\Mp$ given by
\[\Mm=\Ind_{\gup}^{\gu}\sfk\aand \Mp=\Ind_{\gum}^{\gu}\sfk.\]
The modules $\Mm$ and $\Mpv$, the dual of $\Mp$ with appropriate
topology, are equicontinuous $\gu$-modules.

The module $\Mm$ is an equicontinuous $\gu$-module with respect 
to the discrete topology. The topology on $\Mp$ comes, instead, from 
the identification of vector spaces 
\[\Mp\simeq \Ue{\gup}=\bigcup_{n\geq0} (\Ue{\gup})_n\] 
where $(\Ue{\gup})_n$ is the set of elements of degree at most $n$. 
The topology on $(U\gup)_n$ is defined through the linear isomorphism
\[\xi_n:\bigoplus_{j=0}^n S^j\gup\to(\Ue{\gup})_n\]
where $S^j\gup$ is considered as a topological subspace of $(\gum^{\ten j})^*$, 
embedded with the weak topology.
Finally, $U\gup$ is equipped with the topology of the colimit. 
Namely, a set $U\subseteq U\gup$ is open if and only if $U\cap (U\gup)_n$ 
is open for all $n$. With respect to the topology just described, 
the action of $\gu$ on $\Mp$ is continuous.

Consider now the vector space of continuous linear functionals on $\Mp$
\[\Mpd=\Hom_{\sfk}(\Mp,\sfk)\simeq\lim\Hom_{\sfk}((U\gup)_n, \sfk).\]
It is natural to put the discrete topology on $(U\gup)_n^*$, since, 
as a vector space,
\[(U\gup)_n^*\simeq\bigoplus_{j=0}^n S^j\gup^*\simeq\bigoplus_{j=0}^n S^j
\gum\simeq (U\gum)_n.\]
We then consider on $\Mpd$ the topology of the limit and denote the resulting 
topological space by $\Mpv$. This defines, in particular, a filtration by subspaces 
$(\Mpv)_n$ satisfying
\[0\to(\Mpv)_n\to\Mpv\to\UE{\gup}_n^*\to 0\]
and such that $\Mpv=\lim \Mpv/(\Mpv)_n$. The topology of the limit on $\Mpd$ is, 
in general, stronger than the weak topology of the dual. 
Since the action of $\gu$ on $\Mp$ is continuous, $\Mpv$ has a natural structure 
of $\gu$--module. In particular, this is an equicontinuous $\gu$--action.

\subsection{Properties of Verma modules}

The modules $M_{\pm}$ are identified, as $\b_\pm$--modules,
with the enveloping  universal algebras $U\gupm$. The comultiplication
on the latter induce the $U\gu$--intertwiners $i_{\pm}:M_{\pm}\to M_{\pm}\ten M_{\pm}$, 
mapping the vectors $1_{\pm}$ to the $\gump$-invariant vectors $1_{\pm}\ten 1_{\pm}$.

For any $f,g\in\Mpv$, consider the linear functional $\Mp\to\sfk$ 
defined by $v\mapsto (f\ten g)(\ip(v))$. 
This defines a morphism of modules $\ips:\Mpv\ten\Mpv\to\Mpv$.
The pairs $(\Mm,\im)$ and $(\Mpv, \ips)$ form, respectively, 
a coalgebra and an algebra object in $\hEq{\gu}{\Phi}$.\\

For any $V\in\hEq{\gu}{\Phi}$, the vector space $\Hom_{\gu}(\Mm,\Mpv\ten V)$ 
is naturally isomorphic to $V$, as topological vector space, through the isomorphism 
$f\mapsto (1_+\ten 1)\noEKff{\b}(1_-)$.


\subsection{The fiber functor}\label{ss:EK fiber}
Let $\noEKff{\b}:\hEq{\gu}{\Phi}\to\vect_\sfK$ be the functor given by
\[ \noEKff{\b}(V)=\Hom_{\hEq{\gu}{\Phi}}(\Mm,\Mpv\ten V).\]
Define a natural transformation $J_{V,W}:\noEKff{\b}(V)\ten \noEKff{\b}(W)\to
\noEKff{\b}(V\ten W)$ by
\[J_{V,W}(v\ten w)= (\ipd\ten\id^{\otimes 2})\circ A^{-1}\circ\beta_{23}^{-1}\circ A\circ 
(v\ten w)\circ \im\]
where $A$ is the isomorphism
\[(V_1\ten V_2)\ten (V_3\ten V_4)\to V_1\ten( (V_2\ten V_3)\ten V_4)\]
defined by the action of $(1\ten{\Phi}_{2,3,4}){\Phi}_{1,2,34}$.

\begin{theorem}\cite{ek-1}
The natural transformation $J$ is invertible, and defines a tensor structure
on the functor $\noEKff{\b}$, that is satisfies
\[ \noEKff{\b}(\Phi_{U,V,W})\,J_{U\otimes V,W}\,J_{U,V}\otimes\id=J_{U,V\otimes W}\,
\id\otimes J_{V,W}\]
as morphisms $(\noEKff{\b}(U)\ten \noEKff{\b}(V))\ten \noEKff{\b}(W)\to \noEKff{\b}
(U\ten(V\ten W))$, for any $U,V,W\in\hEq{\gu}{\Phi}$.
\end{theorem}

The algebra of endomorphisms of $\noEKff{\b}$ is therefore naturally endowed
with a topological bialgebra structure.\footnote{By \emph{topological} bialgebra
we do not mean topological over $\sfk{\fml}$. We are instead referring
to the fact that the algebra $\sfEnd{\noEKff{\b}}$ has a natural comultiplication 
$\Delta:\sfEnd{\noEKff{\b}}\to\sfEnd{\noEKff{\b}\otimes \noEKff{\b}}$, and that 
$\sfEnd{\noEKff{\b}\otimes \noEKff{\b}}$ can be interpreted as a completion of 
$\sfEnd{\noEKff{\b}}^{\ten 2}$.}

\subsection{Etingof--Kazhdan quantisation}\label{ss:eksum}

Let $V\in\hEq{\gu}{\Phi}$, and consider the linear map $m_V:\noEKff{\b}(M_-)
\ten \noEKff{\b}(V)\to \noEKff{\b}(V)$ given by
\[m_V(x)\,v=(\ipd\ten 1)\Phi^{-1}(1\ten v)x.\]
The map $m_V$ satisfies the associativity relation
\[ m_V\circ (m_{M_-}\otimes\id_V)=m_V\circ (\id_{M_-}\otimes m_V)\]
as morphisms $\noEKff{\b}(M_-)\ten \noEKff{\b}(M_-)\ten \noEKff{\b}(V)\to \noEKff{\b}
(V)$, and the unit condition $m_V(u_-)=\id_V$, where $u_-\in \noEKff{\b}(M_-)$ is the 
element mapping $1_-\in M_-$ to $1_+^*\otimes 1_-\in M_+^*\ten M_-$.

As a consequence, $\noEKff{\b}(M_-)$ is an associative algebra with unit $u$
and multiplication $m_{M_-}$, which acts on the functor $\noEKff{\b}$. The
corresponding map $\noEKff{\b}(M_-)\to\End(\noEKff{\b})$ is an embedding since 
$\noEKff{\b}(M_-)$ is unital and acts on itself, and we shall identify $\noEKff{\b}(M_-)$
with its image in $\End(\noEKff{\b})$.

The (topological) coproduct on $\End(\noEKff{\b})$ maps $\noEKff{\b}(M_-)$ to $F
(M_-)\ten \noEKff{\b}(M_-)$, and coincides with the coproduct induced by
the coalgebra structure on $M_-$ given by $\Delta=J_{M_-,M_-}^{-1}
\circ \noEKff{\b}(i_-)$. It follows that
\[\Uhb=(\noEKff{\b}(M_-),m_{M_-},u_-,J_{M_-,M_-}^{-1}\circ \noEKff{\b}(i_-),m_\sfk)\]
is a bialgebra, with counit $m_{\sfk}:\noEKff{\b}(M_-)\to\sfK$, endowed with
an identification of $\sfK$--modules $\noEKff{\b}(M_-)\cong M_-{\fml}\cong
U\b{\fml}$.

\begin{theorem}\cite{ek-1,ek-2}
\begin{enumerate}
\item $\Uhb$ is a Hopf algebra, which is a quantisation of the Lie bialgebra $\b$.
\item The assignment $\b\mapsto\Uhb$ is functorial in the Lie bialgebra $\b$.
\end{enumerate}
\end{theorem}

\noindent
\remark The fact that $\Uhb$ admits an antipode follows because it is
a deformation of $U\b$ as a bialgebra, and the latter has an antipode.
A formula for the antipode is given explicitly in \cite[Prop. 4.2]{ee}, and
one has
\[
S=\sum_{n\geq0}(-1)^nm^{(n)}(\id - \iota\circ\epsilon)^{\ten n}\Delta^{(n)}.
\]

\subsection{Drinfeld--Yetter modules over a Hopf algebra}\label{ss:DY Hopf}

If $(B,m,\imath,\Delta,\epsilon,S)$ is a Hopf algebra, a \DYt module over
$B$ is a triple $(V,\pi,\pi^*)$, where $\pi:B\otimes V\to V$ gives $V$ the
structure of a left $B$--module, that is
\begin{equation}\label{eq:action}
\pi\circ m\otimes\id_V=\pi\circ\id_B\otimes\pi
\aand
\pi\circ\imath\otimes\id_V=\id_V
\end{equation}
as maps $B\otimes B\otimes V\to V$ and $V\to V$ respectively, $\pi^*:V
\to\b\otimes V$ gives $V$ the structure of a right $B$--comodule, that is
\begin{equation}\label{eq:coaction}
\Delta^{21}\ten\id_V\circ\pi^*=\id_B\otimes\pi^*\circ\pi^*
\aand
\epsilon\otimes\id_V\circ\pi^*=\id_V
\end{equation}
as maps $V\to B\otimes B\otimes V$ and $V\to V$ respectively, and the maps
$\pi,\pi^*$ satisfy the following compatibility condition in $\End(B\ten V)$
\begin{equation}\label{eq:DY Hopf}
\pi^*\circ\pi=
m^{(3)}\otimes\pi \circ (1\,3)(2\,4) \circ
S^{-1}\otimes\id^{\otimes 4}\circ \Delta^{(3)}\otimes\pi^*
\end{equation}
where $m^{(3)}=m\circ m\otimes\id:B^{\otimes 3}\to B$ and $\Delta^{(3)}=
\Delta\otimes\id\circ\Delta:B\to B^{\otimes 3}$ are the iterated multiplication
and comultiplication respectively. The category $\DrY{B}$ of \DYt modules
over $B$ has a natural structure of braided monoidal category. For any
$V,W\in\DrY{B}$, the action and coaction on the tensor product 
$V\ten W$ are defined, respectively, by
\[
\pi_{V\ten W}=\pi_V\ten\pi_W\circ (2\,3)\circ\Delta\ten\id^{\ten 2}
\qquad\mbox{and}\qquad
\pi^*_{V\ten W}=m^{21}\ten\id^{\ten 2}\circ(2\,3)\circ\pi^*_V\ten\pi^*_W.
\]
The associativity constraints are trivial and the braiding is defined by $\beta_{VW}=(21)\circ R_{VW}$,
where the $R$--matrix $R_{VW}\in\End(V\ten W)$ is defined by
\[
R_{VW}=\pi_V\ten\id\circ(2\,3)\id\ten\pi_W^*.
\]
It follows from \eqref{eq:action}--\eqref{eq:coaction}, and $m\circ\id\ten
S\ten \Delta=\iota\circ\epsilon$, that $R_{VW}$ is an invertible 
endomorphism with inverse
\[
R_{VW}^{-1}=\pi_V\ten\id\circ S\ten \id^{\ten 2}\circ(2\,3)\id\ten\pi_W^*.
\]
The braiding $\beta_{VW}$ is therefore invertible, with inverse $R^{-1}
_{VW}\circ(1\,2)$.

\subsection{Drinfeld--Yetter modules and quantum double}\label{ss:DY-qD}

Let $B$ be a finite--dimensional Hopf algebra, and $B^{\circ}$ the Hopf
algebra $B^*$ with opposite coproduct. The quantum double of $B$ is
the unique quasitriangular Hopf algebra $(DB, R)$ such that 1) $DB=
B\otimes B^\circ$ as vector spaces 2) $B$ and $B^\circ$ are Hopf subalgebras
of $DB$ and 3) $R$ is the canonical element corresponding to $B\ten
B^{\circ}\subset DB\ten DB$. The multiplication in $DB$ is given 
in Sweedler notation by
\begin{equation}\label{eq:q-double-mult}
b\ten f \cdot b'\ten f'=\langle S^{-1}(b'_1), f_1\rangle\langle b'_3, f_3\rangle\,b\cdot b'_2
\otimes f_2\cdot f 
\end{equation}
where $f,f'\in B^{\circ}$, $b,b'\in B$, and $\iip{-}{-}$ is the evaluation pairing \cite
[\S 13]{drin-2}. It is easy to see that there is a canonical equivalence of braided
tensor categories $\Rep DB \simeq \DrY{B}$. For completeness, we include
a detailed proof of this result in the Appendix (cf. \S\ref{ss:app-4}).
More specifically, any Drinfeld--Yetter $B$--module $V$ is acted upon by $B^{\circ}$
via
\[
\xymatrix{B^{\circ}\ten V\ar[r]^(.45){\id\ten\pi^*_V} &
B^{\circ}\ten B\ten V\ar[r]^(.65){\mathsf{ev}\ten\id}&V}.
\]
The compatibility condition between action and coaction guarantees that the
actions of $B, B^{\circ}$ on $V$ give rise to an action of $DB$. Conversely,
given any $DB$--module $V$, the map
\[
\xymatrix@C=0.55in{V\ar[r]^{\iota\ten\id} & B\ten V \ar[r]^{\id\ten\pi_V(R)} & B\ten V}
\]
defines a compatible coaction on $B$ on $V$. It is straightforward to
check that this equivalence preserves the tensor product and braiding.

\subsection{Duality for quantum enveloping algebras}\label{ss:dualityQUE}

We review below (one half of) Drinfeld's duality principle \cite{drin-2,gav},
adapted to the setting of Lie bialgebras of arbitrary dimension.

Recall that a {\it quantised enveloping algebra} (QUE) is a Hopf algebra
$B$ in $\VecK$ such that 
\begin{itemize}
\item $B$ is endowed with the $\hbar$--adic topology, that is $\{\hbar
^n B\}_{n\geq 0}$ is a basis of neighborhoods of $0$. Equivalently,
$B$ is isomorphic, as topological $\sfK$--module, to $B_0\fml$, for
some discrete topological vector space $B_0$.
\item $B/\hbar B$ is a connected, cocommutative Hopf algebra over
$\sfk$. Equivalently, $B/\hbar B$ is isomorphic to $U\b$ for some
Lie bialgebra $(\b, [,]_{\b},\delta_{\b})$ and, under this identification,
\[
\delta_{\b}(b)=\frac{\Delta(\wt{b})-\Delta^{21}(\wt{b})}{\hbar} \mod\hbar
\]
where $\wt{b}\in B$ is any lift of $b\in\b$. 
\end{itemize}

A {\it quantised formal series Hopf algebra} (QFSH) is a Hopf algebra
$B$ in $\VecK$ such that 
\begin{itemize}
\item $B$ is endowed with the $I$--adic topology, where $I=\epsilon^
{-1}(\hbar\sfk\fml)$, that is $\{I^n\}_{n\geq 0}$ is a basis of neighborhoods
of $0$.
\item $B/\hbar B$ is a local, commutative Hopf algebra. Equivalently,
$B/\hbar B$ is isomorphic, as Poisson Hopf algebra, to $(\wh{S\b},m,
\Delta,P)$ for some Lie bialgebra $(\b,[\cdot,\cdot]_\b,\delta_\b)$. Here,
$\wh{S\b}$ is the completion of the symmetric algebra of $\b$ \wrt its grading, $m$ is
the commutative multiplication on $S\b$, $\Delta$ is the
coproduct obtained by formally starting from the \BCH multiplication
on $S\b^*\cong U\b^*$ (see \S \ref{ss:BCH}), reversing all arrows
and replacing $[\cdot,\cdot]_{\b^*}=\delta_b^t$ by $\delta_\b$, and
the Poisson bracket $P$ is given by the Lie bracket $[\cdot,\cdot]
_\b$.
\end{itemize}

We denote the categories of quantised  enveloping algebras and
quantised formal series Hopf algebras over $\sfk$ by $\QQUE(\sfk)$
and $\QQFSH(\sfk)$ respectively. For every $B\in\QQUE(\sfk)$, set
\[B'=\{b\in B\;|\; (\id-\iota\circ\epsilon)^{\ten n}\circ\Delta^{(n)}(b)\in
\hbar^nB^{\ten n}\;\text{for any $n\geq 0$}\}\]
where $\Delta^{(n)}$ denotes the iterated coproduct defined inductively
by $\Delta^{(0)}=\iota\circ\epsilon$, $\Delta^{(1)} = \id$, and
\[\Delta^{(n+1)}=\Delta\otimes \id^{\otimes (n-1)}\circ \Delta^{(n)}.\]
Then $B'$, endowed with the topology induced by the $\hbar$--adic
topology on $B$, is a Hopf subalgebra of $B$, and a quantised
formal series Hopf algebra. Moreover, if $B/\hbar B$ is isomorphic
to $U\b$ for some Lie bialgebra $\b$, then $B'/\hbar B'$ is isomorphic
to the Poisson Hopf algebra $(\wh{S\b},m,\Delta,P)$ described above.

If $B\in\QQUE(\sfk)$ is a quantisation of Lie bialgebra $(\b,[\cdot,\cdot]
_\b,\delta_b)$ where $\b$ is finite--dimensional, then $B^\vee=(B')^*$
is a QUE, with underlying Lie bialgebra $(\b^*,\delta_\b^t,[\cdot,\cdot]
_\b^t)$. The tensor product $B\otimes (B^\vee)^\circ$ can be
endowed with a unique quasitriangular Hopf algebra structure
such that $B$ and $(B^\vee)^\circ$ are Hopf subalgebras, and
is called the quantum double $DB$ of $B$. The corresponding $R$--matrix
$R$ lies in the $\hbar$--adic completion of the algebraic tensor
product $B'\otimes(B^\vee)^\circ\subset DB^{\otimes 2}$.

\subsection{Admissible Drinfeld--Yetter modules}\label{ss:admDY}

The following notion is due to P. Etingof \cite{e}.

\begin{definition}
A Drinfeld--Yetter module $\V$ over a quantised enveloping algebra $B$ 
is called \emph{admissible} if its coaction $\pi_\V^*:\V\to B\otimes\V$ factors
through $B'\otimes\V$.
\end{definition}

In analogy with the result described in \S\ref{ss:DY-qD}, the category
$\aDrY{B}$ of admissible \DYt modules over $B$ is equivalent, as a
braided tensor category,  to that of modules over the quantum double
$DB$ of $B$ (cf. \S\ref{ss:app-5} for a detailed proof). 

\subsection{Admissibility criterion}

By definition of $B'$, and the exactness of the tensor product
of topological vector spaces, a \DYt module $\V$ is admissible
if, and only if, for any $n\geq 1$, the image of
\[(\id-\iota\circ\epsilon)^{\ten n}\circ\Delta^{(n)}\ten\id_\V\circ\pi_{\V}^*\]
lies in $\hbar^n B^{\otimes n}\otimes\V$. The following gives a
more economical criterion.

\begin{proposition}\label{prop:adm-cond}
A \DYt module $(\V, \pi_{\V}, \pi_{\V}^*)$ is admissible if, and only if
\begin{equation}\label{eq:adm-cond}
\ima\left((\id-\iota\circ\epsilon)\ten\id\circ\pi_{\V}^*\right)\subset\hbar B\ten\V.
\end{equation}
\end{proposition}

\begin{pf}
If $\V$ is admissible, its coaction factors through $B'\ten\V$ and 
\eqref{eq:adm-cond} follows by the definition of $B'$. Conversely, 
assume $\pi_{\V}^*$ satisfies \eqref{eq:adm-cond}. Then, it follows from
\eqref{eq:coaction} that
$\Delta^{(n)}\ten\id_\V\circ\pi_{\V}^*=\sigma\circ(\pi^*_{\V})^{(n)}$
for some $\sigma\in\SS_n$. In particular, \eqref{eq:adm-cond} implies
that
\[
\ima\left((\id-\iota\circ\epsilon)^{\ten n}\ten\id_\V\circ\Delta^{(n)}\ten\id\circ
\pi_{\V}^*\right)\subset\hbar^n B^{\ten n}\ten\V.
\]
\end{pf}

\subsection{Semiclassical limit of admissible Drinfeld--Yetter modules}\label{ss:sc-lim-admDY}

Let $B$ be a QUE, with underlying Lie bialgebra $(\b,[,],\delta)$. The category
of admissible \DYt modules over $B$ is a quantisation of the category of
\DYt modules over $\b$. Specifically, for any $\V\in\aDrY{B}$ with action
and coaction
\[
\pi_{\V}: B\ten\V\to\V
\aand
\pi_{\V}^*:\V\to B'\ten\V
\]
we set $\SC{\V}=\V/\hbar\V$ and define the maps
\[
\SC{\pi_{\V}}:\b\ten\SC{\V}\to\SC{\V}
\aand
\SC{\pi_{\V}^*}:\SC{\V}\to\b\ten\SC{\V}
\]
as follows. For every $b\in\b$ and $v\in\SC{\V}$ we set
\begin{align*}
\SC{\pi_{\V}}\,b\ten v	&=\pi_{\V}\,\wt{b}\ten\wt{v}\mod\hbar\\
\intertext{where $\wt{b}\in B$ and $\wt{v}\in\V$ are arbitrarily lifts of $b,v$, and}
\SC{\pi_{\V}^*}\,v	&=\frac{1}{\hbar}\,\left(\id_B-\iota_B\circ\epsilon_B\right)\ten
\id\circ\pi_{\V}^*\,\wt{v}\mod\hbar.
\end{align*}
Note that the map $\SC{\pi_{\V}^*}$ is well--defined since $\pi_{\V}^*$
is an admissible coaction. Namely, set $p_B^{(n)}=(\id-\iota\circ\epsilon)^{\ten n}
\Delta^{(n)}$. By definition, $x\in B'$ if and only if $p_B^{(n)}(x)\in\hbar^nB^{\ten n}$ 
for any $n\geq0$. It is easy to verify that
\[
\Delta(p_B(x)) = p_B^{(2)}(x)+1\ten p_B(x)+p_B(x)\ten 1.
\]
It follows that $\hbar^{-1}p_B(x)$ modulo $\hbar$ is primitive, and
therefore belongs to $\b$.

\begin{proposition}
$(\SC{\V},\SC{\pi_{\V}}, \SC{\pi_{\V}^*})$ is a \DYt module over $\b$.
\end{proposition}

\begin{pf}
Set $V=\SC{\V}$, $\pi_V=\SC{\pi_{\V}}$, and $\pi_V^*=\SC{\pi_{\V}^*}$. 
It is clear that $(V,\pi_V)$ is a $\b$--module. To prove that $(V,\pi_V^*)$ is 
a $\b$--comodule it is enough to observe that
\[
\wt{p}\ten\wt{p}\circ(\Delta_B-\Delta^{21}_B)|_{B'}=\left.\delta\circ\wt{p}\right|_{B'}\mod\hbar
\]
where $\wt{p}=\hbar^{-1}p_B$. 
Finally, we have to prove the compatibility between $\pi_V$ and $\pi_V^*$, \ie
\begin{equation}\label{eq:Lie-DY-comp}
\pi^*_V\circ\pi_V=\id\ten\pi_V\circ(1\,2)\circ\id\ten\pi^*_V+[,]\ten\id\circ\id\ten\pi^*_V-
\id\ten\pi_V\circ\delta\ten\id.
\end{equation}
This follows by studying modulo $\hbar^2$
the relation \eqref{eq:DY Hopf} on the vector $\wt{b}\ten\wt{v}$.
More specifically, one has
\[
S^{-1}\ten(2\,3)\circ\Delta^{(3)}(\wt{b})=S^{-1}\ten\id^{\ten 2}\circ\Delta^{(3)}(\wt{b})-
\hbar 1\ten\delta(b) \mod\hbar^2.
\]
Therefore, \eqref{eq:DY Hopf} reads
\begin{align*}
\pi^*_{\V}\circ\pi_{\V}(\wt{b}\ten\wt{v})=m^{(3)}\otimes\pi_{\V} \circ &(1\,3\,4\,2) \circ 
S^{-1}\otimes\id^{\otimes 4}\circ \Delta^{(3)}\otimes\pi_{\V}^*(\wt{b}\ten\wt{v})\\
-&\hbar\, m^{(3)}\otimes\pi_{\V} \circ (1\,3\,4\,2) \circ 1\ten\delta\ten\pi_{\V}^*(b\ten\wt{v})\, 
\mod\hbar^2
.
\end{align*}
Now, since we are working modulo $\hbar^2$, it is enough to
consider the coaction up to its linear term
\[
\pi_{\V}^*(\wt{v})=1\ten\wt{v}+\hbar\wt{v}_1\ten\wt{v}_0 \mod\hbar^2.
\]
Then, we see that
\[
m^{(3)}\otimes\pi_{\V} \circ (1\,3\,4\,2) \circ S^{-1}\otimes\id^{\otimes 4}\circ \Delta^{(3)}
\otimes\id^{\ten 2}(\wt{b}\ten1\ten\wt{v})=\epsilon(\wt{b}_1)1\ten\pi_{\V}(\wt{b}_2\ten\wt{v}),
\]
which vanishes after composition with $\wt{p}\ten\id$, and
\begin{align*}
m^{(3)}\otimes&\pi_{\V} \circ (1\,3\,4\,2) \circ S^{-1}\otimes\id^{\otimes 4}\circ \Delta^{(3)}
\otimes\id^{\ten 2}(\hbar\,\wt{b}\ten\wt{v}_1\ten\wt{v}_0)\\
=&m_0^{(3)}\otimes\pi_{\V} \circ (1\,3\,4\,2) \circ S_0\otimes\id^{\otimes 4}\circ \Delta_0^{(3)}
\otimes\id^{\ten 2}(\hbar\,\wt{b}\ten\wt{v}_1\ten\wt{v}_0)
\,\mod\hbar^2,
\end{align*}
which reduces to
\[
\hbar\wt{v}_1\ten\pi_{\V}(b\ten\wt{v}_0)+\hbar(b\wt{v})_1-\wt{v}_1b)\ten\wt{v}_0.
\]
After composition with $\wt{p}\ten\id$ and modulo $\hbar$, this gives
\[
\id\ten\pi_V\circ(1\,2)\circ\id\ten\pi^*_V+[,]\ten\id\circ\id\ten\pi^*_V
\]
Finally, one has
\[
-\hbar\, m^{(3)}\otimes\pi_{\V} \circ (1\,3\,4\,2) \circ 1\ten\delta\ten\id^{\ten 2}(b\ten1\ten\wt{v})=
-\hbar\id\ten\pi_{\V}\circ\delta(b\ten\wt{v})
\]
which, after composition with $\wt{p}\ten\id$ and modulo $\hbar$, gives
\[
-\id\ten\pi_{V}\circ\delta(b\ten{v})
\]
and proves the compatibility \eqref{eq:Lie-DY-comp}.
\end{pf}

\subsection{Quantisation of representations}\label{ss:qRep}

We now return to the setting of \ref{ss:eksum}. 
Let $U,V\in\hDrY{\b}{\Phi}=\hEq{\gu}{\Phi}$, and $R^J_{U,V}\in\End_\sfK(\noEKff{\b}(U)\ten
\noEKff{\b}(V))$ the
twisted $R$--matrix defined by the diagram
\[\xymatrix@C=2cm{
\noEKff{\b}(U)\otimes \noEKff{\b}(V) \ar[r]^{(1\,2)R^J_{U,V}} \ar[d]_{J_{U,V}}& \noEKff{\b}(V)\otimes 
\noEKff{\b}(U)\ar[d]^{J_{V,U}}\\
\noEKff{\b}(U\otimes V) \ar[r]_{\noEKff{\b}(\beta_{U,V})} & \noEKff{\b}(V\otimes U)
}\]
where $\beta_{U,V}=(1\,2)\exp^{\hbar/2\,\Omega_{U,V}}$ is the commutativity
constraint in $\hDrY{\b}{\Phi}$. Then, the map
\[m^*_V:\noEKff{\b}(V)\to \noEKff{\b}(\Mm)\ten \noEKff{\b}(V)\qquad\qquad m^*_V(v)= 
R^J_{M_-,V}\,u_-\ten v\]
defines a right coaction of $\noEKff{\b}(M_-)$ on $\noEKff{\b}(V)$, which is compatible with the
left action of $\noEKff{\b}(M_-)$ on $V$ given by $m_V$ in the sense of \eqref{eq:DY Hopf}.
$(\noEKff{\b}(V), m_V, m^*_V)$ is therefore a \DYt module over $\noEKff{\b}(\Mm)$.

\begin{lemma}\label{le:Fb admissible}\hfill
\begin{enumerate}
\item The \DYt module $\noEKeq{\b}(V)=(\noEKff{\b}(V),m_V,m_V^*)$ is admissible.
\item The semiclassical limit of $\noEKeq{\b}(V)$ is equal to $V\in\DrY{\b}$.
\end{enumerate}
\end{lemma}
\begin{pf}
(i) We have to show that the coaction defined through the twisted $R$--matrix
is admissible, \ie it factors through $(\EK{\b})'\ten\noEKff{b}(V)$. It follows from 
\cite{e3} and \cite[Section 3.2]{ATL2} that the action of $R^J$ can be written
as an infinite sum of elements of the form
\[
\hbar^Nx_1\dots x_k\cdot y_1\cdots y_s \ten x_{k+1}\cdots x_N\cdot y_{s+1}\cdots y_{N}
\]
with $x_i\in\b, y_j\in\b^*$. Since $\b^*$ acts trivially on $u_-$, it follows that the first 
component of $R^J_{M_-,V}\,u_-\ten v$ lies in $(\EK\b)'$.

(ii) It is enough to observe from the formulae in \S\ref{ss:eksum} and \S\ref{ss:sc-lim-admDY} that 
the quantum action $m_V$ reduces to $\pi_V$ modulo $\hbar$. Similarly the quantum
coaction $m_V^*$ defined by the $R$--matrix satisfies
\[
\frac{1}{\hbar}(\id-\varepsilon)\ten\id\circ R^J_{\Mm,V}\circ u_-\ten\id=r_{\Mm,V}\circ u_-\ten\id=\pi_V^*
\quad\mod\hbar
\]
and gives back the coaction $\pi_V^*$ on $V$.
\end{pf}

\subsection{Tannakian equivalence}\label{ss:Tannaka}

\begin{theorem}\label{thm:EK-equiv}\cite{ek-6}
The fiber functor $\noEKff{\b}:\hDrY{\b}{\Phi}\to\vect_\sfK$ lifts to an equivalence
of braided tensor categories $\noEKeq{\b}:\hDrY{\b}{\Phi}\to\aDrY{\EK{\b}}.$
\end{theorem}
\noindent
The proof of Theorem \ref{ss:Tannaka} is given in \cite{ek-2,ek-6}, and relies on
the action of the Grothendieck--Teichmuller group on the braided tensor structure
of $\hDrY{\b}{\Phi}$ and $\aDrY{\EK{\b}}$.\footnote{We are grateful to Pavel Etingof
for pointing out that, contrary to what is stated in \cite[\S 4.1]{ek-6}, the functor
$\noEKeq{\b}$ is not an equivalence if the target category is taken to be \DYt
modules over $\Uhb$, and that attention needs to be restricted to admissible
ones.} We sketch an alternative proof in \S\ref{sss:qpropmod-2}, which relies
on a $\PROP$ic interpretation of the functor $\noEKeq{\b}$.

\subsection{Quantisation of bimodules}\label{ss:quant-bimod}

For later use, we will need to consider the quantisation of certain \DYt \emph{bimodules}.
Given two Lie bialgebras $\c$ and $\d$ (resp. Hopf algebras $C$ and $D$), a \DYt $(\c,\d)$--module
(resp. $(C,D)$--module) is a vector space endowed with  commuting actions and coactions
over $\c$ and $\d$ (resp. $C$ and $D$). It is immediate to verify that there are canonical
equivalences
\begin{eqnarray}
\DrY{(\c,\d)}&\simeq&\DrY{\c\oplus\d\op},\\
\DrY{(C,D)}&\simeq&\DrY{C\ten D\op},
\end{eqnarray}
where $\DrY{(\c,\d)}$ and $\DrY{(C,D)}$ denote the categories of bimodules, 
$\c\oplus\d\op$ is the direct sum Lie bialgebra, and $C\ten D\op$ is the tensor 
Hopf algebra.

\begin{proposition}
There is a commutative diagram of braided tensor functors
\[
\xymatrix{
\hDrY{\c\oplus\d\op}{\Phi}\ar[r]^(.4){\noEKeq{\c\oplus\d\op}}\ar@{=}[d]& 
\aDrY{\EK\c\ten\EK\d\op}\ar@{=}[d]\\
\hDrY{(\c,\d)}{\Phi} \ar[r]_(.4){\noEKeq{\c}\circ\noEKeq{\d}}
 & 
\aDrY{(\EK\c,\EK\d)}
}
\]
and similarly with $\noEKeq{\c}\circ\noEKeq{\d}$ replaced by $\noEKeq{\d}\circ\noEKeq{\c}$. 
\end{proposition}

\begin{pf}
Let $\Mv{\c\oplus\d\op}, \Mv{\c}, \Mv{\d\op}$ (resp. $\Mdv{\c\oplus\d\op}, \Mdv{\c}, \Mdv{\d\op}$)
be the Verma modules $\Mm$ (resp. $\Mpv$) corresponding
to the Lie bialgebras $\c\oplus\d\op$, $\c$, and $\d\op$.
As $\c\oplus\d\op$--modules, $\Mv{\c\oplus\d\op}\simeq\Mv{\c}\ten\Mv{\d\op}$ 
and $\Mdv{\c\oplus\d\op}\simeq\Mdv{\c}\ten\Mdv{\d\op}$. Moreover, there is an isomorphism
of functors
\begin{align*}
\noEKeq{\c\oplus\d\op}(V)
&=
\Hom^{\c\oplus\d\op}_{\c\oplus\d\op}(\Mv{\c}\ten\Mv{\d\op}, \Mdv{\c}\ten\Mdv{\d\op}\ten V)\\
&\simeq\; 
\Hom_{\c}^{\c}(\Mv{\c}, \Hom_{\d\op}^{\d\op}(\Mv{\d\op}, \Mdv{\c}\ten\Mdv{\d\op}\ten V))\\
&\simeq\;
\Hom_{\c}^{\c}(\Mv{\c}, \Mdv{\c}\ten\Hom_{\d\op}^{\d\op}(\Mv{\d\op}, \Mdv{\d\op}\ten V))
=\noEKeq{\c}\circ\noEKeq{\d\op}(V)
\end{align*}
where the first identification follows by adjointness and the second one follows by the
isomorphism of \DYt $\c$--modules between $\Hom_{\d\op}^{\d\op}(\Mv{\d\op}, \Mdv{\c}
\ten\Mdv{\d\op}\ten V)$ and $\Mdv{\c}\ten\Hom_{\d\op}^{\d\op}(\Mv{\d\op}, \Mdv{\d\op}\ten V)$. 
It is straightforward to verify that the isomorphism $\noEKeq{\c\oplus\d\op}\simeq\noEKeq{\c}
\circ\noEKeq{\d\op}$ preserves the tensor structure. 
\end{pf}


\section{Tensor structures on restriction functors}\label{s:Gamma}

\subsection{}

In this section, we consider a split inclusion of Manin triples
\[\iD: (\gd,\gdm,\gdp)\hookrightarrow(\gu,\gum,\gup).\]
We then define a relative version of the Verma modules $M_\pm$, and
use them to prove the following
\begin{theorem}\label{th:Gamma}
There is a tensor functor
\[(\FF{\a}{\b}, J\resped):\hEq{\gu}{\Phi}\longrightarrow\hEq{\gd}{\Phi}\]
such that $\FF{\a}{\b}$ is isomorphic to the restriction functor $\iD^*$. 
\end{theorem}


\subsection{Split inclusions of Manin triples}\label{ss:split}

\begin{definition}
An embedding of Manin triples
\[i:(\gd,\gdm,\gdp)\longrightarrow(\gu,\gum,\gup)\]
is a continuous Lie algebra homomorphism $i:\gd\to\gu$ preserving inner
products, and such that $i(\gdpm)\subset\gupm$.
\end{definition}

Denote the restriction of $i$ to $\gdpm$ by $i_\pm:\a_\pm\to\b_\pm$.
$i_\pm$ are continuous embeddings, and give rise to maps $p_\pm=
i_\mp^t:\gupm\to\gdpm$, defined
via the identifications $\gupm\simeq\gump^*$ and $\gdpm
\simeq\gdmp^*$ by $\iip{p_\pm(x)}{y}=\iip{x}{i_\mp(y)}$ for any $x\in
\g_{\pm}$ and $y\in\gdmp$. These maps satisfy $p_\pm\circ i_\pm=
\id_{\gdpm}$ since, for any $x\in\gdpm$, $y\in\gdmp$,
\[\iip{p_\pm\circ i_\pm(x)}{y}=\iip{i_\pm(x)}{i_\mp(y)}.\]
This yields in particular a direct sum decomposition $\gupm=i(\gdpm)
\oplus\Lmpm$, where
\[\Lmpm=\Ker(p_\pm)=\gupm\cap i(\gd)^\perp.\]

\begin{definition}
The embedding $i:(\gd,\gdm,\gdp)\longrightarrow(\gu,\gum,\gup)$ is called
{\it split} if the subspaces $\Lmpm\subset\gupm$ are Lie subalgebras.
\end{definition}

\subsection{Split pairs of Lie bialgebras}\label{ss:splitlba}

We now reformulate the above notion in terms of bialgebras.

\begin{definition}\label{def:splitlba}
A \emph{split pair} of Lie bialgebras is the data of
\begin{itemize}
\item Lie bialgebras $\a,\b$.
\item Lie bialgebra morphisms $i:\a\to\b$ and $p:\b\to\a$ such that $p\circ i=\id_{\a}$.
\end{itemize}
\end{definition}

\begin{proposition}\label{pr:one to one}
There is a one--to--one correspondence between split inclusions of
Manin triples and split pairs of Lie bialgebras. Specifically,
\begin{enumerate}
\item If $i:(\gd,\gdm,\gdp)\longrightarrow(\gu,\gum,\gup)$ is a split inclusion
of Manin triples, then $(\gdm,\gum,i_-,i_+^*)$ is a split pair of Lie bialgebras.
\item Conversely, if $(\a,\b,i,p)$ is a split pair of Lie bialgebras, then 
$i\oplus p^*:(\ga, \a, \a^*)\longrightarrow(\gb, \b, \b^*)$ is a split
inclusion of Manin triples.
\end{enumerate}
\end{proposition}

\subsection{Proof of (i) of Proposition \ref{pr:one to one}}\label{ss:(i) of one to one}

Given a split inclusion \[i=i_-\oplus i_+:(\gd,\gdm,\gdp)\longrightarrow(\gu,\gum,\gup),\]
we need to show that $i_-$ and $i_+^*$ are Lie bialgebra morphisms.
By assumption, $i_-$ is a morphism of Lie algebras, and $i_+^*$ one
of coalgebras. Since $i_-=(i_-^*)^*$, it suffices to show that $p_\pm=i
_\mp^*$ preserve Lie brackets. 

We claim to this end that $\m_\pm$ are ideals in $\gupm$. Since
$[\m_\pm,\m_\pm]\subseteq\m_\pm$ by assumption, this amounts
to showing that $[i(\gdpm),\m_\pm]\subseteq\m_\pm$. This
follows from the fact that $[i(\gdpm),\m_\pm]\subseteq\gupm$,
and from
\begin{multline*}
\iip{[i(\gdpm),\m_\pm]}{i(\gdmp)}
=
\iip{\m_\pm}{[i(\gdpm),i(\gdmp)]}\subset
\iip{\m_\pm}{i(\gdpm)}+\iip{\m_\pm}{i(\gdmp)},
\end{multline*}
where the first term is zero since $\gupm$ is isotropic, and the second
one is zero by definition of $\m_\pm$.

Let now $X_1,X_2\in\gupm$, and write $X_j=i_\pm(x_j)+y_j$, where
$x_j\in\gdpm$ and $y_j\in\m_\pm$. Since $\m_\pm=\Ker(p_\pm)$
and $p_\pm\circ i_\pm=\id$, we have $[p_\pm (X_1),p_\pm (X_2)]=
[x_1,x_2]$, while
\[p_\pm [X_1,X_2]=
p_\pm\left(i_\pm[x_1,x_2]+[i_\pm x_1,y_2]+[y_1,i_\pm x_2]+[y_1,y_2]\right)=
[x_1,x_2],\]
where the last equality follows from the fact that $\m_\pm$ is an ideal.

\subsection{Proof of (ii) of Proposition \ref{pr:one to one}}

The bracket on $\ga$ is defined by
\[[a,\phi]=\sfad^*(a)(\phi)-\sfad^*(\phi)(a)=-\iip{\phi}{[a,-]_{\a}}+\iip{\phi\ten\id}{\delta_{\a}(a)}\]
for any $a\in\a$, $\phi\in\a^*$. Analogously for $\gb$. Therefore, the equalities
\begin{align*}
\iip{p^*(\phi)\ten\id}{\delta_{\b}(i(a))}
=&\iip{\phi\ten\id}{(p\ten\id)(i\ten i)\delta_{\a}(a)}\\
=&\iip{\phi\ten\id}{(\id\ten i)\delta_{\a}(a)}=i(\iip{\phi\ten\id}{\delta_{\a}(a)})
\end{align*}
and
\begin{align*}
\iip{p^*(\phi)}{[i(a), b]_{\b}}=\iip{\phi}{p([i(a), b]_{\b})}=\iip{\phi}{[a, p(b)]_{\a}}
\end{align*}
for all $a\in\a$ and $b\in\b$, imply that the map $i\oplus p^*$ is a Lie algebra
map. It also respects the inner product, since for any $a\in\a,\phi\in\a^*$,
\[\iip{p^*(\phi)}{i(a)}=\iip{\phi}{p\circ i(a)}=\iip{\phi}{a}\]
Finally, $\m_-=\Ker(p)$ and $\m_+=\Ker(i^*)$ are clearly subalgebras.


\subsection{Parabolic Lie subalgebras}\label{s:parabolic} 

Let
\[\iD=\im\oplus\ip:(\gd,\gdm,\gdp)\to(\gu,\gum,\gup)\]
be a split embedding of Manin triples. The following summarizes the
properties of the subspaces
\[\Lmpm=\gupm\cap\iD(\gd)^{\perp}\aand\Lppm=\Lmpm\oplus\iD(\gd).\]
\begin{proposition}\hfill\break\vspace{-0.4cm}
\begin{itemize}
\item[(i)] $\Lmpm$ is an ideal in $\gupm$, so that $\gupm=\Lmpm\rtimes\iD_\pm(\gdpm)$.
\item[(ii)] $[\iD(\gd),\m_\pm]\subset\Lmpm$, so that $\Lp_{\pm}=\m_\pm\rtimes\iD(\gd)$
are Lie subalgebras of $\gu$.
\item[(iii)] $\delta(\m_-)\subset\m_-\ten\im(\gdm)+\im(\gdm)\ten\m_-$, so that $\m_-
\subseteq\gum$ is a coideal.
\end{itemize}
\end{proposition}

\begin{pf} (i) was proved in \S\ref{ss:(i) of one to one}. (ii) Since
\[\iip{[\iD(\gd),\m_\pm]}{\iD(\gd)}=\iip{\m_\pm}{[\iD(\gd),\iD(\gd)]}=0,\]
we have $[\iD(\gd),\m_\pm]\subset\iD(\gd)^{\perp}=\m_-\oplus\m_+$. Moreover,
\[\iip{[\iD(\gd),\m_\pm]}{\m_\pm}=\iip{\iD(\gd)}{[\m_\pm,\m_\pm]}=\iip{\iD(\gd)}{\m_\pm}=0,\]
since $\m_\pm$ is a subalgebra, and it follows that $[\iD(\gd),\m_\pm]\subset\m_\pm$.
(iii) is clear since $\Lmm$ is the kernel of a Lie coalgebra map.
\end{pf}

\subsection{The relative Verma Modules}\label{ss:rel Verma}

\begin{definition}
Given a split embedding of Manin triples $\iD:\gd\to\gu$, and the corresponding
decomposition $\gu=\Lmm\oplus\Lpp$, the relative Verma modules $\Lm,\Np$
are defined by
\[\Lm=\ind_{\Lpp}^{\gu}\sfk\aand\Np=\ind_{\Lmm}^{\gu}\sfk.\]
\end{definition}
Since $\Lpp$ and $\Lmm$ are invariant under the adjoint action of $\iD(\ga)$,
the right action of $\ga$ on $U\gb$ descends to one on $\Lm$ and $\Np$, so
both are $(\gb,\ga)$--bimodules, with the right action of $\ga$ on $\Lm$ being
trivial.

\begin{proposition}
The $(\gu,\gd)$--modules $\Lm$ and $\Npv$ are equicontinuous.
\end{proposition}

\noindent
The description of the appropriate topologies on $\Lm$ and $\Npv$, and the
proof of their equicontinuity will be carried out in \S\ref{ss:Lm}--\S\ref{ss:end N+*}.

\subsection{Equicontinuity of $\Lm$}\label{ss:Lm}

As vector spaces, 
\[\Lm\simeq U\Lmm\subset U\gum.\]
It is therefore natural to equip $\Lm$ with the discrete topology. The set of
operators $\{\pi_{\Lm}(x)\}_{x\in\gu}$ is then an equicontinuous family, and
the continuity of $\pi_{\Lm}$ reduces to checking that, for every element
$v\in \Lm$, the set
\[Z_v=\{X\in\gu|\; X\,v=0\}\]
is a neighborhood of zero in $\gu$. Since $U\Lmm$ embeds in $U\gum$, the
proof is identical to that of \cite[Lemma 7.2]{ek-1}. We proceed by induction on
the length of $v=Y_n\cdots Y_1\mathbf{1}_{-}$, $Y_i\in\Lmm$. If $n=0$, then
$v=\mathbf{1}_-$ and $Z_v=\Lpp=\im(\a_-)\oplus\gup$ is open in $\gb$. If
$n\geq 1$, the identity
\[X\, Y_n\cdots Y_1{\mathbf 1_-}=
Y_nXY_{n-1}\cdots Y_1{\mathbf 1_-}+
[X,Y_n]Y_{n-1}\cdots Y_1{\mathbf 1_-}\]
shows that $Z_{Y_n\cdots Y_1{\mathbf 1_-}}\supset Z_{Y_{n-1}\cdots Y_1{\mathbf 1_-}}
\cap\ad(Y_n)^{-1}(Z_{Y_{n-1}\cdots Y_1{\mathbf 1_-}})$, and the conclusion
follows from the continuity of the bracket.

\subsection{Topology of $\Np$}\label{ss:Npv-1}
As vector spaces,
\[\Np=\text{Ind}_{\Lmm}^{\gu}{\sfk}\simeq U\Lp_+\simeq\colim U_n\Lp_+,\]
where $\{U_n\Lp_+\}$ denotes the standard filtration of $U\Lp_+$, so that
\[U_n\Lp_+\simeq\bigoplus_{m=0}^nS^m\Lp_+=
\bigoplus_{i+j\leq n}\left(S^i\gup\otimes S^j\gdm\right).\]
We turn this isomorphism into an isomorphism of topological vector spaces,
by taking on $S^i\gup$ and $S^j\gdm$ the topologies induced by the
embeddings
\[S^i\gup\hookrightarrow (\gum^{\otimes i})^*
\aand S^j\gdm\hookrightarrow(\gdm^{\otimes j})^{**}.\]
With respect to these topologies, $U_m\Lp_+$ is closed inside $U_n\Lp_+$
for $m<n$, and we equip $\Np$ with the direct limit topology. We shall need the
following
\begin{lemma}
For any $x\in\gu, y\in\gd$, the maps $\pi_{\Np}(x):\Np\to \Np$ are continuous.
\end{lemma}
\begin{pf}
We need to show that for any neighborhood of the origin $U\subset \Np$, there
exists a neighborhood of zero $U'\subset \Np$ such that $\pi_{\Np}(x)U'\subset
U$. The topology on $\Np$ comes from the decomposition $U\Lpp\simeq 
U\gup\ten U\gdm$, so that an open neighborhood of zero in $\Np$ has the form 
$U\ten U\gdm + U\gup\ten V$, with $U$ open in $U\gup$ and $V$ open in $U\gdm$. 
We apply the same procedure used in \cite[Lemma 7.3]{ek-1} to construct an open set 
$U'\ten U\gdm$, with $U'$ open in $U\gup$, such that 
\[\pi_{\Np}(x)(U'\ten U\gdm)\subset U\ten U\gdm\subset U\ten U\gdm + U\gup\ten V.\]
\end{pf}

\subsection{Topology of $\Npv$}

As vector spaces, $\Npd\simeq (U\Lp_+)^*\simeq\lim(U_n\Lp_+)^*$.
In analogy with \S\ref{ss:vermaEK}, we consider the discrete topology on $(U_n\Lp_+)^*$ and
the limit topology on $\Npd$, denoting the resulting topological space by $\Npv$.
This defines a filtration of subspaces $\{(\Npv)_n\}$ on $\Npv$ by
\[0\to(\Npv)_n\to(U\Lp_+)^*\to(U_n\Lp_+)^*\to0,\]
so that $\Npv\supset (\Npv)_0\supset(\Npv)_1\supset\cdots$, and
we get an isomorphism of vector spaces
\[\Npv\simeq\lim \Npv/(\Npv)_n.\]
\begin{lemma}
$\{\pi_{\Npv}(x)\}_{x\in\g}$ is an equicontinuous family of operators.
\end{lemma}
\begin{pf}
Since $\Lp_+$ acts on $N_+$ by multiplication, $\Lp_+(\Npv)_n\subset(\Npv)_{n-1}$.
If $x\in\Lmm$ and $x_i\in U\Lp_+$ for $i=1,\dots, n$, then in $U\gb$,
\[xx_1\cdots x_n=x_1\cdots x_nx-\sum_{i=0}^nx_1\cdots x_{i-1}[x_i,x]x_{i+1}\cdots x_n,\]
where $[x_i,x]\in\gu$. Iterating shows that $(x.f)(x_1\cdots x_n)=0$ if $f\in (\Npv)_n$,
so that $x(\Npv)_n\subset(\Npv)_n$. Then, for any neighborhood of zero of the form
$U=(\Npv)_n$, it is enough to take $U'=(\Npv)_{n+1}$ to get $\gu(\Npv)_{n+1}\subset
(\Npv)_n$.
\end{pf}

\subsection{Equicontinuity of $\Npv$} \label{ss:end N+*}
In order to show that the module $\Npv$ is equicontinuous it is enough to prove 
the following
\begin{lemma}
The map $\pi_{\Npv}: \gu\to\emph{End}(\Npv)$ is a continuous map.
\end{lemma}
\begin{pf}
Since $\gum$ is discrete, it is enough to check that, for any $f\in\Npv$ and $n\in\IN$,
the subset
$$Y(\noEKff{\b},n)=\{b\in\gup|\;b.f\in(\Npv)_n\}$$
is open in $\gup$, and $b^i.f\in(\Npv)_n$ for all but finitely many $i\in I$.
Since $f\in \Npv\simeq\lim \Npv/(\Npv)_n$, we have $f=\{f_n\}$ where $f_n$ is the class
of $f$ modulo $(\Npv)_n$.
The classes $\{f_n\}$ are identified with elements in
$(U_n\Lpp)^*\simeq\oplus(S^j\Lpp)^*\simeq\oplus S^j\Lpm$ and we 
denote by $T_n(f)\subset I$ the finite set of indices corresponding to the elements $b_i$
appearing in the expression of $f_n$. Following \cite[Lemma 7.3]{ek-1}, for any finite set $J\subset I$,
we denote by $S(J)\subset I$ the finite set of indices corresponding to the generators $b^i$
satisfying $([b\ten b^i],\delta(b_j))\neq0$ for some $j\in J$, $b\in\gup$, and we define iteratively
the finite sets $S_{n+1}(J)=S(S_{n}(J))$, $S_0(J)=J$. It follows immediately that for every
$i\in I\setminus S_n(T_{n+1}(f))$ $b^i.f\in(\Npv)_n$. 
\end{pf}

Similarly, one shows that the right $\gd$--action on $\Npv$
is equicontinuous by adapting the steps in \S\ref{ss:Npv-1}--\S\ref{ss:end N+*}. 
This completes the proof of Proposition \ref{ss:rel Verma}.

\subsection{Coalgebra structure on $L_-$ and $N_+$}

Define the $(\gu,\ga)$--module maps
\[i_-:\Lm\to\Lm\ten\Lm\aand i_+:\Np\to\Np\ten\Np\]
by mapping $\1_\mp$ to $\1_\mp\otimes\1_\mp$. Note that, under the identifications
$\Lm\simeq U\Lmm$ and $\Np\simeq U\Lpp$, $\im$ and $\ip$ correspond to the
coproducts on $U\Lmm$ and $U\Lpp$ respectively.

Following \cite[Prop. 1.2]{drin-4}, we consider the invertible element 
$T\in U\gb{\fml}^{\ten 2}$ satisfying the relations\footnote{$T$ is the element underlying the
identification $(V\ten W)^*\simeq W^*\ten V^*$ in the category of representations
of a quasi--Hopf algebra.}
\begin{eqnarray*}
S^{\ten 3}({\PhiU}^{321})\cdot(T\otimes 1)\cdot(\Delta\otimes 1)(T)&=&
(1\otimes T)(1\otimes \Delta)(T)\cdot{\PhiU},\\
T\Delta(S(a))&=&(S\otimes S)(\Delta(a))T.
\end{eqnarray*}
Let $\Npv$ be as before and $f,g\in\Npv$. Consider the linear functional in $\Hom_{\sfk}(\Np,\sfk)$ 
defined by
\[v\mapsto (f\otimes g)(T\cdot i_+(v)).\]
This functional is continuous, so it belongs to $\Npv$ and induces a map
$i_+^{\vee}\in\Hom_{\sfk}(\Npv\otimes\Npv,\Npv)$ by
\[i_+^{\vee}(f\otimes g)(v)=(g\otimes f)(T\cdot i_+(v)).\]
For any $a\in\g$, we have
\begin{align*}
i_+^{\vee}(a(f\otimes g))(v)&=(f\otimes g)((S\otimes S)(\Delta(a))T\cdot i_+(v))\\
&=(f\otimes g)(T\Delta(S(a))\cdot i_+(v))\\
&=i_+^{\vee}(f\otimes g)(S(a).v)=(a.i_+^{\vee}(f\otimes g))(v)
\end{align*}
therefore $i_+^{\vee}\in\Hom_{\g}(\Npv\ten\Npv,\Npv)$.

\subsection{}

Set
\[\Phi_\a=\Phi(\hbar\Omega^\a_{12},\hbar\Omega^\a_{23})
\aand
\Phi_\b=\Phi(\hbar\Omega^\b_{12},\hbar\Omega^\b_{23})\]
and let $\hEq{(\gb,\ga)}{\Phi}$ be the Drinfeld category of equicontinuous
$(\gb,\ga)$--bimodules, with commutativity and associativity constraints
given respectively by
\[(1\,2)\circ\exp(\half{\hbar}\Omega^\b)\circ\exp(-\half{\hbar}\Omega^\a)^{\rho}
\aand
\Phi_\b\circ(\Phi_a^{-1})^{\rho}\]
where $(-)^{\rho}$ denotes the right $\gd$--action.

The following shows that $\Lm$ and $\Npv$ are coalgebra and algebra objects
in $\hEq{(\gb,\ga)}{\Phi}$.

\begin{proposition}\label{pr:coalgebra}
The following relations hold
\begin{itemize}
\item[(i)] As morphisms $\Lm\to\Lm\otimes(\Lm\otimes\Lm)$,
\[{\PhiU}(i_-\otimes 1)i_-=(1\otimes i_-)i_-.\]
\item[(ii)] As morphisms $(\Npv\otimes\Npv)\otimes\Npv\to\Npv$
\[i_+^{\vee}(1\otimes i_+^{\vee}){\PhiU}=i_+^{\vee}(i_+^{\vee}\otimes 1)S^{\ten 3}(\PhiD^{-1})^{\rho}.\]
\end{itemize}
\end{proposition}
\begin{pf}
We begin by showing that
\begin{equation}\label{eq:coalgebra}
{\PhiU}\,\1_-^{\otimes 3}=\1_{-}^{\otimes 3}\aand{\PhiU}\,\1_+^{\otimes 3}=\PhiD\,\1_+^{\otimes 3}.
\end{equation}
To prove the first identity, it is enough to notice that, since $\gup\1_-=0$ and $\Omega=\sum 
(a_i\otimes b^i+b^i\otimes a_i)$, $\Omega_{ij}(\1_{-}^{\otimes 3})=0$.
Then ${\PhiU}\,\1_{-}^{\otimes 3}=\1_{-}^{\otimes 3}$. To prove the second one, we notice that
$\Lmm\1_+=0$ and that we can rewrite 
\[\Omega=\sum_{j\in I_{\a}}(a_j\otimes b^j + b^j\otimes a_j)+
\sum_{i\in I\setminus I_{\a}}(a_i\otimes b^i + b^i\otimes a_i)=\Omega_{\a}+
\sum_{i\in I\setminus I_{\a}}(a_i\otimes b^i + b^i\otimes a_i),\]
where $\{a_j\}_{j\in I_{\a}}$ is a basis of $\gdm$ and $\{b^j\}_{j\in I_{\a}}$ is the dual basis of $\gdp$. 
Then
\[\Omega_{ij}\,\1_+^{\otimes 3}=\Omega_{\a,ij}\,\1_+^{\otimes 3}\]
and, since for any element $x\in\gd$ the right and the left $\gd$-action coincide on $\1_{+}$, \ie 
$x.\1_{+}=\1_{+}.x$, we have
\[\Omega_{ij}\,\1_+^{\otimes 3}=\,\1_+^{\otimes 3}\Omega_{\a,ij}\]
and consequently ${\PhiU}\,\1_+^{\otimes 3}={\PhiD}\,\1_+^{\otimes 3}$.

To prove (i), note that since the comultiplication in $U\Lmm$ is coassociative, we
have $(i_-\otimes 1)i_-=(1\otimes i_-)i_-$. We therefore have to show that ${\PhiU}
(i_-\otimes 1)i_-=(1\otimes i_-)i_-$. This is an obvious consequence of \eqref
{eq:coalgebra} and the fact that $\Lm$ is generated by $\1_-$.

To prove (ii), consider $v\in\Np$, $f,g,h\in\Npv$, then
\begin{align*}
\begin{split}
\ipd&(1\ten\ipd)({\PhiU}(f\ten g\ten h))(v)\\
&=(h\ten g\ten f)((S^{\ten 3}({\PhiU}^{321})\cdot(T\otimes 1)\cdot(\Delta\otimes 1)(T))\cdot(\ip\ten 1)\ip(v))\\
&=(h\ten g\ten f)((1\otimes T)(1\otimes \Delta)(T)\cdot{\PhiU}(\ip\ten 1)\ip(v))\\
&=(h\ten g\ten f)((1\otimes T)(1\otimes \Delta)(T)(1\ten \ip)\ip(v)\PhiD)\\
&=(S^{\ten3}(\PhiD)^{\rho}(h\ten g\ten f))((1\otimes T)(1\otimes \Delta)(T)(1\ten \ip)\ip(v))\\
&=\ipd(\ipd\ten1)(S^{\ten3}(\PhiD^{321})^{\rho}(f\ten g\ten h))(v)\\
&=\ipd(\ipd\ten1)S^{\ten 3}(\PhiD^{-1})^{\rho}(f\ten g\ten h)(v)
\end{split}
\end{align*}
and (ii) is proved.
\end{pf}

\subsection{The relative fiber functor}\label{ss:gammafun}

To any representation $V\in\hEq{\gu}{\Phi}$, we can associate the
$\sfk{\fml}$--module
\[\FF{\a}{\b}(V)=\Hom_{\hEq{\gu}{\Phi}}(\Lm,\Npv\ten V),\]
where $\Hom_{\Eq{\gu}}$ is the set of continuous homomorphisms, equipped with
the weak topology. The right $\gd$--action on $N_+^*$ endows 
$\FF{\a}{\b}(V)$ with the structure of a left $\gd$--module.

\begin{proposition}\label{pr:isom to forget}
For any $V\in\hEq{\gu}{\Phi}$, $\FF{\a}{\b}(V)$ is isomorphic to $V$
as equicontinous $\gd$--module. The isomorphism is given by
\[\alpha_V: f\mapsto (\1_+\ten1)\noEKff{\b}(\1_-)\]
for any $f\in\Hom_{\Eq{\gu}}(\Lm,\Npv\ten V)$. 
The assignment $V\mapsto\FF{\a}{\b}(V)$ extends to a functor
$\FF{\a}{\b}:\hEq{\gu}{\Phi}\to\hEq{\gd}{\Phi}$.
\end{proposition}

The proof is carried out in \S\ref{ss:gammafun-1} and \S\ref{ss:gammafun-2}.

\subsection{}\label{ss:gammafun-1}
By Frobenius reciprocity, we get an isomorphism
\[
\Hom_{\gu}(\Lm, \Npv\ten V)\simeq\Hom_{\Lpp}(\sfk, \Npv\ten V)\simeq\Hom_{\sfk}(\sfk, V)\simeq V\]
given by the map
\[ f\mapsto (\1_+\ten1)\noEKff{\b}(\1_-).\]
For $f\in \FF{\a}{\b}(V)$ and $x\in U\gd$, $x.f\in \FF{\a}{\b}(V)$ is defined by
\[x.f=(S(x)^{\rho}\otimes\id)\circ f.\]
For any $x\in U\gd$, we have
\[\sum_{i,j}x_i^{(1)}f_j\ten x_i^{(2)}v_j=\varepsilon(x)\noEKff{\b}(1_-),\]
where $\Delta(x)=\sum_i x_i^{(1)}\ten x_i^{(2)}$ and $\noEKff{\b}(\1_-)=\sum_j f_j\ten v_j$.
Using the identity
\[1\ten x  = \sum_i (S(x_i^{(1)})\ten 1)\cdot\Delta(x_i^{(2)}),\]
we obtain
\[(1\ten x)\noEKff{\b}(\1_-)=\sum_i (S(x_i^{(1)}\varepsilon(x_i^{(2)}))\ten 1) \noEKff{\b}(\1_-)
= (S(x)\ten 1)\noEKff{\b}(\1_-).\]
Finally, we have
\begin{align*}
x.\alpha_V(f)&=\iip{\1_+\ten\id}{(1\ten x)\noEKff{\b}(\1_-)}\\&=\iip{\1_+\ten\id}{(S(x)\ten 1)\noEKff{\b}(\1_-)}\\
&=\iip{\1_+\ten\id}{(S(x)^{\rho}\ten 1)\noEKff{\b}(\1_-)}=\alpha_V(x.f).
\end{align*}
Therefore, $\FF{\a}{\b}(V)$ is isomorphic to $V$ as equicontinuous $\gd$-module.

\subsection{}\label{ss:gammafun-2}

For any continuous $\varphi\in\Hom_{\gb}(V,V')$, define a map $\FF{\a}{\b}(\varphi):
\FF{\a}{\b}(V)\to \FF{\a}{\b}(V')$ by
\[\FF{\a}{\b}(\varphi):\; f\mapsto(\id\otimes \varphi)\circ f.\]
This map is continuous and for all $x\in\gd$
\[\FF{\a}{\b}(\varphi)(x\,f)=(S(x)^{\rho}\ten\varphi)\circ f=x\,\FF{\a}{\b}(\varphi)(f),\]
therefore $\FF{\a}{\b}(\varphi)\in\Hom_{\gd}(\FF{\a}{\b}(V),\FF{\a}{\b}(V'))$.
Since the diagram
\[\xymatrix{\FF{\a}{\b}(V)\ar[d]_{\alpha_{V}}\ar[r]^{\FF{\a}{\b}(\varphi)}&\FF{\a}{\b}(V')\ar[d]^{\alpha_{V'}}\\V\ar[r]^{\varphi}&V'}\]
is commutative for all $\varphi\in\Hom_{\gb}(V,V')$, we have a well--defined functor 
\[\FF{\a}{\b}:\hEq{\g}{{\PhiU}}\to\hEq{\gd}{\Phi}\]
which is naturally isomorphic to the pullback functor induced by the inclusion $i_{\a}
:\gd\hookrightarrow \gb$ via the natural transformation
\[\alpha_{V}: \FF{\a}{\b}(V)\simeq i_{\a}^*V.\]
This completes the proof of Proposition \ref{ss:gammafun}.


\subsection{Tensor structure on $\FF{\a}{\b}$}\label{s:tensor-structure}

Denote the tensor product in the categories $\hEq{\gu}{\Phi}$, $\hEq{\gd}{\Phi}$
by $\ten$, and let $B_{1234}$ and $B'_{1234}$ be the associativity constraints
\[B_{1234}:(V_1\ten V_2)\ten(V_3\ten V_4)\to V_1\ten((V_2\ten V_3)\ten V_4)\] 
and
\[B'_{1234}:(V_1\ten V_2)\ten(V_3\ten V_4)\to (V_1\ten(V_2\ten V_3))\ten V_4.\]

\noindent
For any $v\in \FF{\a}{\b}(V),w\in \FF{\a}{\b}(W)$, define $J_{VW}(v\ten w)\in\FF
{\a}{\b}(V\otimes W)$ by the composition
\begin{multline*}
\Lm\xrightarrow{\im}\Lm\ten\Lm\xrightarrow{v\ten w}(\Npv\ten V)\ten(\Npv\ten W)\xrightarrow{B}\Npv\ten((V\ten\Npv)\ten W)\\
\xrightarrow{\beta^{-1}_{23}}\Npv\ten((\Npv\ten V)\ten W)\xrightarrow{B'}(\Npv\ten\Npv)\ten(V\ten W)\xrightarrow{\ipd\ten 1}\Npv\ten(V\ten W).
\end{multline*}
The map $J_{VW}(v\ten w)$ is clearly
a continuous $\gu$-morphism from $\Lm$ to $\Npv\ten(V\ten W)$, so one gets a well-defined
map
\[J_{VW}:\FF{\a}{\b}(V)\ten\FF{\a}{\b}(W)\to\FF{\a}{\b}(V\ten W).\]

\begin{proposition}\label{pr:tensor structure}
The maps $J_{VW}$ are isomorphisms of $\gd$--modules, and define a tensor structure
on the functor $\FF{\a}{\b}$.
\end{proposition}

\noindent
The proof of Proposition \ref{pr:tensor structure} is carried out in 
\S\ref{ss:start tensor}--\S\ref{ss:end tensor}.

\subsection{}\label{ss:start tensor}

The map $J_{VW}$ is compatible with the $\gd$--action.
Indeed, $\ipd$ is a morphism of right $\gd$--modules and, 
for any $x\in\gd$,
\begin{align*}
x\,J_{VW}(v\ten w)&=(S(x)^{\rho}\ten\id)(\ipd\ten\id\ten\id)\wt{A}(v\ten w)\im\\
&=(\ipd\ten\id\ten\id)(\Delta(S(x))^{\rho})_{12}\wt{A}(v\ten w)\im\\
&=(\ipd\ten\id\ten\id)\wt{A}((S\ten S)(\Delta(x)))^{\rho})_{13}(v\ten w)\im=J_{VW}(x\,(v\ten w))
\end{align*}
where $\wt{A}=A'\beta_{32}^{-1}A$.

$J_{VW}$ is an isomorphism, since it is an isomorphism modulo $\hbar$. Indeed,
\[J_{VW}(v\ten w)\equiv (i_+^*\ten 1)(1\ten s\ten 1)(v\ten w)i_-  \mod\hbar.\]
To prove that $J_{VW}$ define a tensor structure on $\FF{\a}{\b}$, we need to show
that, for any $V_1,V_2,V_3\in\hEq{\gu}{\Phi}$ the following diagram is commutative
\[{\tiny
\xymatrix@R=1cm@C=1cm{(\FF{\a}{\b}(V_1)\ten \FF{\a}{\b}(V_2))\ten \FF{\a}{\b}(V_3) \ar[d]_{\PhiD} \ar[r]^(.53){J_{12}\ten 1} & \FF{\a}{\b}(V_1\ten V_2)\ten \FF{\a}{\b}(V_3) \ar[r]^{J_{12,3}} & \FF{\a}{\b}((V_1\ten V_2)\ten V_3) \ar[d]^{\FF{\a}{\b}({\PhiU})} \\
\FF{\a}{\b}(V_1)\ten(\FF{\a}{\b}(V_2)\ten \FF{\a}{\b}(V_3)) \ar[r]^(.53){1\ten J_{23}} & \FF{\a}{\b}(V_1)\ten \FF{\a}{\b}(V_2\ten V_3) \ar[r]^{J_{1,23}} & \FF{\a}{\b}(V_1\ten(V_2\ten V_3))}}\]
where $J_{ij}$ denotes the map $J_{V_i,V_j}$ and $J_{i j,k}$ the map $J_{V_i\ten
V_j,V_k}$.

\subsection{}

For any $v_i\in \FF{\a}{\b}(V_i)$, $i=1,2,3$, the map $\FF{\a}{\b}({\PhiU})J_{12,3}J_{12}\ten1
(v_1 \ten v_2\ten v_3)$ is given by the composition
\begin{multline*}
(1\ten{\PhiU})(i_+^*\ten 1^{\ten 3})A_4(1\ten\beta_{1\ten2, \Npv}\ten1)A_3((i_+^*\ten 1)\ten1^{\ten 3})(A_2\ten1\ten1)\\
\cdot(1\ten\beta_{\Npv,1}\ten 1^{\ten 3})(A_1\ten1\ten1)(v_1\ten v_2\ten v_3)(i_-\ten1)i_-,
\end{multline*}
where 
\[\begin{array}{ccc}
A_1=B_{\Npv,1,\Npv,2}, &\hspace{.5cm} &A_3=B_{\Npv,1\ten2,\Npv,3},\\[.1ex]
A_2=B^{-1}_{\Npv,\Npv,1,2}, &\hspace{.5cm} &A_4=B^{-1}_{\Npv,\Npv,1\ten2,3}.
\end{array}\]
By functoriality of associativity and commutativity isomorphisms, we have
\[A_3(i_+^*\ten1^{\ten 4})=(i_+^*\ten1^{\ten 4})A_5,\]
where $A_5=B_{\Npv\ten\Npv,12,\Npv,3}$,
\[(1\ten\beta_{12,\Npv}\ten1)(i_+^*\ten1^{\ten 4})=(i_+^*\ten1^{\ten 4})(1^{\ten2}\ten\beta_{12,\Npv}\ten1^{\ten2})\]
and
\[A_4(i_+^*\ten1^{\ten 4})=(i_+^*\ten1^{\ten 4})A_6,\]
where $A_6=B^{-1}_{\Npv\ten\Npv,\Npv,1\ten2,3}$. 
Finally, we have
\begin{multline}\label{eq:one side}
\FF{\a}{\b}({\PhiU})J_{12,3}(J_{12}\ten1)(v_1\ten v_2\ten v_3)\\
=(1^{\ten 3}\ten{\PhiU}_{123})((i_+^*(i_+^*\ten1))\ten 1^{\ten 3})
\,A\,(v_1\ten v_2\ten v_3)(i_-\ten 1)i_-,
\end{multline}
where
\[A=A_6(1^{\ten 2}\ten\beta_{1\ten2,\Npv}\ten1^{\ten 2})A_5(A_2\ten1^{\ten 2})(1\ten\beta_{\Npv,1}\ten1^{\ten 3})(A_1\ten1\ten1).\]

\subsection{}

On the other hand, $J_{1,2\ten3}(1\ten J_{23})\PhiD(v_1\ten v_2\ten v_3)$ corresponds
to the composition
\begin{equation*}
\begin{split}
(i_+^*\ten 1^{\ten 3})A_4'&(1\ten\beta_{\Npv,1}\ten1^{\ten 2})A_3'(1^{\ten 2}\ten i_+^*\ten1^{\ten 2})(1\ten1\ten A_2')\\&(1^{\ten 3}\ten\beta_{2,\Npv}\ten 1)(1\ten1\ten A_1')\PhiD(v_1\ten v_2\ten v_3)(1\ten i_-)i_-,
\end{split}
\end{equation*}
where 
\[\begin{array}{ccc}
A_1'=B_{\Npv,2,\Npv,3}, &\hspace{.5cm} &A_3'=B_{\Npv,1,\Npv,2\ten3},\\
A_2'=B^{-1}_{\Npv,\Npv,2,3}, &\hspace{.5cm} &A_4'=B^{-1}_{\Npv,\Npv,1,2\ten3}.
\end{array}\]
By functoriality of associativity and commutativity isomorphisms, we have
\[A_3'(1^{\ten 2}\ten i_+^*\ten 1^{\ten 2})=(1^{\ten 2}\ten i_+^*\ten 1^{\ten 2})A_5',\]
where $A_5'=B_{\Npv,1,\Npv\ten\Npv,2\ten 3}$,
\[(1\ten\beta_{1,\Npv}\ten1^{\ten 2})(1^{\ten 2}\ten i_+^*\ten 1^{\ten 2})=(1\ten i_+^*\ten 1^{\ten 3})(1\ten\beta_{1,\Npv\ten\Npv}\ten 1^{\ten 2}),\]
and
\[A_4'(1\ten i_+^*\ten 1^{\ten 3})=(1\ten i_+^*\ten 1^{\ten 3})A_6',\]
where $A_6'=B^{-1}_{\Npv,\Npv\ten\Npv,1,2\ten3}$. Thus,
\begin{multline}\label{eq:another side}
J_{1,23}(1\ten J_{23})\PhiD(v_1\ten v_2\ten v_3)\\
=
(i_+^*\ten 1^{\ten 3})((1 \ten i_+^*)\ten1^{\ten 3})
\,B\, 
\PhiD(v_1\ten v_2\ten v_3)(1\ten i_-)i_-,
\end{multline}
where
\[B=A_6'(1\ten\beta_{1,\Npv\ten\Npv}\ten1^{\ten 2})A_5'(1^{\ten 2}\ten A_2')(1^{\ten 3}\ten\beta_{2,\Npv}\ten 1)(1\ten1\ten A_1').\]

\subsection{}\label{ss:end tensor}

Comparing \eqref{eq:one side} and \eqref{eq:another side}, we see that it suffices
to show that the outer arrows of the following form a commutative diagram.
\[
\xymatrix@R=1.5cm@C=.1cm{& &\Lm \ar[dll]_{(i_-\ten 1)i_-}  \ar[drr]^{(1\ten i_-)i_-}& & \\
(\Lm\ten\Lm)\ten\Lm \ar[d]_{v_1\ten v_2\ten v_3}\ar[rrrr]^{\PhiU} & & & & \Lm\ten(\Lm\ten\Lm) 
\ar[d]^{\PhiD(v_1\ten v_2\ten v_3)}\\
{\scriptstyle ((\Npv\ten V_1)\ten(\Npv\ten V_2))\ten(\Npv\ten V_3)} \ar[d]_{A} \ar[rrrr]^{{\PhiU}}&&&&
{\scriptstyle (\Npv\ten V_1)\ten((\Npv\ten V_2)\ten(\Npv\ten V_3))}\ar[d]^{B}\\
{\scriptstyle ((\Npv\ten\Npv)\ten\Npv)\ten((V_1\ten V_2)\ten V_3)}\ar[d]_{(i_+^*(i_+^*\ten 1))\ten1^{\ten 3}} 
\ar[rrrr]^{{\PhiU}\ten{\PhiU}}&&&&
{\scriptstyle (\Npv\ten(\Npv\ten\Npv))\ten(V_1\ten (V_2\ten V_3))}\ar[d]^{(i_+^*(1\ten i_+^*))\ten1^{\ten 3}}\\
\Npv\ten((V_1\ten V_2)\ten V_3)\ar[rrrr]^{1\ten{\PhiU}}&&&&\Npv\ten(V_1\ten (V_2\ten V_3))
}
\]
Using the pentagon and the hexagon axiom, one shows that
$({\PhiU}\ten{\PhiU})A=B{\PhiU}$.
We have to show that 
\[\FF{\a}{\b}({\PhiU})J_{12,3}(J_{12}\ten1)(v_1\ten  v_2  \ten v_3)=J_{1,23}(1\ten J_{23})\PhiD(v_1\ten v_2\ten v_3)\]
in $\Hom_{\g}(\Lm,\Npv\ten(V_1\ten(V_2\ten V_3)))$:
\begin{equation*}
\begin{split}
J_{1,23}&(\id\ten J_{23})\PhiD(v_1\ten v_2\ten v_3)\\
&=(\ipd(\id\ten\ipd)\ten\id^{\ten 3})B\PhiD(v_1\ten v_2\ten v_3)(\id\ten i_-)\im\\
&=(\ipd(\id\ten\ipd)\ten\id^{\ten 3})B\PhiD(v_1\ten v_2\ten v_3){\PhiU}(\im\ten\id)\im\\
&=(\ipd(\id\ten\ipd)\ten\id^{\ten 3})B{\PhiU}\PhiD(v_1\ten v_2\ten v_3)(\im\ten\id)\im\\
&=(\ipd(\id\ten\ipd){\PhiU}\ten{\PhiU})A\PhiD(v_1\ten v_2\ten v_3)(\im\ten\id)\im\\
&=(\ipd(\id\ten\ipd){\PhiU}\ten{\PhiU})(S^{\ten 3}(\PhiD)^{\rho}\ten\id^{\ten 3})A(v_1\ten v_2\ten v_3)(\im\ten\id)\im\\
&=(\ipd(\id\ten\ipd){\PhiU} S^{\ten 3}(\PhiD)^{\rho}\ten{\PhiU})A(v_1\ten v_2\ten v_3)(\im\ten\id)\im\\
&=(\ipd(\ipd\ten\id)\ten{\PhiU})A(v_1\ten v_2\ten v_3)(\im\ten\id)\im\\
&=\FF{\a}{\b}({\PhiU})J_{12,3}(J_{12}\ten\id)(v_1\ten  v_2  \ten v_3),
\end{split}
\end{equation*}
where the second and seventh equalities follow from Proposition \ref{pr:coalgebra}, the fifth
one from the definition of the $\gd$--action on the modules $\FF{\a}{\b}(V_i)$ and the others 
from functoriality of the associator ${\PhiU}$. 
We shall henceforth denote the tensor structure on $\FF{\a}{\b}$ by $J\resped$.


\subsection{Infinitesimal of relative twist $J\resped$} 

The following is a straightforward extension of the computation
of the 1--jet of the Etingof--Kazhdan twist.

\begin{proposition}\label{pr:jets-rel-twist}
Under the natural identification
$\alpha_V:\FF{\a}{\b}(V)\to V$, the relative twist $J\resped$ satisfies
\[\alpha_{V\ten W}\circ J\resped\circ(\alpha_V^{-1}\ten\alpha_W^{-1})
\equiv
\id^{\otimes 2}+\frac{\hbar}{2}(r_\b+\iD^{\otimes 2}(r_{\a}^{21}))\quad\mod\hbar^2\]
in $\End(V\ten W)$, where $r_\b,r_\a$ are the canonical elements in
$\b\ten\b^*,\a\ten\a^*$ respectively.
\end{proposition}

\begin{pf}
For $v\in V, w\in W$, let 
\[\alpha^{-1}_V(v)(1_-)=\sum f_i\ten v_i\qquad \alpha^{-1}_W(w)(1_-)=\sum g_j\ten w_j\]
in $(\Npv\ten V)^{\Lpp}$ and $(\Npv\ten W)^{\Lpp}$ respectively. Then we observe
\[\langle (1_+\ten 1)^{\ten2},\ol{\Omega}_{23}\sum_{i,j}f_i\ten v_i\ten g_j\ten w_j\rangle
= -\ol{r}(v\ten w)\]
and
\[\langle (1_+\ten 1)^{\ten2},\iD^{\otimes 2}(\Omega_\a)_{23}\sum_{i,j}f_i\ten v_i\ten g_j
\ten w_j\rangle=-\iD^{\otimes 2}(\Omega_{\a})(v\ten w),\]
where $\Omega=\ol{\Omega}+\iD^{\otimes 2}(\Omega_{\a})$. Together with the fact that
$\Phi_\a,\Phi_\b=1^{\otimes 3}\mod\hbar^2$, we obtain
\[\alpha_{V\ten W}\circ J\resped\circ(\alpha_V^{-1}\ten\alpha_W^{-1})(v\ten w)\equiv v\ten 
w + \frac{\hbar}{2}(r_\b+\iD^{\otimes 2}(r_{\a}))(v\ten w)\mod\hbar^2.\]
\end{pf}

The computation of the $1$--jet of $J\resped$ shows, in particular,
that for $\a=\b$ the twist $J_{\b,\b}$ is not trivial.
However, by adapting the arguments of \cite[Prop. 9.7]{ek-1} and \cite[Prop. 3.10]{brochier}, 
one observes that the functor $\FF{\b}{\b}$ is in fact tensor equivalent
to the Etingof--Kazhdan functor $\noEKff{\b}$ equipped with the trivial tensor
structure, and it is therefore tensor equivalent to the identity functor on 
$\hEq{\g}{\Phi}\simeq\hDrY{\gum}{\Phi}$. 

 This complete the proof of Theorem \ref{th:Gamma}.


\section{Quantisation of Verma modules}\label{s:Gammah}

In this section, we construct a natural transformation $v_{\a,\b}$ making 
the following diagram commute
\[\xymatrix@C=2.5cm{
\hDrY{\b}{\Phi} \ar[r]^{\noEKeq\b} \ar[d]_{(F_{\a,\b}, J\resped)}& 
\aDrY{\Uhb} \ar[d]^{(\Res_{\Uha,\Uhb}, \id^{\otimes 2})}
\ar@{<=}[dl]_{v\resped}
\\
\hDrY{\a}{\Phi} \ar[r]_{\noEKeq\a} & 
\aDrY{\Uha}
}\]
where $(F_{\a,\b}, J\resped)$ is the tensor functor constructed in Section
\ref{s:Gamma}, and $\Res_{\Uhb,\Uha}$ is the tensor functor induced by
the split embedding $\Uha\hookrightarrow\Uhb$.
The proof relies on the construction of a quantum restriction functor $F
\resped^\hbar$ using a quantum version $\Lmh,\Nph$ of the relative Verma modules
$\Lm,\Npv$. The natural transformation $v\resped$ is then obtained 
under the assumption that the \nEK quantisation of $\Lm,\Npv$ is given
by their their quantum counterparts $\Lmh,\Nph$, a fact which will proved in Section
\S\ref{se:rel prop}. We begin in \S\ref{ss:quantum-verma}--\S\ref{ss:quant Mm Mpv}
by discussing the analogous, but simpler, quantisation of the Verma
modules $\Mm,\Mpv$. The quantum Verma modules $\Lmh,\Nph$ are
studied in \S\ref{ss:relquantumverma1}--\S\ref{ss:q restriction}. Their
definition relies on the Radford biproduct associated on a split pair
of Hopf algebras, which is reviewed in \S\ref{ss:splitHA}--\S\ref{ss:qrad-DY}

\subsection{The quantum Verma module $\MmB$}\label{ss:quantum-verma}

The Verma module $\Mm$ defined in \S \ref{ss:vermaEK} has a
natural counterpart in the category of \DYt modules over a Hopf
algebra $(B, m,\iota, \Delta, \epsilon, S)$ over $\sfk$. Following
\cite[\S 2.3]{ek-2}, define the \DYt $B$--module $\MmB$ to be
the $\sfk$--module $B$, with action $\pi_-$ given by the multiplication
$m$. Since $\MmB$ is a free $B$--module of rank one, with cyclic
vector $1$, the coaction $\pi_-^*$ on $\MmB$ is uniquely determined
by its compatibility \eqref{eq:DY Hopf} with the action, together
with the requirement that $\pi_-^* 1=1\otimes 1$. $\pi_-^*$ is
easily seen to be given by
\begin{equation}\label{eq:adj-coact}
\pi_-^*=m^{21}\ten\id\circ S^{-1}\ten (2\,3)\circ\Delta^{(3)}.
\end{equation}
which coincides with the adjoint coaction of $B$ on itself.

$\MmB$ satisfies the following universal property. 
For every $V\in\DrY{B}$,
\begin{equation}\label{eq:univ-prop-1}
\Hom_B^B(\MmB, V)\simeq\Hom^B(\sfk, V).
\end{equation}
The isomorphism \eqref{eq:univ-prop-1} is easily described. In one direction,
for any \DYt morphism $g:\MmB\to V$, one sets $\phi(g)=g\circ\iota$. Since
$\iota$ commutes with the coaction, $\phi(g)$ is a morphism of $B$--comodules.
Conversely, for any morphism of comodules $f:\sfk\to V$, one defines a map
$\psi(f):\MmB\to V$ by
\[
\psi(f)=\pi_V\circ\id\ten f,
\]
where $\pi_V$ denotes the action on $V$. It is easy to see that
$\psi(f)$ is a morphism of \DYt modules and $\phi,\psi$
are inverse of each other.

It follows from \eqref{eq:univ-prop-1} that $\MmB$ is naturally
endowed with a cocommutative coalgebra structure in $\DrY{B}$,
induced by the coproduct of $B$.

\subsection{The quantum Verma module $\MpB$}\label{ss:quantum-verma-p}

The module $\MpB$ is defined from $\MmB$ by reversing all arrows
and exchanging product with coproduct and action with coaction. 
Namely, set $\MpB=B$ with coaction $\pi^*_+=\Delta^{21}$, and
action
\begin{equation}\label{eq:adj-action}
\pi_+=m^{(3)}\circ(2\,3)\circ(\id\ten S^{-1}\ten\id)\circ\Delta^{21}\ten\id,
\end{equation}
given by the adjoint action of $B$ on itself. The module $\MpB=B$
satisfies the following universal property. For any $V\in\DrY{B}$, 
\begin{equation}\label{eq:univ-prop-2}
\Hom_B^B(V, \MpB)\simeq\Hom_B(V,\sfk).
\end{equation}
The isomorphism is described as follows. For any \DYt
morphism $g:V\to\MpB$, one sets $\phi(g)=\epsilon\circ g$. 
Since the counit commutes with the action of $B$, $\phi(g)$ is a
morphism of $B$--modules.
Conversely, for any morphism of modules $f:V\to\sfk$, one 
defines a map $\psi(f): V\to\MpB$ by
\[
\psi(f)=\id\ten f\circ\pi^*_V,
\]
where $\pi^*_V$ denotes the coaction on $V$. It is easy to see that
$\psi(f)$ is a morphism of \DYt modules and $\phi,\psi$ are
inverse of each other.
Moreover, it follows from \eqref{eq:univ-prop-2} that $\MpB$ is
naturally endowed with a commutative algebra structure in $\DrY{B}$,
induced by the product of $B$.

As pointed out in \S \ref{ss:DY-qD}, for any \fd Hopf algebra $B$,
the category of \DYt $B$--modules is equivalent to that of modules
over the quantum double $DB=B\otimes B^\circ$ (see also Appendix
\ref{s:app}). Under this equivalence, the module $\MmB$ corresponds
to the $DB$--module induced by the inclusion $B^\circ\subset DB$,
while $\MpB$ corresponds to the coinduced $DB$--module
corresponding to the inclusion $B\subset DB$.

\subsection{Admissibility of $\MmB$}\label{ss:adm-Mm}

If $B$ is a quantised enveloping algebra, it is natural to restrict attention
to the category of admissible \DYt module defined in \S \ref{ss:admDY},
\ie those whose coaction factors through the Hopf subalgebra $B'\subset
B$. 

\begin{proposition}\label{prop:adm-Mm}
Let $B$ be a QUE.
The adjoint coaction \eqref{eq:adj-coact} of $B$ on itself factors through
$B'\otimes B$. In particular, the \DYt module $\MmB$ is admissible.
\end{proposition}

\begin{pf}
By Proposition \ref{prop:adm-cond}, it is enough to check that
$\pi_-^*$ satisfies \eqref{eq:adm-cond}. One has
\begin{align*}
m^{21}\ten\id\circ S^{-1}\ten (2\,3)\circ(\id\ten\Delta^{21}\circ\Delta)
=&(m^{21}\circ S^{-1}\ten\id\circ\Delta)\ten\id\circ\Delta\\
=&(\iota\circ\epsilon)\ten\id\circ\Delta=\iota\ten\id\\
=&(\iota\circ\epsilon)\ten\id\circ\pi_-^*.
\end{align*}
It follows that 
\[
(\id-\iota\circ\epsilon)\ten\id\circ\pi_-^*=
m^{21}\ten\id\circ S^{-1}\ten (2\,3)\circ(\id\ten(\Delta-\Delta^{21})\circ\Delta),
\]
which is divisible by $\hbar$ since $\Delta-\Delta^{21}$ is.
Therefore $\pi_-^*$ satisfies \eqref{eq:adm-cond} and $\MmB$ is
admissible.
\end{pf}

\subsection{The quantum Verma module $\MpBp$}\label{ss:Mpv}

It is easy to check that the module $\MpB$ is not admissible. This suggests
modifying the definition of $\MpB$ in order to obtain a solution to the universal
property \eqref{eq:univ-prop-2} in the category of admissible \DYt modules.

\begin{proposition} Let $B$ be a QUE. 
\begin{enumerate}
\item The adjoint action \eqref{eq:adj-action} $\pi_+$ of $B$ onto itself preserves $B'\subset B$.
\item The $\sfK$--module $\MpBp=B'$, with action given by $\pi_+$ and coaction
by $\pi_+^*=\Delta^{21}$ is an admissible \DYt $B$--module.
\item $\MpBp$ satisfies the universal property \eqref{eq:univ-prop-2} in the
category $\aDrY{B}$.
\end{enumerate}
\end{proposition}
\begin{pf}
(ii) and (iii) follows from (i) and the previous discussion.
(i) Set $\delta^{(n)}=(\id-\iota\circ\varepsilon)^{\ten n}\circ\Delta^{(n)}$. We have to show 
that, for any $n\geq 1$, $\ima(\delta^{(n)}\circ\pi_+)\subset\hbar^nB^{\ten n}$. 
We proceed by induction. For $n=1$, 
\[
\delta^{(1)}\circ\pi_+=\pi_+\circ\id\ten\delta^{(1)}.
\] 
Namely, for every $b\in B, b'\in B'$, one has
\begin{align*}
b_2b'S^{-1}(b_1)-\varepsilon(b_2bS^{-1}b_1)&=b_2b'S^{-1}b_1-\varepsilon(b)\varepsilon(b')\\
&=b_2bS^{-1}b_1-b_2S^{-1}b_1\varepsilon(b')\\
&=b_2bS^{-1}b_1-b_2\varepsilon(b')S^{-1}b_1=\pi_+\circ\id\ten\delta^{(1)}.
\end{align*}
Assume now that $\ima(\delta^{(n)}\circ\pi_+)\subset\hbar^nB^{\ten n}$ and set 
$\Delta-\Delta^{21}=\hbar\,\Xi$.
Then
\begin{align*}
\delta^{(n+1)}\circ\pi_+&=
\id\ten\delta^{(1)}\circ\id\ten m^{(3)}\circ(1\,4\,3\,2)\circ\id\ten(\delta^{(n)}\circ\pi_+)
\ten\id^{\ten 2}\\
&\qquad\qquad\qquad\qquad\qquad\qquad
\circ S^{-1}\ten\id\ten(3\,4)\ten\id\circ\Delta^{(3)}\ten\Delta\\
&=(\delta^{(n)}\circ\pi_+)\ten(\delta^{(1)}\circ\pi_+)\circ(2\,3)\circ\Delta\ten\Delta\\
&\qquad\qquad
+\hbar\,\id\ten\delta^{(1)}\circ\id\ten m^{(3)}\circ(1\,4\,3\,2)\circ\id\ten(\delta^{(n)}
\circ\pi_+)\ten\id^{\ten 2}\\
&\qquad\qquad\qquad\qquad\qquad\qquad
\circ S^{-1}\ten\id\ten(3\,4)\ten\id\circ\Xi\ten\id^{\ten 2}\circ\Delta\ten\Delta.
\end{align*}
It follows by induction that $\ima(\delta^{(n+1)}\circ\pi_+)\subset\hbar^{n+1}B^{\ten n+1}$.
\end{pf}

\subsection{Quantisation of $\Mm$ and $\Mpv$}\label{ss:quant Mm Mpv}

Let now $\b$ be a Lie bialgebra, and $\EK\b$ its quantisation.
We denote by $\Mmh$ and $\Mpvh$ the admissible \DYt
$\EK\b$--modules $M_{\EK\b}$ and $M_{\EK\b}^{\vee}$,
respectively.

\begin{proposition}\label{pr:quantumverma}
The following holds in the category $\aDrY{\EK{\b}}$
\begin{itemize}
\item[(a)] $\noEKeq{\b}(\Mm)\simeq{\Mmh}$ as coalgebras,\\
\item[(b)] $\noEKeq{\b}(\Mpv)\simeq{\Mpvh}$ as algebras.
\end{itemize}
\end{proposition}
\begin{pf}
The Hopf algebra $\EK\b$ is constructed on the vector space $\noEKff{\b}(\Mm)$ 
with unit element $u\in \noEKff{\b}(\Mm)$ defined by $u(\um)=\coup\ten\um$, 
where $\coup\in\Mpv$ is given by $\coup(x1_+)=\cou(x)$ for any $x\in U\gup$. 
Consequently, the action of $\EK\b$ on $u\in \noEKff{\b}(\Mm)$ is free, as multiplication 
with the unit element. The coaction of $\EK\b$ on $\noEKff{\b}(\Mm)$ is defined using 
the $\R$-matrix associated to the braided tensor functor $F$,\ie
\[
\pi^*_{\Mm}:\noEKff{\b}(\Mm)\to \noEKff{\b}(\Mm)\ten \noEKff{\b}(\Mm),\quad \pi^*(x)=\R(u\ten x),
\]
where $x\in \noEKff{\b}(\Mm)$ and $\R_{VW}\in\End_{\EK\g}(\noEKff{\b}(V)\ten \noEKff{\b}(W))$ 
is given by $\R_{VW}=\sigma J_{WV}^{-1}\noEKff{\b}(\beta_{VW})J_{VW}$, 
$\{J_{V,W}\}_{V,W\in\DrY{\b}}$ being the tensor structure on $F$. It is easy to show that
$J(u\ten u)|_{\um}=\coup\ten\um\ten\um$, and, since $\Omega(\um\ten\um)=0$, we have
\[\pi^*_{\Mm}=\R(u\ten u)=u\ten u.\]
By construction the coalgebra structure on $\noEKff{\b}(\Mm)$ coincides with that on $\EK\b$. This proves $(a)$.\\

The module $\Mpv$ satisfies the following universal property: for any $V\in\hDrY{\b}{\Phi}$, one has
\[\Hom_{\b}^{\b}(V,\Mpv)\simeq\Hom_{\b}(V,\sfk).\]
Since $\noEKeq{\b}$ is an equivalence of categories,
\[\Hom_{\EK\b}^{\EK\b}(\noEKeq{\b}(V),\noEKeq{\b}(\Mpv))\simeq
\Hom_{\b}^{\b}(V,\Mpv)\simeq\Hom_{\b}(V,\sfk).\]
Using the isomorphism $\alpha_V:\noEKff{\b}(V)\to V$ defined by
\[\alpha_V(f)=\langle f(\um), \up\ten\id\rangle,\]
we obtain a map $\Hom_{\b}(V,\sfK)\to\Hom_{\sfK}(\noEKff{\b}(V),\sfK)$. 
For any $x\in U\b$ let $\psi_x:\Mm\to\Mpv\ten\Mm$ be the morphism  
defined by $\psi_x(\um)=\coup\ten x\um$.
It is clear that, if $f(\um)=f_{(1)}\ten f_{(2)}$ in sumless Swedler's notation,
\begin{eqnarray*}
\alpha_V(\psi_x.f)&=&\langle (\ipd\ten\id)\Phi^{-1}(\id\ten f)(\coup\ten x.\um), \up\ten\id \rangle\\
&=&\langle \Phi^{-1}(\coup\ten\id\ten\id)(\id\ten\Delta(x))(\id\ten f_{(1)}\ten f_{(2)}), (T\ten \id)(\up\ten\up\id) \rangle\\
&=&\langle \Delta(x)(f_{(1)}\ten f_{(2)}), \up\ten\id\rangle\\
&=&\langle f_{(1)},\up\rangle x.f_{(2)}\\
&=&x.\alpha_V(f),
\end{eqnarray*}
using the fact that $(\epsilon\ten1\ten1)(\Phi)=1^{\ten 2}$ and $(\epsilon\ten 1)(T)=1$.
So, clearly, if $\phi\in\Hom_{\b}(V,\sfk)$, then $\phi\circ\alpha_V\in\Hom_{\EK\b}(\noEKff{\b}(V),\sfk{\fml})$ and
we obtain an isomorphism of \DYt modules $\noEKff{\b}(\Mpv)\simeq\Mpvh$. By universal property,
this is compatible with the algebra structures and $(b)$ is proved.
\end{pf}

\subsection{Split pair of Hopf algebras}\label{ss:splitHA}

We briefly recall the characterisation of a Hopf algebra with a projection
given by Radford \cite{rad2}. Let $A\stackrel{i}{\to} B\stackrel{p}{\to}A$,
$p\circ i=\id_A$, be a split pair of Hopf algebras. One can describe the
kernel of the projection $p$ and obtain $B$ as a semidirect product of
Hopf algebras. Let $\pi\in\End(B)$ be the idempotent $\pi=i\circ p$, and
define $\Pi\in\End(B)$ by
\begin{equation}\label{eq:RadPi}
\Pi=m\circ\id\ten(S\circ \pi)\circ\Delta
\end{equation}
(equivalently, $\Pi=\id*(S\circ\pi)$ in the convolution algebra $\End(B)$).

It is easy to show, using standard graphical calculus, that $\Pi$ is
an idempotent such that
\[\Pi\circ\pi=\iota\circ\epsilon=\pi\circ\Pi.\]
Moreover, if $L=\Pi(B)\subseteq B$, then $B\simeq L\ten A$ via the mutually
inverse maps
\begin{equation}\label{eq:rad-iso}
\xymatrix@C=0.65in{B\ar@<2pt>[r]^{\Pi\ten\pi\circ\Delta} & L\ten A\ar@<2pt>[l]^{m}}.
\end{equation}
The following result describes the structure of $L$.

\begin{theorem}\hfill
\begin{enumerate}
\item The following identities hold:
\begin{eqnarray}\label{eq:L-rad-1}
\Pi\circ m&=&m^{(3)}\circ(2\,3)\circ\id\ten(S\circ\pi)\ten\id\circ\Delta\ten\Pi
\end{eqnarray}
and
\begin{eqnarray}\label{eq:L-rad-1-2}
\Pi\circ m\circ\id\ten\pi&=&\Pi\ten\epsilon.
\end{eqnarray}
Moreover, $L$ is a $B$--module with respect to the adjoint action of $B$
\begin{eqnarray*}
\pi_L&=&m^{(3)}\circ\id\ten\id\ten(S\circ \pi)\circ(2\,3)\circ\Delta\ten\Pi.
\end{eqnarray*}
\item The following identity holds
\begin{eqnarray}\label{eq:L-rad-2}
\id\ten\pi\circ\Delta\circ\Pi&=&\Pi\ten\id
\end{eqnarray}
and it characterises $L$. Moreover,  
$L$ is a subalgebra in $B$.
\item The following identities hold:
\begin{eqnarray}\label{eq:L-rad-3}
\Delta\circ\Pi&=&m\ten\id\circ(2\,3)\circ\id\ten\Pi\ten(S\circ\pi)\circ\Delta^{(3)}
\end{eqnarray}
and
\begin{eqnarray}\label{eq:L-rad-3-2}
\id\ten\pi\circ\Delta\circ\Pi&=&\Pi\ten\eta.
\end{eqnarray}
Moreover, the map 
\begin{eqnarray}
\label{eq:Lmh-coact}\pi_L^*&=&m^{21}\ten\id\circ S^{-1}\ten(2\,3)\circ\Delta^{(3)}\circ\Pi
\end{eqnarray}
defines a $B$--comodule structure on $L$.
\item $(L, \pi_L, \pi_L^*)$ is a \DYt $B$--module 
(and therefore a \DYt $A$--module by restriction), and satisfies the universal property
\begin{equation}\label{eq:Lmh-univ}
\Hom_{B}^{B}(L, V)\simeq\Hom_{A}^{B}(\sfk,V)
\end{equation}
for every $V\in\DrY{B}$.
\item $L$ is a cocommutative coalgebra object in $\DrY{B}$ with comultiplication
\begin{equation}\label{eq:Lmh-coprod}
\Delta_L=\Pi\ten\Pi\circ\Delta\circ\Pi.
\end{equation}
\item $L$ is a Hopf algebra in $\DrY{A}$.
\item If $(B,A)$ is a split pair of QUEs, then $L$ is an admissible 
\DYt $B$--module.
\end{enumerate}
\end{theorem}

\begin{pf}
The identities \eqref{eq:L-rad-1}, \eqref{eq:L-rad-1-2}, 
\eqref{eq:L-rad-2}, \eqref{eq:L-rad-3}, and \eqref{eq:L-rad-3-2}
are verified by standard graphical calculus for Hopf algebras
(e.g. \cite[Lecture 8]{es}). 

$(i)$ The identity \eqref{eq:L-rad-1} implies that the adjoint action of $B$ on itself 
preserves $L$, and therefore the module structure on $L$ is well--defined.

$(ii)$ One shows that
\[
\id\ten\pi\circ\Delta\circ (m\circ\Pi\ten\Pi)=(m\circ\Pi\ten\Pi)\ten\iota.
\]
Since \eqref{eq:L-rad-2} characterises $L$, it follows that $L$ closed under multiplication in $B$.

$(iii)$ It follows from \eqref{eq:L-rad-3} that the coaction \eqref{eq:Lmh-coact} takes 
values in $B\ten L$. One then shows directly that $\pi_L^*$ is a coaction.

$(iv)$ The compatibility between $\pi_L$ and $\pi_L^*$ is verified by direct inspection. 
Then, we observe that the action of $A$ and the coaction of $B$ on the 
unit element $u_L\in L$ are trivial. In particular, by restriction to $u_L\in L$, we 
get a map
\[
\Hom_{B}^{B}(L, V)\to\Hom_{A}^{B}(\sfk,V)
\]
for every $V\in\DrY{B}$, whose inverse is given by composition with the action of $B$ restricted
to $L$.

$(v)$ By \eqref{eq:Lmh-univ}, the coalgebra structure on $L$ is uniquely determined 
by the condition $\Delta_{L}(u_L\ten u_L)=
u_L\ten u_L$.
Since the action of $B$ on $u_L$ coincides with the projection $\Pi$, one recovers \eqref{eq:Lmh-coprod}.
Finally, since the $R$--matrix preserves $u_L\ten u_L$, \ie 
\[
\pi_L\ten\id\circ(2\,3)\circ\id\ten\pi_L^*(u_L\ten u_L)=u_L\ten u_L,
\]
it follows that $\Delta_L=(1\,2)\circ R\circ \Delta$, \ie $L$ is a cocommutative coalgebra in the category
$\DrY{B}$.

$(vi)$ The compatibility between the product and the coproduct on $L$ is 
a straightforward computation. 
It follows from \eqref{eq:L-rad-1-2} that the map $\Pi$ is a morphism of $A$--modules
with respect to the adjoint action of $A$ on $B$. Since the multiplication in $B$ commutes with 
the action of $A$ and $L$ is a subalgebra in $B$, $L$ is a $A$--module algebra.
Similarly one shows that $L$ is a $A$--comodule coalgebra.  

$(vii)$ The admissibility of $L$ follows as in Proposition \ref{prop:adm-Mm}.
\end{pf}

\subsection{Radford biproduct}\label{ss:rad-biprod}
The object $L\ten A$ is naturally endowed with a Hopf algebra structure, induced by
the identification $B\simeq L\ten A$ \eqref{eq:rad-iso}. The description of such structure 
relies exclusively on that of $A$ and $L$ (as Hopf algebra in $\DrY{A}$).

Conversely, given a Hopf algebra $A$ in $\vect$ and a Hopf algebra $L$ in $\DrY{A}$,
the tensor product $L\ten A$ is endowed with a Hopf algebra structure, called \emph{Radford biproduct} 
(further studied by Majid under the name of \emph{bosonisation} \cite{majid}).\\
\newcommand{\Rad}[2]{#1\star #2}

Namely, let $(A, m_A, \eta_A, \Delta_A,\varepsilon_A, S_A)$ be a Hopf algebra in $\vect$ 
and $(L, m_L, \eta_L, \Delta_L,$ $\varepsilon_L, S_L)$ a Hopf algebra in $\DrY{A}$,
with \DYt structure $(L,\pi_L,\pi_L^*)$. 
The Radford biproduct $\Rad{L}{A}$ is the Hopf algebra defined on $L\ten A$ 
with operations
\begin{eqnarray*}
m_{\Rad{L}{A}}&=&
m_L\ten m_A\circ\id\ten\pi_L\ten\id^{\ten 2}\circ(34)\circ\id\ten\Delta_A\ten\id^{\ten 2},\\
\Delta_{\Rad{L}{A}}&=&
\id\ten m_A\ten\id^{\ten 2}\circ(34)\circ\id\ten\pi_L^*\ten\id^{\ten 2}\circ
\Delta_L\ten\Delta_A,\\
S_{\Rad{L}{A}}&=&
\pi_L\ten\id^{\ten 2}\circ(2\,3)\circ\Delta_L\ten\id\circ S_L\ten S_A\circ 
m_A\ten\id\circ (2\,3)\circ\pi_L^*\ten\id.
\end{eqnarray*}
unit $\eta_L\ten\eta_A$ and counit $\varepsilon_L\ten\varepsilon_A$.
It is easy to see that $\Rad{L}{A}$ contains $L$ as a subalgebra and 
$A$ as a Hopf subalgebra. In fact, 
\[
\xymatrix{
A\ar[r]^(.4){\eta_L\ten\id} & \Rad{L}{A} \ar[r]^(.6){\epsilon_L\ten\id}& A
}
\]
is a split pair of Hopf algebras.

\subsection{Radford biproduct and \DYt modules}\label{ss:rad-dy}
Since $L$ is a Hopf algebra in $\DrY{A}$, a Drinfeld--Yetter module over the 
Radford biproduct $\Rad{L}{A}$ is conveniently described as a Drinfeld--Yetter module
over $L$ in the category $\DrY{A}$. 
\begin{proposition}\label{pr:DY equiv}
There is a canonical equivalence of braided tensor categories
\[\DrY{\Rad{L}{A}}\simeq\DrY{L,A},\]
where $\DrY{L,A}$ denotes the category of \DYt $L$--modules in $\DrY{A}$. 
Namely,
\begin{enumerate}
\item Any $(V, \pi_{L,V}, \pi_{L,V}^*, \pi_{A,V}, \pi_{A,V}^*)\in\DrY{L,A}$
is naturally a \DYt $\Rad{L}{A}$--module with action and coaction
\[
\pi_{\Rad{L}{A}, V}=\pi_{L,V}\circ\id\ten\pi_{A,V}
\aand
\pi_{\Rad{L}{A}, V}^*=\id\ten\pi_{A,V}^*\circ\pi_{L,V}^*.
\]
\item Conversely, any $(V,\pi_{\Rad{L}{A}, V},\pi_{\Rad{L}{A}, V}^*)\in
\DrY{\Rad{L}{A}}$ has, by restriction to $L$ and $A$, a structure of
\DYt $L$--module in $\DrY{A}$ with 
\[
\pi_{A,V}=\pi_{\Rad{L}{A},V}\circ\iota_L\ten\id_A\ten\id_V
\qquad
\pi_{A,V}^*=\varepsilon_L\ten\id_A\ten\id_V\circ\pi_{\Rad{L}{A},V}^*
\] 
and
\[
\pi_{L,V}=\pi_{\Rad{L}{A},V}\circ\id_L\ten\iota_A\ten\id_V
\qquad
\pi_{L,V}^*=\id_L\ten\varepsilon_A\ten\id_V\circ\pi_{\Rad{L}{A},V}^*.
\] 
\end{enumerate}
\end{proposition}

\begin{proof}
First, one shows that the assignments in $(i)$ and $(ii)$ define an equivalence 
between the tensor categories $\Rep(\Rad{L}{A})$ of $(\Rad{L}{A})$--modules 
and $\Rep_A(L)$ of $L$--modules in $\Rep(A)$ (cf. \cite[Sec. 1.7, Prop. 2]{sommer}).
This restricts to an equivalence of braided tensor categories
at the level of Drinfeld centers (cf. \cite[XIII.4]{kassel}), which are equivalent
to the categories $\DrY{\Rad{L}{A}}$ and $\DrY{L,A}$, respectively.
\end{proof}

\subsection{Admissible \DYt modules}\label{ss:qrad-DY}

In Section \S\ref{s:gammahopf}, we will study the relation between the
quantisation of a split pair of Lie bialgebras and the Radford biproduct.
In particular, we will need the following result.
\begin{proposition}
Let $(B,A)$ be a split pair of QUEs. The equivalence
given by Proposition \ref{pr:DY equiv} restricts
to an equivalence of admissible \DYt modules
\[
\aDrY{L,A}\simeq\aDrY{\Rad{L}{A}}.
\]
\end{proposition} 

This reduces to proving the following lemma. 

\begin{lemma}
Let $A$ be a QUE, $L$ a Hopf algebra object in $\aDrY{A}$
and $B=\Rad{L}{A}$ the corresponding Radford biproduct. Then
$B'=\Rad{L'}{A'}$,
where, for any (possibly braided) Hopf algebra $H$ over $\sfk{\fml}$,
we set  $p_H=\id_H-\iota_H\circ\epsilon_H$, $d_H^{(n)}=p_H^{\ten n}\circ\Delta_H^{(n)}$,
and
\[
H'=\{x\in H\;|\; d_H^{(n)}(x)\in\hbar^nH^{\ten n}\;\forall n\}.
\]
\end{lemma}

\begin{pf}
The Hopf algebras $(B',A')$ form a split pair, and, by Radford theorem, 
there is an isomorphism of Hopf algebras $B'\simeq\Rad{\wt{L}}{A'}$,
where $\wt{L}$ is the image of the idempotent \eqref{eq:RadPi} corresponding 
to $B'$. It is easy to show that
\[
\wt{L}=\{x\in B'\;|\;\id\ten\pi_A\circ\Delta_{B}(x)=x\ten1\}=L\cap B'.
\]
We want to show that 
\begin{align*}
L'=\{x\in L\;|\; d_L^{(n)}(x)\in\hbar^nL^{\ten n}\}=L\cap B',
\end{align*}
from which it follows that $B'=\Rad{L'}{A'}$.

It follows immediately from the description of the Radford biproduct
in \S\ref{ss:splitHA} that $\Delta_L^{(n)}=\Pi^{\ten n}\circ\Delta_B^{(n)}\circ\Pi$, 
and therefore
\begin{align*}
d_L^{(n)} & = p_L^{\ten n}\circ\Delta_L^{(n)} \\
& = \left((\id - \epsilon_L)\circ\Pi\right)^{\ten n}\circ\Delta_B\circ\Pi \\
& = \Pi^{\ten n}\circ(\id-\epsilon_B)^{\ten n}\circ\Delta_B\circ\Pi \\
& = \Pi^{\ten n}\circ d_B^{(n)}\circ \Pi .
\end{align*}
Hence, $L\cap B'\subset L'$. Conversely, since $\Delta_B(L)\subset B\ten L$ and, 
under the identification $B\simeq\Rad{L}{A}$, $p_B=p_L\ten\id+\epsilon_L\ten p_A$, 
one has
\begin{align*}
d_B^{(2)}\circ\Pi & = p_B^{\ten 2}\circ\Delta_B\circ\Pi \\
&=p_B^{\ten 2}\circ\id\ten\pi_L^*\circ\Delta_L\\
&=p_L\ten\id\ten p_L\circ \id\ten\pi_L^*\circ\Delta_L + 
\epsilon_L\ten p_A \ten p_L\circ \id\ten\pi_L^*\circ\Delta_L\\
&= \id\ten\pi_L^*\circ d_L^{(2)}+d_A^{(1)}\ten\id\circ\pi_L^*\circ d_L^{(1)},
\end{align*}
where $\pi_L^*:L\to A'\ten L$ denotes the admissible coaction of
$A$ on $L$.
It is then easy to see, by induction on $n$, that
\begin{equation}\label{eq:admind}
d_B^{(n)}\circ\Pi=\sum_{k=0}^{n-1} \gamma^{(n-k)}_{n}\circ d_L^{(k)},
\end{equation}
where $\gamma_{n}^{(n-k)}:L^{\ten k}\to(L\ten A)^{\ten k-1}\ten L$ 
is divisible by $\hbar^{n-k}$.
Namely, we have
\begin{align*}
d_B^{(n)}\circ\Pi &= p_B\ten(d_B^{(n-1)}\circ\Pi)\circ\Delta_B \\
& = \sum_{k=0}^{n-1} p_B\ten\id^{\ten n-1}\circ \id\ten\gamma^{(n-k-1)}_{n-1}\circ \id\ten 
d_L^{(k)} \circ\id\ten\pi_L^*\circ\Delta_L \\
& = \sum_{k=0}^{n-1} \id\ten\gamma^{(n-k-1)}_{n-1}\circ\id\ten\pi_{L^{\ten n-1}}^*\circ d_L^{(k+1)} +
d_A^{(1)}\ten\gamma^{(n-k-1)}_{n-1}\circ\pi_{L^{\ten n-1}}^*\circ d_L^{(k)} .
\end{align*}
Since the coaction of $A$ on $L$ is admissible, one can proceed by setting for every $k=0,1,\dots, n-1$
\[
\gamma_{n}^{(n-k)}=\id\ten\gamma^{(n-k)}_{n-1}\circ\id\ten\pi_{L^{\ten n-1}}^*
+d_A^{(1)}\ten\gamma^{(n-k-1)}_{n-1}\circ\pi_{L^{\ten n-1}}^*.
\]
Finally, relation \eqref{eq:admind} implies $L'\subset L\cap B'$ and proves the result.
\end{pf}

\subsection{The quantum relative Verma modules}\label{ss:relquantumverma1}

The relative Verma modules $\Lm$, $\Npv$, constructed in Section \S\ref{s:Gamma}
for any split pair of Lie bialgebras, have natural quantum counterparts, given by the 
relative analogues of the quantum Verma modules $\MmB$, $\MpBp$ with
respect to a split pair of Hopf algebras. In the following, we focus on the case of
as split pair of QUEs. 
 
We have seen in \S\ref{ss:quantum-verma} that the \DYt structures of 
$\MmB$ and $\MpBp$ are tailored around the universal properties \eqref{eq:univ-prop-1}, 
\eqref{eq:univ-prop-2}. The same principle applies to the construction of the 
quantum relative Verma modules.

Let $A\stackrel{i_{\hbar}}{\to} B\stackrel{p_{\hbar}}{\to}A$, $p_{\hbar}\circ i_{\hbar}=\id_A$,
be a split pair of QUEs, and $\a\stackrel{i}{\to} \b\stackrel{p}{\to}\a$ the corresponding split
pair of Lie bialgebras. We denote by $\gLm{B,A}$ the \DYt $B$--module $L$ constructed in
\S\ref{ss:splitHA} on the image of the idempotent \eqref{eq:RadPi}
\begin{equation*}
\Pi=m_B\circ\id\ten(S_B\circ \pi_{\hbar})\circ\Delta_{B}.
\end{equation*}
It follows from \S\ref{ss:splitHA} that $\gLm{B,A}\subset B$ is a subalgebra and $B\simeq\Rad{\gLm{B,A}}{A}$.
In particular, for $A=\sfK$, it is easy to check that the modules $\gLm{B,\sfK}$ and $\MmB$ coincide.

The construction of $\gNm{B,A}$ is similar to that of $\MpBp$ described in \S\ref{ss:Mpv}.
As a \DYt $B$--modules it is realised on $B'\ten A^*$ with coaction
$\pi^*_{\gNm{B,A}}=\Delta_{B}^{\operatorname{21}}\ten\id$
and action
\begin{equation}\label{eq:Npvh-act}
\pi_{\gNm{B,A}}=
m_{B}^{(3)}\ten\coadd_{A}\circ\id^{\ten 2}\ten S_{B}^{-1}\ten\id^{\ten 2}\circ(24)\circ\id\ten p_{\hbar}\ten\id^{\ten 3}\circ\Delta_{B}^{(3)}\ten\id^{\ten 2},
\end{equation}
where $\coadd_{A}$ denotes the coadjoint action of $A$ on $A^*$.
The structure of \DYt $A\op$--module is given 
by the same formulae applied on the right to $A'\ten A^*\subset B'\ten A^*$.
For $A=\sfK$, the modules $\MpBp$ and $\gNm{B,\sfK}$ coincide.

\subsection{Properties of relative Verma modules}\label{ss:relquantumverma1-bis}

\begin{proposition}\label{le:q rel verma}\hfill
\begin{enumerate}
\item $\gLm{B,A}, \gNm{B,A}$ are admissible \DYt $B$--modules and
satisfy the universal properties
\begin{eqnarray}\label{eq:L-univ-p}
\Hom_{B}^{B}(\gLm{B,A}, \V)&\simeq&\Hom_{A}^{B}(K,\V),\\
\label{eq:N-univ-p}
\Hom_{B}^{B}(\V,\gNm{B,A})&\simeq&\Hom_{{\gLm{B,A}}}(\V, K),
\end{eqnarray}
for any admissible $\V\in\aDrY{B}$. Moreover, $\gLm{B,A}, \gNm{B,A}$ have, respectively, natural
structures of coalgebra and algebra objects in $\aDrY{B}$.
\item $\gNm{B,A}$ satisfies the universal 
property (equivalent to \eqref{eq:N-univ-p})
\begin{equation}\label{eq:N-univ-prop-int}
\Hom_{B}^{B}(\V,\gNm{B,A})\simeq\Hom^{B}_{A}(\V, p_{\hbar}^*\gNm{A,A}),
\end{equation}
where $B$ acts on $\gNm{A,A}$ by projection to $A$.
\item As \DYt $(\b,\a)$--bimodules, 
the semiclassical limit of $\gLm{B,A}, \gNm{B,A}$ are, respectively, $\Lm$ and $\Npv$.
\end{enumerate}
\end{proposition}

\begin{pf}
(i)
The universal properties \eqref{eq:L-univ-p}, \eqref{eq:N-univ-p}
generalise those of $\MmB$ and $\MpBp$ described in \S\ref{ss:quantum-verma}.
Specifically, we observe that the action of $A$ and the coaction of $B$ on the 
element $1\in{\gLm{B,A}}$ are trivial. In particular, by restriction to $1\in{\gLm{B,A}}$, we 
get a map
\[
\Hom_{B}^{B}({\gLm{B,A}}, \V)\to\Hom_{A}^{B}(K,\V),
\]
whose inverse is given by composition with the action of $B$ restricted
to $\gLm{B,A}$. This proves \eqref{eq:L-univ-p} and implies that 
$\gLm{B,A}$ is endowed with a coalgebra structure in $\DrY{B}$. 

The proof of \eqref{eq:N-univ-p} goes along the same lines.
There is a map
\[
\Hom_{B}^{B}(\V,\gNm{B,A})\to\Hom_{{\gLm{B,A}}}(\V, K),
\]
obtained by composition with the counits of $B'$ and $A^*$,
whose inverse is given by
\[
f\mapsto \id\ten\id\ten f \circ \pi_V^*\ten\id\ten\pi_V\circ R_{A}\ten\id,
\]
where $R_{A}$ denotes the universal $R$--matrix of $A$.
It follows in particular that $\gNm{B,A}$ is naturally endowed with an algebra structure
in $\DrY{B}$.

(ii) The isomorphism \eqref{eq:N-univ-prop-int} is proved similarly. The map
\[
\Hom_{B}^{B}(\V,\gNm{B,A})\to\Hom^{A}_{B}(\V, p_{\hbar}^*\gNm{A,A})
\]
is obtained by composition with the canonical projection $\gNm{B,A}\to\gNm{A,A}$ induced
by $B'\to A'$. The inverse map is obtained by precomposing with the coaction
of $\V$ projected onto $\gLm{B,A}'$. It is easy to check from Lemma \ref{ss:qrad-DY} that this is 
well--defined and gives the desired isomorphism.

(iii) The idempotent $\Pi$ reduces modulo $\hbar$ to the canonical projection $U\b\to U\m$, where
$\m=\ker(p)$.
It follows that $\SC{\gLm{B,A}}\simeq U\m$. One checks that, under the canonical identification
$\Lm\simeq U\m\simeq\SC{\gLm{B,A}}$, the $B$--action on $\gLm{B,A}$ reduces to the standard 
$\b$--action on $\Lm$. 
Since the coaction on $\Lm$ is uniquely determined by the condition $\pi_{\Lm}^*(u_L)=0$,  
one concludes that 
$\SC{\pi^*_{\gLm{B,A}}}=\pi_{\Lm}^*$, and therefore $\SC{\gLm{B,A}}\simeq\Lm$ in $\DrY{\b}$.

The proof for $\gNm{B,A}$ goes as follows. It is clear that $\SC{\gNm{B,A}}\simeq\Npv$ as $\b$--comodules.
One then observes that the maps
\[
\xymatrix{
\Hom_{\b}^{\b}(V,\SC{\gNm{B,A}})
\ar@<2pt>[r]^(.6){\psi}
&\Hom_{{\m}}(V, \sfk)
\ar@<2pt>[l]^(.4){\phi}
},
\]
defined by $\psi=\SC{\varepsilon_{\gNm{B,A}}}\circ -$ and
\[
\phi=\id\ten - \circ \sum_{n,m\geq 0}\Sym_n\ten\Sym_m\ten\id\circ\id^{\ten m}\ten\pi_V^{(m)}\circ\id^{\ten n}\ten r_{\a}^{(m)}\circ(\pi_V^*)^{(n)},
\]
are well--defined and inverse of each other.
It follows that $\SC{\gNm{B,A}}$ satisfies the universal property of $\Npv$, and it 
is therefore isomorphic to $\Npv$ in $\DrY{\b}$. Similarly
for the structure of \DYt $\a^{op}$--module.
\end{pf}

\subsection{Quantum restriction functor}\label{ss:q restriction}

Let $(B,A)$ be a split pair of QUEs.
The relative quantum Verma modules $\gLm{B,A}$, $\gNm{B,A}$
allow to define the functor
\[\FF{A}{B}:\aDrY{B}\to\aDrY{A}
\qquad\FF{A}{B}(\V)=\Hom_{B}^{B}(\gLm{B,A},\gNm{B,A}\ten\V).\]
\begin{proposition}\label{prop:hGamma-res}
The functor $\FF{A}{B}$ is tensor isomorphic to the restriction
functor $\Res_{A,B}:\aDrY{B}\to\aDrY{A}$.
\end{proposition}
\begin{pf}
For any $\V\in\aDrY{B}$, the description of the isomorphism $\FF{A}{B}(\V)\simeq\V$ 
as admissible \DYt $A$--module is identical to that of Proposition 
\ref{pr:isom to forget}, and relies on the admissibility of $\gLm{B,A}$
and $\gNm{B,A}$ and their universal properties in $\aDrY{A}$. 

It is immediate to verify that the tensor structure 
$J_{A,B}$ on $\FF{A}{B}$ is trivial, and
therefore $\FF{A}{B}$ is isomorphic to $\Res_{A,B}$ as
tensor functors. 
Namely, for every $v\in\FF{A}{B}(\V)$, 
$w\in\FF{A}{B}(\mathcal{W})$, one has
\begin{align*}
(1_+\ten\id\ten\id)&J_{A,B}(v\ten w)(1_-)\\
=&(1_+\ten\id\ten\id)(i^{\hbar}_+\ten\id\ten\id)\circ R_{23}^{-1}(2\,3)\circ v(1_-)\ten w(1_-)\\
=&(1_+\ten\id\ten\id)(i^{\hbar}_+\ten\id\ten\id)\circ(2\,3)\circ v(1_-)\ten w(1_-)\\
=&(1_+\ten\id)v(1_-)\ten(1_+\ten\id)w(1_-),
\end{align*}
where the second identification follows from the fact that the coaction on 
$v(1_-)$, $w(1_-)$ is trivial.
\end{pf}

\subsection{Quantisation of $\Lm$ and $\Npv$}\label{ss:relquantumverma2}
The following is a relative analogue of Proposition \ref{pr:quantumverma}

\begin{theorem}\label{thm:quantization-LN} The following holds in the
category of \DYt $(\EK{\b}, \EK{\gdm})$--modules
\begin{enumerate}
\item[(a)] $\noEKeq{\b}\circ\noEKeq{\a}(\Lm)\simeq\Lmh$ as coalgebras,
\item[(b)] $\noEKeq{\b}\circ\noEKeq{\a}(\Npv)\simeq{\Npvh}$ as algebras.
\end{enumerate}
\end{theorem}

The proof amounts to constructing the intertwiners
corresponding to the universal properties of $\Lmh, \Npvh$.
A direct construction along the lines of the proof of Proposition 
\ref{pr:quantumverma} is not straightforward, however.
We prove this theorem in Section \S\ref{se:rel prop} by 
describing the modules $\Lm, \Npv$, and their quantisation, 
in the framework of $\PROP$s. These descriptions show that 
the quantisation of the classical intertwiners of $\Lm$ and $\Npv$ satisfy 
the required properties and yield canonical identifications
\[\noEKeq{\b}\noEKeq{\a}(\Lm)\simeq\Lmh,\qquad \noEKeq{\b}\noEKeq{\a}(\Npv)\simeq{\Npvh}.\]

\subsection{The natural transformation $v\resped$}\label{ss:nat-trans}

\begin{theorem}\label{thm:ek-fac}
Let $\a\hookrightarrow\b$ be a split pair of Lie bialgebras, $\gd\hookrightarrow\gu$
the corresponding inclusion of Manin triples, and
\[(\Res_{\a, \b}, J\resped):\hDrY{\b}{\Phi}\to\hDrY{\a}{\Phi}\]
the tensor restriction functor constructed in Section \S\ref{s:Gamma}. Then,
there exists a natural isomorphism $v\resped$ such that the following is a
commutative diagram
\[\xymatrix@C=2cm{
\hDrY{\b}{\Phi} \ar[r]^{\noEKeq\b} \ar[d]_{(\Res_{\a,\b}, J\resped)}& 
\aDrY{\Uhb} \ar[d]^{(\Res_{\Uha,\Uhb}, \id)}
\ar@{<=}[dl]_{\wt{v}\resped}
\\
\hDrY{\a}{\Phi} \ar[r]_{\noEKeq\a} & 
\aDrY{\Uha}
}\]
where $\Res_{\Uhb,\Uha}$ is the functor induced by the split embedding
$\Uha\hookrightarrow\Uhb$.
\end{theorem}

\begin{pf}
By construction of the tensor structure $J\resped$ and Proposition
\ref{prop:hGamma-res}, it is equivalent to prove the commutativity
of the above diagram when $\Res_{\a,\b}$ and $\Res_{\Uha,\Uhb}$
are replaced by the tensor restriction functors $\FF{\a}{\b}$ and
$\FF{\a}{\b}^\hbar:=F_{\EK\gdm,\EK\b}$, respectively. Let now $V\in\hDrY{\b}{\Phi}$,
then
\[\begin{split}
\FF{\a}{\b}^\hbar\circ\noEKeq{\b}(V)
&=
\Hom_{\hDrY{\Uhb}{}}(\Lmh,\Npvh\otimes\noEKeq{\b}(V))\\
&\cong
\Hom_{\hDrY{\Uhb}{}}(\noEKeq{\b}\circ\noEKeq{\a}(\Lm),\noEKeq{\b}\circ\noEKeq{\a}(\Npv)\otimes\noEKeq{\b}(V))\\
&\cong
\Hom_{\hDrY{\b}{\Phi}}(\noEKeq{\a}(\Lm),\noEKeq{\a}(\Npv)\otimes V)\\
&=
\noEKeq{\a}\left(\Hom_{\hDrY{\b}{\Phi}}(\Lm,\Npv\otimes V)\right)\\
&=
\noEKeq{\a}\circ\FF{\a}{\b}(V),
\end{split}\]
where the first isomorphism follows by Theorem \ref{thm:quantization-LN},
the second one from the fact that $\noEKeq{\b}$ is a tensor equivalence,
the subsequent equality from the fact that $\Lm$ is a trivial \DYt module
over $\a$, and the final one by definition of $\FF{\a}{\b}$. This gives rise
to an isomorphism $\wt{v}_{\a,\b}:\FF{\a}{\b}^\hbar\circ\noEKeq{\b}\cong
\noEKeq{\a}\circ\FF{\a}{\b}$ which is readily seen to preserve the tensor
structures.
\end{pf}


\section{Quantisation of split pairs of Lie bialgebras}\label{s:gammahopf}


\newcommand{\gums}{\b}
\newcommand{\gdms}{\a}
In this section, we provide a complementary interpretation of the 
natural isomorphism constructed in the previous section, which relies 
on a generalisation of the results described in \S\ref{ss:eksum}. 

Specifically, we show that the relative twist on the functor $\FF{\a}{\b}$
gives rise to a Hopf algebra $\HL$
in the braided tensor category $\aDrY{\EK{\gdms}}$.
This allows to construct, through the Radford biproduct of $\HL$ 
and $\EK{\gdms}$, an alternative quantisation of the split 
pair $(\gums,\gdms)$.
The natural isomorphism $v\resped$ provides an isomorphism between
this quantisation and $\EK{\gums}$, which restricts to the  identity
on $\EK\gdms$.

\subsection{Lifting the tensor functor $\FF{\a}{\b}$}\label{ss:gammalift}
The tensor structure $J\resped$ on the functor $\FF{\a}{\b}$ induces on the
algebra $\sfEnd{\FF{\a}{\b}}$ a bialgebra structure with coproduct
defined for every $\phi\in\sfEnd{\FF{\a}{\b}}$ by the relation
\[
\xymatrix{
\FF{\a}{\b}(V)\ten\FF{\a}{\b}(W)\ar[d]_{J\resped}\ar[r]^{\Delta(\phi)_{V,W}} & 
\FF{\a}{\b}(V)\ten\FF{\a}{\b}(W)\ar[d]^{J\resped}\\
\FF{\a}{\b}(V\ten W)\ar[r]^{\phi_{V\ten W}} & \FF{\a}{\b}(V\ten W)
}
\]
for every $\phi\in\sfEnd{\FF{\a}{\b}}$ and $V,W\in\hDrY{\gums}{\Phi}$, 
and counit given by evaluation on the trivial module $V=K\in\hDrY{\gums}{\Phi}$.
The object $\FF{\a}{\b}(\Lm)$ naturally embeds in $\sfEnd{\FF{\a}{\b}}$
and inherits a (non--topological) Hopf algebra structure. It is essential 
to observe that these structures are constructed in $\hDrY{\gdms}{\Phi}$, 
and therefore the notion of Hopf algebra is adapted to the associativity and 
commutativity constraints in $\hDrY{\gdms}{\Phi}$. 
More specifically, we prove the following

\begin{theorem}\label{thm:braidhopf}\hfill
\begin{itemize}
\item[(i)] The object $\FF{\a}{\b}(\Lm)$  
is a Hopf algebra in the category $\hDrY{\gdms}{\Phi}$.
\item[(ii)] The tensor functor $(\FF{\a}{\b}, J\resped)$ naturally lifts
to a braided tensor functor $\wtFF{\a}{\b}$ from $\hDrY{\gums}{\Phi}$
to the category of \DYt $\FF{\a}{\b}(\Lm)$--modules
in $\hDrY{\gdms}{\Phi}$.
\item[(iii)] For $\a=0$, the Hopf algebra $\FF{\a}{\b}(\Lm)$ is the 
Etingof--Kazhdan quantisation of $\b$ and $\wtFF{\a}{\b}$ is 
the braided tensor equivalence \ref{ss:qRep}. For $\a=\b$, 
$\wtFF{\a}{\b}$ is the identity functor on $\hDrY{\gums}{\Phi}$.
\end{itemize}
\end{theorem}

The proof is carried out in \S\ref{ss:algstructure}--\S\ref{ss:compcoaction}.

\subsubsection{Algebra structure on $\FF{\a}{\b}(\Lm)$}\label{ss:algstructure}
For every $V\in\hDrY{\gums}{\Phi}$, we define a map $\mu_V: 
\FF{\a}{\b}(\Lm)\ten\FF{\a}{\b}(V)\to\FF{\a}{\b}(V)$  
\[\mu_V(x\ten v)=(\ipd\ten\id){\PhiU}^{-1}(\id\ten v)x.\]
Let $u\in\FF{\a}{\b}(\Lm)$ be the element satisfying $u(1_-)=\varepsilon\ten 1_-$. 
\begin{proposition}\label{prop:prod}
The map $\mu_V$ is a morphism in $\hDrY{\gdms}{\Phi}$, it is natural in $V$, 
and it satisfies
\[
\mu_V(\id\ten\mu_V)\PhiD=\mu_V(\mu_{\Lm}\ten\id)
\quad\mbox{and}\quad
\mu_V\circ(u\ten\id)=\id_V.
\]
Therefore, $(\FF{\a}{\b}(\Lm),\mu,u)$ is an associative algebra object in 
$\hDrY{\gdms}{\Phi}$ acting on the functor $\FF{\a}{\b}$.
\end{proposition}

\begin{pf}
The fact that the map $\mu_V$ is a morphism in $\hDrY{\gdms}{\Phi}$, 
its naturality in $V$, and its compatibility with $u\in\FF{\a}{\b}(\Lm)$ are 
straightforward.

Let now $x,y\in\FF{\a}{\b}(\Lm)$, $v\in\FF{\a}{\b}(V)$
and $\PhiD=\sum P_{\sfK}\ten Q_{\sfK}\ten R_{\sfK}$. Then
\begin{align*}
\mu_V&(\id\ten\mu_V)\PhiD(x\ten y\ten v)\\
&=(\ipd\ten\id){\PhiU}^{-1}(\id\ten\ipd\ten\id)(\id\ten{\PhiU}^{-1})
(\id\ten\id\ten R_{\sfK}.v)(\id\ten Q_{\sfK}.y)(P_{\sfK}.x)\\
&=(\ipd\ten\id)(\id\ten\ipd\ten\id){\PhiU}^{-1}_{1,23,4}{\PhiU}^{-1}_{234}
{\PhiD}_{123}^{\rho}(\id\ten\id\ten v)(\id\ten y)x\\
&=(\ipd\ten\id)(\id\ten\ipd\ten\id){\PhiU}_{123}{\PhiD}_{123}^{\rho}
{\PhiU}^{-1}_{12,3,4}{\PhiU}^{-1}_{1,2,34}(\id\ten\id\ten v)(\id\ten y)x\\
&=(\ipd\ten\id)(\ipd\ten\id\ten\id){\PhiU}^{-1}_{12,3,4}{\PhiU}^{-1}_{1,2,34}
(\id\ten\id\ten v)(\id\ten y)x\\
&=(\ipd\ten\id){\PhiU}^{-1}(\id\ten\id\ten v)(\ipd\ten\id){\PhiU}^{-1}(\id\ten y)x\\
&=\mu_V(\mu_{\Lm}\ten\id)(x\ten y\ten v)
\end{align*}
by applications of the pentagon axiom and the associativity of $\ipd$.
\end{pf}

\subsubsection{Coalgebra structure on $\FF{\a}{\b}(\Lm)$}
\newcommand{\DeltaL}{{\Delta_L}}
The coproduct $\im$ on $\Lm$ induces a coproduct on $\FF{\a}{\b}(\Lm)$ by
the formula
\[\DeltaL: \FF{\a}{\b}(\Lm)\to\FF{\a}{\b}(\Lm)\ten\FF{\a}{\b}(\Lm)\qquad
\DeltaL=J\resped^{-1}\circ\FF{\a}{\b}(\im).\]
\begin{proposition}\label{prop:coprod}
The map $\DeltaL$ is a morphism in $\hDrY{\gdms}{\Phi}$ and it satisfies
\[\PhiD(\DeltaL\ten\id)\DeltaL=(\id\ten\DeltaL)\DeltaL.\]
\end{proposition}

\begin{pf}
The maps $J\resped$ and $\FF{\a}{\b}(\im)$ 
are both morphisms $\hDrY{\gdms}{\Phi}$, therefore $\DeltaL$ is. 
Moreover, 
\begin{align*}
\PhiD(\DeltaL\ten\id)\DeltaL
&=\PhiD J_{\FF{\a}{\b}, 1,2}^{-1}(\FF{\a}{\b}(\im)\ten\id)J\resped^{-1}\FF{\a}{\b}(\im)\\
&=\PhiD J_{\FF{\a}{\b},1,2}^{-1}J_{\FF{\a}{\b}, 12,3}^{-1}\FF{\a}{\b}(\im\ten\id)\FF{\a}{\b}(\im)\\
&=J_{\FF{\a}{\b},2,3}^{-1}J_{\FF{\a}{\b}, 1,23}^{-1}\FF{\a}{\b}({\PhiU})\FF{\a}{\b}(\im\ten\id)\FF{\a}{\b}(\im)\\
&=J_{\FF{\a}{\b},2,3}^{-1}(\id\ten\FF{\a}{\b}(\im))J\resped^{-1}\FF{\a}{\b}(\im)\\
&=(\id\ten\DeltaL)\DeltaL.
\end{align*}
\end{pf}
For simplicity, from now on, we omit the action of the associator $\PhiD$ 
since we proved that it is natural with respect to $\mu$, $\DeltaL$ and $J\resped$.

\subsubsection{Relation with $\sfEnd{\FF{\a}{\b}}$}
It follows from Proposition \ref{ss:algstructure} that the collection of morphisms
$\mu_V$, $V\in\hDrY{\gums}{\Phi}$, defines a morphism of algebras from 
$\FF{\a}{\b}(\Lm)$ to $\sfEnd{\FF{\a}{\b}}$, whose injectivity follows 
immediately from the action of $\FF{\a}{\b}(\Lm)$ on itself.
We now show that this embedding is in fact a morphism of coalgebras.
To this extent, we denote by $\Delta$ the coproduct on $\sfEnd{\FF{\a}{\b}}$, 
by $\varphi^{(1)}$ the embedding defined in \S\ref{ss:algstructure}, and 
by $\varphi^{(2)}$ the embedding from $\FF{\a}{\b}(\Lm)^{\ten 2}$ into $\sfEnd{\FF{\a}{\b}^2}$
defined for every $V,W\in\hDrY{\gdms}{\Phi}$ by
\[
\varphi^{(2)}_{V,W}=(\mu_V\ten\mu_W)\circ\beta_{\a,23}
\]
as a map from $\FF{\a}{\b}(\Lm)\ten\FF{\a}{\b}(\Lm)\ten\FF{\a}{\b}(V)\ten\F\noEKff{\b}(W)$ to $\FF{\a}{\b}(V)\ten\FF{\a}{\b}(W)$.
\begin{proposition}\label{prop:tenprodgamma}
One has
\begin{equation}\label{eq:end-f}
\Delta\circ\varphi^{(1)}=\varphi^{(2)}\circ\DeltaL.
\end{equation}
In particular, the counit of $\sfEnd{\FF{\a}{\b}}$, \ie the evaluation
on the identity object $V=K\in\hDrY{\gums}{\Phi}$, restricts to a counit 
$\varepsilon=\mu_K$ on $\FF{\a}{\b}(\Lm)$ and the object 
$(\FF{\a}{\b}(\Lm),\mu_L,u,\DeltaL,\varepsilon)$ is a bialgebra object in 
$\hDrY{\gdms}{\Phi}$.
\end{proposition}
\begin{pf}
From the definition of $\Delta$, it follows that \eqref{eq:end-f} is equivalent 
to
\[\mu_{V\ten W}(\id\ten J\resped)=J\resped(\mu_V\ten\mu_W)\beta_{\a,23}(\DeltaL\ten\id\ten\id)\]
as morphisms from $\FF{\a}{\b}(\Lm)\ten\FF{\a}{\b}(V)\ten\FF{\a}{\b}(W)$ to $\FF{\a}{\b}(V\ten W)$,
where $V,W\in\hDrY{\gums}{\Phi}$.
In $\FF{\a}{\b}(V\ten W)$, one has to show
\begin{align*}
(\ipd\ten\id\ten\id)&(\id\ten\ipd\ten\id\ten\id)\beta_{34}(\id\ten v\ten w)(\id\ten\im)x\\
&=(\ipd\ten\id\ten\id)\beta_{23}(\ipd\ten\id\ten\ipd\ten\id)(\id\ten b_i.v \ten \id\ten w)
(x_1\ten a_i.x_2)\im,
\end{align*}
where $v\in\FF{\a}{\b}(V), w\in\FF{\a}{\b}(W), x\in\FF{\a}{\b}(\Lm)$, and, in sumless Sweedler notation,
$\DeltaL(x)=x_1\ten x_2$, $\exp(-\hbar\OmegaD/2)=a_i\ten b_i$. 

It is more convenient, in this case, to give a pictorial proof, relying on the 
graphical calculus for Lie bialgebras (e.g. \cite[Lecture 19]{es}).
We read the diagrams from top to bottom and from left to right.

We have to show that the following diagrams are equivalent:
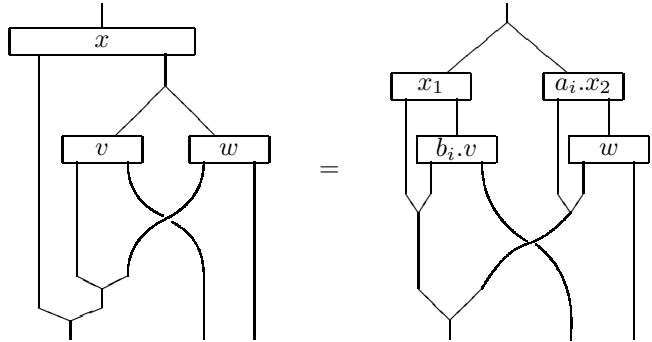
\[
\xy
0;/r.20pc/:
(0,0)*{
\xy
(0,-5)*{}="A";
(0,-10)*{}="B";
(-10,-20)*=<30pt,10pt>[F]{v}="C"; 
(10,-20)*=<30pt,10pt>[F]{w}="D";
"A";"B" **\dir{-};
"C";"B" **\dir{-};
"D";"B" **\dir{-};
(-14,-22)*{};(-14,-40)**\dir{-};
(14,-22)*{};(14,-40)**\dir{-};
(-6,-22)*{};(-0.5,-30.5)**\crv{(-6,-28)};
(6,-22)*{};(0,-31)**\crv{(6,-28)};
(-6,-40)*{};(0,-31)**\crv{(-6,-34)};
(6,-40)*{};(1,-31.5)**\crv{(6,-34)};
(-14,-40)*{}; (-10,-42)**\dir{-}; 
(-6,-40)*{}; (-10,-42)**\dir{-};
(-10,-42)*{};(-10,-45)**\dir{-};
(6,-40)*{};(6,-45)**\dir{-};
(14,-40)*{};(14,-45)**\dir{-};
(-10,-3)*=<70pt,10pt>[F]{x}="X";
(-10,3)*{};"X" **\dir{-};
(-20,-5)*{};(-20,-45)**\dir{-};
(-10,-45);(-15,-47)**\dir{-};
(-20,-45);(-15,-47)**\dir{-};
(-15,-47);(-15,-50)**\dir{-};
(6,-45);(6,-50)**\dir{-};
(14,-45);(14,-50)**\dir{-};
\endxy
}="A";
(60,0)*{
\xy
(-10,3);(-10,0)*{}="O" **\dir{-};
(-22,-10)*=<30pt,10pt>[F]{x_1}="X1";
(2,-10)*=<30pt,10pt>[F]{a_i.x_2}="X2";
"O";"X1"**\dir{-};
"O";"X2"**\dir{-};
(-2,-12);(-2,-27)**\dir{-};
(-26,-12);(-26,-27)**\dir{-};
(-18,-20)*=<30pt,10pt>[F]{b_i.v}="V";
(6,-20)*=<30pt,10pt>[F]{w}="W";
(-18,-12);"V"**\dir{-};
(6,-12);"W"**\dir{-};
(10,-22);(10,-50)**\dir{-};
(2,-22);(2,-27)**\dir{-};
(-22,-22);(-22,-27)**\dir{-};
(-26,-27);(-24,-30)**\dir{-};
(-22,-27);(-24,-30)**\dir{-};
(-24,-30);(-24,-42)**\dir{-};
(-2,-27);(0,-30)**\dir{-};
(2,-27);(0,-30)**\dir{-};
(-14,-22)*{};(-7,-34.5)**\crv{(-14,-30)};
(0,-30)*{};(-7,-35)**\crv{(-3.5,-34)};
(-14,-42)*{};(-7,-35)**\crv{(-10.5,-36)};
(0,-50)*{};(-5.5,-35)**\crv{(1,-40)};
(-24,-42);(-19,-47)**\dir{-};
(-14,-42);(-19,-47)**\dir{-};
(-19,-47);(-19,-50)**\dir{-};
\endxy
}="B";
(30,0)*{=};
\endxy
\]
The coproduct $\DeltaL$ is defined by the equation
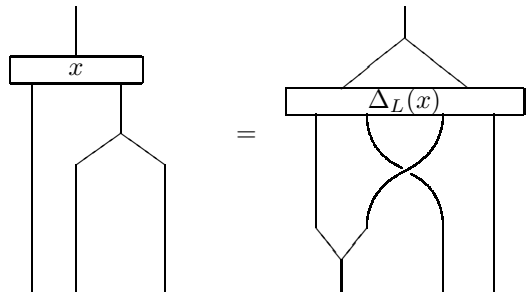
\[
\xy
0;/r.20pc/:
(30,0)*{
\xy
(0,-5)*{}="A";
(0,-10)*{}="B";
(-10,-18)*{}="C";
(0,-20)*=<90pt,10pt>[F]{\DeltaL(x)}; 
(10,-18)*{}="D";
"A";"B" **\dir{-};
"C";"B" **\dir{-};
"D";"B" **\dir{-};
(-14,-22)*{};(-14,-40)**\dir{-};
(14,-22)*{};(14,-40)**\dir{-};
(-6,-22)*{};(-0.5,-30.5)**\crv{(-6,-28)};
(6,-22)*{};(0,-31)**\crv{(6,-28)};
(-6,-40)*{};(0,-31)**\crv{(-6,-34)};
(6,-40)*{};(1,-31.5)**\crv{(6,-34)};
(-14,-40)*{}; (-10,-45)**\dir{-}; 
(-6,-40)*{}; (-10,-45)**\dir{-};
(-10,-45)*{};(-10,-50)**\dir{-};
(6,-40)*{};(6,-50)**\dir{-};
(14,-40)*{};(14,-50)**\dir{-};
\endxy
};
(-20,0)*{
\xy
(0,-15)*=<50pt,10pt>[F]{x}="X";
(0,-5);"X" **\dir{-};
(-7,-17);(-7,-50)**\dir{-};
(7,-17);(7,-25)**\dir{-};
(7,-25);(0,-30)**\dir{-};
(7,-25);(14,-30)**\dir{-};
(0,-30);(0,-50)**\dir{-};
(14,-30);(14,-50)**\dir{-};
\endxy
};
(5,0)*{=};
\endxy
\]
Since the action of $\gd$ on the objects $\FF{\a}{\b}(V)$ is given by right action on $\Npv$,
we represent the braiding $\beta_{\a}^{-1}$ as
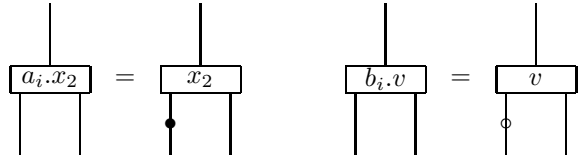
\[
\xy
(0,0)*{
\xy
(0,-10)*=<30pt,10pt>[F]{a_i.x_2}="X";
(0,0);"X"**\dir{-};
(-4,-12);(-4,-20)**\dir{-};
(4,-12);(4,-20)**\dir{-};
\endxy
};
(10,0)*{=};
(20,0)*{
\xy
(0,-10)*=<30pt,10pt>[F]{x_2}="X";
(0,0);"X"**\dir{-};
(-4,-12);(-4,-20)**\dir{-};
(4,-12);(4,-20)**\dir{-};
(-4,-16)*{\bullet};
\endxy
};
\endxy
\qquad\qquad
\xy
(0,0)*{
\xy
(0,-10)*=<30pt,10pt>[F]{b_i.v}="X";
(0,0);"X"**\dir{-};
(-4,-12);(-4,-20)**\dir{-};
(4,-12);(4,-20)**\dir{-};
\endxy
};
(20,0)*{
\xy
(0,-10)*=<30pt,10pt>[F]{v}="X";
(0,0);"X"**\dir{-};
(-4,-12);(-4,-20)**\dir{-};
(4,-12);(4,-20)**\dir{-};
(-4,-16)*{\circ};
\endxy
};
(10,0)*{=};
\endxy
\]
We are allowed to move the black and the white bullets along the lines, since they commute
with the left action of $\gu$. The RHS corresponds to\\
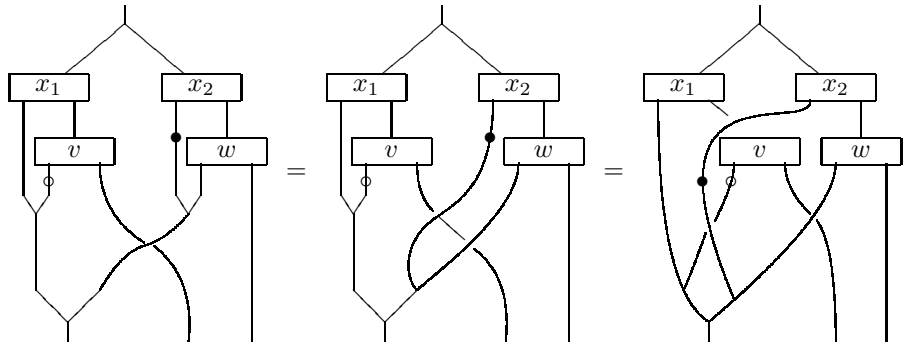
\[
\xy
0;/r.20pc/:
(0,0)*{
\xy
(-10,3);(-10,0)*{}="O" **\dir{-};
(-22,-10)*=<30pt,10pt>[F]{x_1}="X1";
(2,-10)*=<30pt,10pt>[F]{x_2}="X2";
"O";"X1"**\dir{-};
"O";"X2"**\dir{-};
(-2,-12);(-2,-27)**\dir{-};
(-2,-18)*{\bullet};
(-26,-12);(-26,-27)**\dir{-};
(-18,-20)*=<30pt,10pt>[F]{v}="V";
(6,-20)*=<30pt,10pt>[F]{w}="W";
(-18,-12);"V"**\dir{-};
(6,-12);"W"**\dir{-};
(10,-22);(10,-50)**\dir{-};
(2,-22);(2,-27)**\dir{-};
(-22,-22);(-22,-27)**\dir{-};
(-22,-25)*{\circ};
(-26,-27);(-24,-30)**\dir{-};
(-22,-27);(-24,-30)**\dir{-};
(-24,-30);(-24,-42)**\dir{-};
(-2,-27);(0,-30)**\dir{-};
(2,-27);(0,-30)**\dir{-};
(-14,-22)*{};(-7,-34.5)**\crv{(-14,-30)};
(0,-30)*{};(-7,-35)**\crv{(-3.5,-34)};
(-14,-42)*{};(-7,-35)**\crv{(-10.5,-36)};
(0,-50)*{};(-5.5,-35)**\crv{(1,-40)};
(-24,-42);(-19,-47)**\dir{-};
(-14,-42);(-19,-47)**\dir{-};
(-19,-47);(-19,-50)**\dir{-};
\endxy
};
(50,0)*{
\xy
(-10,3);(-10,0)*{}="O" **\dir{-};
(-22,-10)*=<30pt,10pt>[F]{x_1}="X1";
(2,-10)*=<30pt,10pt>[F]{x_2}="X2";
"O";"X1"**\dir{-};
"O";"X2"**\dir{-};
(-2.5,-18)*{\bullet};
(-26,-12);(-26,-27)**\dir{-};
(-18,-20)*=<30pt,10pt>[F]{v}="V";
(6,-20)*=<30pt,10pt>[F]{w}="W";
(-18,-12);"V"**\dir{-};
(6,-12);"W"**\dir{-};
(10,-22);(10,-50)**\dir{-};
(2,-22);(-14,-42)**\crv{(2,-30)&(-12,-40)};
(-22,-22);(-22,-27)**\dir{-};
(-22,-25)*{\circ};
(-26,-27);(-24,-30)**\dir{-};
(-22,-27);(-24,-30)**\dir{-};
(-24,-30);(-24,-42)**\dir{-};
(-14,-22);(-11.5,-29.5)**\crv{(-14,-27)};
(-10.5,-30.5);(-6.5,-34)**\crv{};
(0,-50)*{};(-5,-35.5)**\crv{(1,-40)};
(-24,-42);(-19,-47)**\dir{-};
(-14,-42);(-19,-47)**\dir{-};
(-19,-47);(-19,-50)**\dir{-};
(-2,-12);(-14,-42)**\crv{(-2,-27)&(-16,-30)&(-16,-40)};
\endxy
};
(100,0)*{
\xy
(-10,3);(-10,0)*{}="O" **\dir{-};
(-22,-10)*=<30pt,10pt>[F]{x_1}="X1";
(2,-10)*=<30pt,10pt>[F]{x_2}="X2";
"O";"X1"**\dir{-};
"O";"X2"**\dir{-};
(-26,-12);(-18,-47)**\crv{(-26,-42)};
(6,-20)*=<30pt,10pt>[F]{w}="W";
(6,-12);"W"**\dir{-};
(-10,-20)*=<30pt,10pt>[F]{v}="V";
(-18,-12);(-15,-14.5)**\dir{-};
(-13.5,-16.5);"V"**\dir{-};
(-2,-12);(-14,-43.5)**\crv{(-2,-18)&(-30,-5)};
(-14,-22);(-17,-31)**\crv{(-14,-25)};
(-18.5,-33);(-22,-42)**\crv{(-21,-40)};
(-18,-47);(2,-22)**\crv{(2,-30)};
(10,-22);(10,-50)**\dir{-};
(-18,-47);(-18,-50)**\dir{-};
(-6,-22);(-2,-30)**\crv{(-6,-26)};
(-1,-31);(2,-50)**\crv{(2,-35)};
(-19, -25)*{\bullet}; (-14.5,-25)*{\circ};
\endxy
};
(25,0)*{=};
(75,0)*{=};
\endxy
\]
We now use the fact that the map $\ipd$ satisfies
$\ipd\circ\beta\circ(\beta_{\a}^{-1})^{\rho}=\ipd$,
\ie
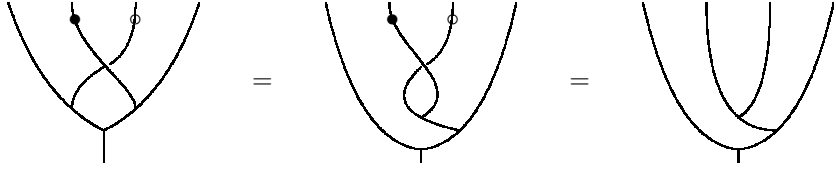
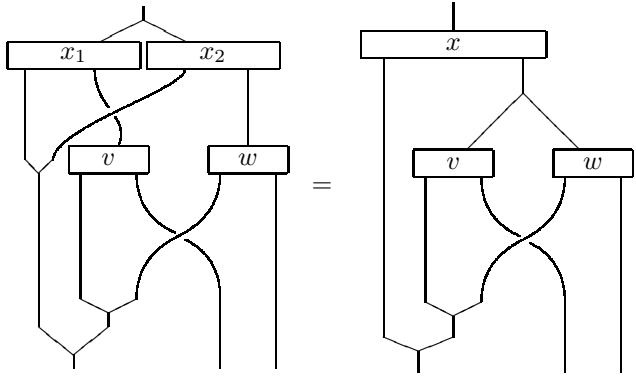
\[
\xy
0;/r.20pc/:
(0,0)*{
\xy
(10.5,-2.5)*{\bullet};
(20, -2.5)*{\circ};
(0,0)*{}="A";
(10,0)*{}="B";
(20,0)*{}="C";
(30,0)*{}="D";
"A";(15,-20)**\crv{(5,-15)};
"B";(20,-16.5)**\crv{(10,-7)&(20,-13)};
"C";(16,-9.5)**\crv{(20,-7)};
(15,-10); (10,-16.5)**\crv{(10,-13)};
"D";(15,-20)**\crv{(25,-15)};
(15,-20);(15,-25)**\dir{-};
\endxy
};
(25,0)*{=};
(50,0)*{
\xy
(10.5,-2.5)*{\bullet};
(20, -2.5)*{\circ};
(0,0)*{}="A";
(10,0)*{}="B";
(20,0)*{}="C";
(30,0)*{}="D";
"A";(15,-23)**\crv{(5,-22)};
"B";(15,-18)**\crv{(10,-7)&(23,-13)};
"C";(16,-9.5)**\crv{(20,-7)};
(15,-10); (21,-20)**\crv{(7.5,-17)};
"D";(15,-23)**\crv{(25,-22)};
(15,-23);(15,-25)**\dir{-};
\endxy
};
(75,0)*{=};
(100,0)*{
\xy
(0,0)*{}="A";
(10,0)*{}="B";
(20,0)*{}="C";
(30,0)*{}="D";
"A";(15,-23)**\crv{(5,-22)};
"B";(21,-20)**\crv{(10,-20)};
"C";(15,-18)**\crv{(20,-15)};
"D";(15,-23)**\crv{(25,-22)};
(15,-23);(15,-25)**\dir{-};
\endxy
};
\endxy
\]
Finally we get
\[
\xy
0;/r.22pc/:
(0,0)*{
\xy
(-15,-5)*=<50pt,10pt>[F]{x_1}="X1";
(5,-5)*=<50pt,10pt>[F]{x_2}="X2";
(-10,-20)*=<30pt,10pt>[F]{v}="C"; 
(10,-20)*=<30pt,10pt>[F]{w}="D";
(-5,0);"X1"**\dir{-};
(-5,0);"X2"**\dir{-};
(-5,0);(-5,2)**\dir{-};
(-14,-22)*{};(-14,-40)**\dir{-};
(14,-22)*{};(14,-40)**\dir{-};
(-6,-22)*{};(-0.5,-30.5)**\crv{(-6,-28)};
(6,-22)*{};(0,-31)**\crv{(6,-28)};
(-6,-40)*{};(0,-31)**\crv{(-6,-34)};
(6,-40)*{};(1,-31.5)**\crv{(6,-34)};
(-14,-40)*{}; (-10,-42)**\dir{-}; 
(-6,-40)*{}; (-10,-42)**\dir{-};
(-10,-42)*{};(-10,-44)**\dir{-};
(6,-40)*{};(6,-50)**\dir{-};
(14,-40)*{};(14,-50)**\dir{-};
(10,-7);(10,-18)**\dir{-};
(1,-7);(-18,-20)**\crv{(1,-10)&(-17,-15)};
(-12,-7);(-10,-12.5)**\crv{(-12,-10)};
(-9,-14);"C" **\crv{(-7,-17)};
(-22,-7);(-22,-20)**\dir{-};
(-22,-20);(-20,-22)**\dir{-};
(-18,-20);(-20,-22)**\dir{-};
(-20,-22);(-20,-44)**\dir{-};
(-20,-44);(-15,-48)**\dir{-};
(-10,-44);(-15,-48)**\dir{-};
(-15,-48);(-15,-50)**\dir{-};
\endxy
};
(25,0)*{=};
(50,0)*{
\xy
(0,-5)*{}="A";
(0,-10)*{}="B";
(-10,-20)*=<30pt,10pt>[F]{v}="C"; 
(10,-20)*=<30pt,10pt>[F]{w}="D";
"A";"B" **\dir{-};
"C";"B" **\dir{-};
"D";"B" **\dir{-};
(-14,-22)*{};(-14,-40)**\dir{-};
(14,-22)*{};(14,-40)**\dir{-};
(-6,-22)*{};(-0.5,-30.5)**\crv{(-6,-28)};
(6,-22)*{};(0,-31)**\crv{(6,-28)};
(-6,-40)*{};(0,-31)**\crv{(-6,-34)};
(6,-40)*{};(1,-31.5)**\crv{(6,-34)};
(-14,-40)*{}; (-10,-42)**\dir{-}; 
(-6,-40)*{}; (-10,-42)**\dir{-};
(-10,-42)*{};(-10,-45)**\dir{-};
(6,-40)*{};(6,-45)**\dir{-};
(14,-40)*{};(14,-45)**\dir{-};
(-10,-3)*=<70pt,10pt>[F]{x}="X";
(-10,3)*{};"X" **\dir{-};
(-20,-5)*{};(-20,-45)**\dir{-};
(-10,-45);(-15,-47)**\dir{-};
(-20,-45);(-15,-47)**\dir{-};
(-15,-47);(-15,-50)**\dir{-};
(6,-45);(6,-50)**\dir{-};
(14,-45);(14,-50)**\dir{-};
\endxy
}
\endxy
\]
It follows immediately from \eqref{eq:end-f} that the restriction of the counit
$\varepsilon$ of $\sfEnd{\FF{\a}{\b}}$ to $\FF{\a}{\b}(\Lm)$ satisfies
\[
(\varepsilon\ten\id)\circ\DeltaL=\id=(\id\ten\varepsilon)\circ\DeltaL.
\]
The compatiblity of product and coproduct on $\FF{\a}{\b}(\Lm)$ follows
from that in $\sfEnd{\FF{\a}{\b}}$, and the tuple 
$(\FF{\a}{\b}(\Lm),\mu,u,\DeltaL,\varepsilon)$ gives a bialgebra object in $\hDrY{\gdms}{\Phi}$.
Moreover, since $\FF{\a}{\b}(\Lm)$ reduces to $U\Lmm$ modulo $\hbar$, there exists a unique
antipode defining $\FF{\a}{\b}(\Lm)$ a Hopf algebra object in $\hDrY{\gdms}{\Phi}$.
This complete the proof of part $(i)$ in \ref{thm:braidhopf}.
\end{pf}

\subsubsection{Twisted $R$--matrix and coactions}\label{ss:compcoaction}

The tensor functor $(\FF{\a}{\b}, J\resped)$ induces a natural braiding
on the subcategory generated in $\hDrY{\gdms}{\Phi}$ by the objects $\FF{\a}{\b}(V)$,
for any $V\in\hDrY{\gums}{\Phi}$. The braiding is given by the usual formula
\[\beta_{J\resped}= J_{\FF{\a}{\b}, 21}^{-1}\circ\FF{\a}{\b}(\beta)\circ J\resped,\]
where $\beta$ is the usual braiding in $\hDrY{\gums}{\Phi}$. Clearly, 
$\beta_{J\resped}$ is a morphism in $\hDrY{\gdms}{\Phi}$.\\

For every $V\in\hDrY{\gums}{\Phi}$, the object $\FF{\a}{\b}(V)$ is endowed with
a trivial comodule structure over $\FF{\a}{\b}(\Lm)$ by the map
\[\eta_V:\FF{\a}{\b}(V)\to\FF{\a}{\b}(\Lm)\ten\FF{\a}{\b}(V)\qquad
\eta_V(v)=u\ten v.\]
The maps $\eta$ are morphisms in $\hDrY{\gdms}{\Phi}$ and they satisfy
\[\PhiD(\eta\ten\id)\eta_V=(\id\ten\eta_V)\eta_V,\]
where $\eta:=\eta_{\Lm}$ and $\eta(u)=u\ten u=\DeltaL(u)$. 
Set 
\newcommand{\Lcoa}{\mu^*}
\[\Lcoa_V:\FF{\a}{\b}(V)\to\FF{\a}{\b}(\Lm)\ten\FF{\a}{\b}(V),\qquad
\Lcoa_V=\Rd_{V}\circ\eta_V,\]
where $\Rd:=\sigma\cdot\beta_{J\resped}$ is the \emph{relative} $R$--matrix.
\begin{proposition}\label{prop:comod}
The map $\Lcoa_V$ defines on $\FF{\a}{\b}(V)$ a structure
of comodule over $\FF{\a}{\b}(\Lm)$. This is compatible with
the action of $\FF{\a}{\b}(\Lm)$ and defines on $\FF{\a}{\b}(V)$
a \DYt structure over $\FF{\a}{\b}(\Lm)$.
\end{proposition}

\begin{pf}
For any $V\in\hDrY{\gums}{\Phi}$
\[
(\id\ten\eta_V)\Rd\eta_V={\Rd}_{23}(\id\ten\eta_V)\eta_V,\qquad
(\eta_V\ten\id)\Rd\eta_V={\Rd}_{13}(\eta_V\ten\id)\eta_V.\]
Therefore, by the hexagon relations of $\Rd$ it follows
\begin{align*}
(\id\ten\Lcoa_V)\Lcoa_V&={\Rd}_{13}(\id\ten\eta_V)\Rd\eta_V\\
&={\Rd}_{13}{\Rd}_{23}(\id\ten\eta_V)\eta_V\\
&=(\DeltaL^{21}\ten\id)\Lcoa_V.
\end{align*}
The compatibility with the action is similarly proved, relying on the fact that
the relative $R$--matrix defines a quasi--triangular structure on 
$\sfEnd{\FF{\a}{\b}}$.
\end{pf}

Let now denote by $\hDrY{L,\a}{\Phi}$ the braided category of \DYt
$\FF{\a}{\b}(\Lm)$--modules in $\hDrY{\a}{\Phi}$ with braiding $\beta_{J\resped}$.
It follows from the last Proposition that the tensor functor $\FF{\a}{\b}$ factors
through $\hDrY{L,\a}{\Phi}$, \ie
\[
\xymatrix{
\hDrY{\gums}{\Phi} \ar[d]_{\FF{\a}{\b}}\ar[r]^{\wtFF{\a}{\b}} & \hDrY{L,\a}{\Phi} \ar[dl]^{\ff}\\
\hDrY{\gdms}{\Phi}
}
\]
where $\ff$ denotes the forgetful functor.
This completes the proof of part $(ii)$ and $(iii)$ of Theorem \ref{ss:gammalift}.

\subsection{Quantisation of a split pair}\label{ss:twostepq}

The application of Etingof--Kazhdan functor ${\noEKeq{\a}}:\hDrY{\gdms}{\Phi}\to\DrY{\EK\gdms}$
allows to further transform $\FF{\a}{\b}(\Lm)$ into a Hopf algebra object 
\[
\HL=({\noEKeq{\a}}\circ\FF{\a}{\b}(\Lm), {\noEKeq{\a}}(\mu), {\noEKeq{\a}}(\delta))
\]
in the category $\aDrY{\EK{\gdms}}$. 
The action and coaction of $\FF{\a}{\b}(\Lm)$
on $\FF{\a}{\b}(V)$ provide through $\noEKeq{\a}$ an action and an admissible coaction of $\HL$
on $\noEKeq{\a}\circ\FF{\a}{\b}(V)$ in the category $\hDrY{\EK{\a}}{}$. 
Namely, we have a diagram of functors
\begin{equation}\label{eq:diag1}
\xymatrix{
\hDrY{\gums}{\Phi} \ar[d]_{\FF{\a}{\b}}\ar[r]^{\wtFF{\a}{\b}} & 
\hDrY{L,\a}{\Phi} \ar[dl]^{\ff} \ar[r]^(.4){\noEKeq{\a}}& \aDrY{\HL,\EK{\a}} \ar[d]^{\ff}\\
\hDrY{\gdms}{\Phi} \ar[rr]^{\noEKeq{\a}} & & \aDrY{\EK{\a}}
}
\end{equation}
where $\hDrY{\HL,\EK{\a}}{}$ denotes the category of \DYt modules over $\HL$
in $\hDrY{\EK\a}{}$.

\subsubsection{}
Since $\HL$ is a Hopf algebra in $\DrY{\EK\a}$, it follows from \S\ref{ss:splitHA}--\S\ref{ss:rad-dy} that
we can construct a Hopf algebra $\HB=\Rad{\HL}{\EK\a}$,
and include the equivalence $\aDrY{\HL,\EK\a}\simeq\aDrY{\HB}$ in the diagram \eqref{eq:diag1},
\begin{equation}\label{eq:diag2}
\xy
(0,0)*+{\hDrY{\gums}{\Phi}}="A";
(0,-15)*+{\hDrY{\gdms}{\Phi}}="E";
(20,0)*+{\hDrY{L,\a}{\Phi}}="B";
(50,0)*+{\aDrY{\HL,\EK{\a}}{}}="C";
(70,0)*+{\aDrY{\HB}}="D";
(70,-15)*+{ \aDrY{\EK\a}}="F";
"A";"E"**\dir{-} ?>*\dir{>}; 
"A";"B"**\dir{-} ?>*\dir{>}; 
"B";"E"**\dir{-} ?>*\dir{>}; 
"B";"C"**\dir{-} ?>*\dir{>}; 
"C";"D"**\dir{-} ?>*\dir{>}; 
"D";"F"**\dir{-} ?>*\dir{>}; 
"C";"F"**\dir{-} ?>*\dir{>}; 
"E";"F"**\dir{-} ?>*\dir{>}; 
(-4,-7.5)*\txt\Small{$\FF{\a}{\b}$};
(80,-7.5)*\txt\Small{$\Res_{\EK\a,\HB}$};
(9,3)*\txt\Small{$\wtFF{\a}{\b}$};
(58,-8.5)*\txt\Small{$\ff$};
(12,-8.5)*\txt\Small{$\ff$};
(33,3)*\txt\Small{$\noEKeq{\a}$};
(33,-12)*\txt\Small{$\noEKeq{\a}$};
(61,2)*\txt\Small{$\simeq$}
\endxy
\end{equation}
The equivalence $\wt{\noEKff{\a}\noEKff{\a,\b}}$, obtained by 
composition along the top row of \eqref{eq:diag2}, is a Tannakian lift of the fiber
functor $\noEKff{\a}\noEKff{\a,\b}$ with tensor structure given by $J_{\a}$ and 
$J_{\a,\b}$.
The isomorphism $\wt{v}\resped$ constructed in Section \S\ref{s:Gammah}
descends to an isomorphism of tensor functors between $\noEKff{\b}$ and
$\noEKff{\a}\circ\noEKff{\a,\b}$, and leads to the following
\begin{theorem}\label{thm:transf}\hfill
\begin{itemize}
\item[(i)]
The pair $(\HB, \EK{\gdms})$ is a split pair of Hopf algebras 
quantising the split pair $(\gums, \gdms)$.
\item[(ii)] There is an isomorphism of split pairs of Hopf algebras 
\[v\resped:(\HB,\EK{\gdms})\simeq(\EK{\gums},\EK{\gdms}).\]
\end{itemize}
\end{theorem}
\begin{pf}
$(i)$ follows by construction. It is in fact immediate to
show that the $1$--jet of the composition of the twists $J\resped$ and 
$J_{\a}$ gives back the $r$--matrix of $\b$.\\
In order to prove $(ii)$, we observe that, by functoriality of the quantisation, 
$(\EK\b,\EK\a)$ is a split pair and, by Radford theorem, there is an isomorphism
of Hopf algebras $\EK\b\simeq\Rad{\Lmh}{\EK\a}$, where $\Lmh$ is the 
Verma module constructed in \S\ref{ss:relquantumverma1}.
Through $\wt{v}\resped$ we obtain an identification
\begin{equation}\label{eq:Liso-2}
\noEKeq{\a}\FF{\a}{\b}(\Lm)\simeq\noEKeq{\b}(\Lm)\simeq{\Lmh}
\end{equation}
as Hopf algebra objects in $\DrY{\EK\a}$. More specifically, it is
clear that \eqref{eq:Liso-2} preserves the structures of coalgebra and
\DYt $\EK{\a}$--module. The algebra structure on
$\Lmh\simeq\noEKeq{\b}(\Lm)$ is induced by its inclusion in 
$\EK\b\simeq\noEKeq{\b}(\Mm)$, which is a subalgebra in $\sfEnd{\noEKff{\b}}$.
Similarly, $\noEKeq{\a}\FF{\a}{\b}(\Lm)$ is a subalgebra in $\sfEnd{\noEKff{\a}\FF{\a}{\b}}$,
which is isomorphic to $\sfEnd{\noEKff{\b}}$ through $v\resped$.
Therefore, we obtain an isomorphism of Hopf algebras
\[
\EK\b\simeq\Rad{\Lmh}{\EK\a}\simeq\Rad{\noEKeq{\a}\FF{\a}{\b}(\Lm)}{\EK\a}=\HB,
\]
which restricts to the identity on $\EK\a$.
\end{pf}


\section{Universal constructions}\label{s:props}

We review in this section the $\PROP$ic quantisation of a Lie bialgebra
$\b$ following \cite{ek-2}. We then extend that description to the Tannakian
functor $\noEKeq{\b}:\hDrY{\b}{\Phi}\to\aDrY{\EK{\b}}$, and use it to give
an alternative proof that $\noEKeq{\b}$ is an equivalence.

\subsection{PROPs \cite{mac,La}}\label{ss:prop-intro} %

A $\PROP$ is a $\sfk$--linear, strict, symmetric monoidal category
$\C$ whose objects are the non--negative integers, and such that
$[n]\ten[m]=[n+m]$. In particular $[0]$ is the unit object, and $[1]
^{\ten n}=[n]$. A morphism between two $\PROP$s $\C,\D$ is a
symmetric tensor functor $\G:\C\to\D$, which is the identity on
objects, \ie $\G([n]_{\C})=[n]_{\D}$, $n\geqslant0$, and has a
trivial tensor structure.

\subsection{Associative algebras}\label{ss:associative}

Let $\preAlg$ be the $\PROP$ whose morphisms are generated
by two elements $\iota:[0]\to[1]$ (the unit) and $m:[2]\to[1]$ (the
multiplication) satisfying the relations
\begin{gather*}
m\circ(m\ten\id_{[1]})=m\circ(\id_{[1]}\ten m),\\
m\circ(\iota\ten\id_{[1]})=\id_{[1]}=m\circ(\id_{[1]}\ten\iota).
\end{gather*}

A (unital, associative) $\sfk$--algebra is the same as a symmetric tensor
functor  $\Alg\to\vectk$. More precisely, there is functor $\ev$ from the
category $\Fun^{\ten}(\Alg,\vectk)$ of symmetric tensor functors from
$\preAlg$ to $\vectk$, to the category $\mathsf{Alg}(\vectk)$ of $\sfk
$--algebras. The functor $\ev$ sends $(\G,J)$ to $\left(\G([1]),\G(m)
\circ J_{[1],[1]},\G(\iota)\right)$, and is easily seen to be an equivalence
of categories.

The equivalence $\mathsf{ev}$ may be restricted to an isomorphism as
follows. Choose a family $\mathbf{b}=\{b_n\}_{n\geq 2}$, where $b_n$
is a complete bracketing on $x_1\cdots x_n$, and restrict $\mathsf{ev}$
to the subcategory $\mathsf{Fun}^{\ten}_{\mathbf{b}}(\preAlg, \vectk)$
of symmetric tensor functors $(\G,J)$ such that
\begin{equation}\label{eq:F on n}
\G([n])=\G([1])^{\otimes n}_{b_n}
\end{equation}
and the following diagram is commutative
\begin{equation}\label{eq:J on n,m}
\xymatrix{
\G([m])\otimes \G([n]) \ar@{=}[d]\ar[rr]^{J_{[m],[n]}} && \G([m+n])\ar@{=}[d]\\
\G([1])^{\otimes m}_{b_m}\otimes \G([1])^{\otimes n}_{b_n}\ar[rr]_{\Phi}&& \G([1])^{\otimes(m+n)}_{b_{m+n}}
}
\end{equation}
where $\Phi$ is the corresponding associativity constraint
in $\vect$. Then, $\ev$ restricts to an isomorphism of categories 
\[\ev:\Fun_{\mathbf{b}}^\otimes(\Alg,\vectk)\to\preAlg(\vectk).\] 
Moreover, it is clear that, for any choices $\mathbf{b}, \mathbf{b}'$,
there is a canonical isomorphism $\Fun_{\mathbf{b}}^\otimes(\Alg,
\vectk)\to\Fun_{\mathbf{b}'}^\otimes(\Alg,\vectk)$. 

\subsection{Modules over a $\PROP$}

The discussion in \ref{ss:associative} may be extended to an arbitrary
$\PROP$ $\sfP$ as follows. Fix henceforth a choice ${\bf b}=\{b_n\}_
{n\geq 2}$ of bracketings. 

\begin{definition}
A module over $\sfP$ in a symmetric monoidal category
$\N$ is a symmetric tensor functor $\sfP\to\N$ such that \eqref{eq:F on n} and 
\eqref{eq:J on n,m} above hold, where $\Phi$ is the 
associativity constraint in $\N$. A morphism of modules over $\sfP$ is a natural
transformation of functors. The category of $\sfP$--modules is
denoted by $\Fun_{\bfb}^\otimes(\sfP,\N)$.
\end{definition}

In particular, by definition, a $\preAlg$--module in $\vectk$ is a $\sfk$--algebra.
The notion of $\sfP$--module can be rephrased in term of morphisms of
$\PROP$s (cf. \cite[\S2.4]{ee}). Namely, for any object $X\in\N$, we can
consider the $\PROP$ $\sfP_{\mathbf{b}, X}$ with morphisms 
\[
\sfP_{\mathbf{b}, X}([n],[m])=\Hom_{\N}(X^{\ten n}_{b_n}, X^{\ten m}_{b_m})
\]
Then, $X$ is a $\sfP$--module in $\N$ if and only if there is a morphism
of $\PROP$s $\sfP\to\sfP_{\mathbf{b}, X}$.

\subsection{Universal Constructions}
Let $\preLA$ be the $\PROP$ generated by a morphism $\mu:[2]
\to[1]$ (the bracket), subject to the relations
\[\mu\circ(\id_{[2]}+(1\,2))=0
\aand \mu\circ(\mu\ten\id_{[1]})\circ(\id_{[3]}+(1\,2\,3)+(3\,1\,2))=0\]
as morphisms $[2]\to[1]$ and $[3]\to[1]$ respectively. Thus, a
$\preLA$--module in $\vectk$ is $\sfk$--Lie algebra.

The fact that any associative algebra $A$ is naturally a Lie algebra
may be described in terms of $\PROP$s as follows. 
Let $(\G_{A},J_A)\in\mathsf{Fun}^{\ten}_{\bfb}(\preAlg,\vectk)$ be such that
$\G_{A}[1]=A$. The corresponding Lie algebra arises from the
composition
\[
\xymatrix{
\preLA\ar[r]^L & \preAlg \ar[r]^{(\G_{A},J_A)} & \vectk},
\]
where $L:\preLA\to \preAlg$ is the strict symmetric tensor functor
mapping $[1]_{\preLA}$ to $[1]_{\preAlg}$, and $\mu$ to $m-m\circ
(21)$. The functor from the category of associative algebras to the
category of Lie algebras is then realized as a morphism of $\PROP$
from $\preLA$ to $\preAlg$.

\subsection{\BCH multiplication}\label{ss:BCH}

The $\PROP$ $\preLA$ is not well--suited for the description of the adjoint
construction, which assigns to a Lie algebra $\g$ its universal enveloping
algebra $\Ug$, however. The latter is isomorphic to the symmetric algebra
$S\g$ of $\g$ as a vector space via the symmetrisation map $\sigma:S\g
\to\Ug$. The multiplication $m$ on $U\g$ can therefore be transported to
$S\g$ as the star product
\[\star:=
\sigma^{-1}\circ m\circ (\sigma\ten\sigma)=
\bigoplus_{k\leq i+j}m_{i,j}^k,\]
where $m_{i,j}^k: S^i\g\ten S^j\g\to S^k\g$ can be expressed in terms of
the bracket on $\g$ (see for example \cite{cor}). Indeed, $S\g$ is spanned
by the elements $Z^k$, $Z\in\g$, $k\in\IN$, and the corresponding generating
series $\exp_S(vZ)\in S\g[[v]]$ is mapped to $\exp_U(vZ)\in\Ug[[v]]$ by
$\sigma$. It follows that, for $X,Y\in\g$,
\[\exp_S(t X)\star\exp_S(sY)=
\sigma^{-1}\left(\exp_U(t X)\exp_U(t X)\right))=
\exp_S(B(tX,sY)),
\]
where $B$ is the \BCH Lie series. Thus,
\[ X^i\star Y^j=
\left.\partial_t^i\partial_s^j\exp_S(B(tX,sY))\right|_{t=s=0}.\]

To describe the above procedure in terms of $\PROP$s, ones needs to
construct the subobjects $S^k[1]\subset [k]$ in $\preLA$. This requires
replacing the $\PROP$ $\preLA$ with its Karoubi envelope.

\subsection{The Karoubi envelope}\label{ss:karoubi}

The Karoubi envelope of a category $\C$ is the category $\Kar{\C}$ whose
objects are pairs $(X,\pi)$, where $X\in\C$ and $\pi:X\to X$ is an idempotent.
The morphisms in $\Kar{\C}$ are defined as
\[
{\Kar{\C}}((X,\pi), (Y,\rho))=\{f\in\C(X,Y)\;|\; \rho\circ f=f=f\circ\pi\}
\]
with $\id_{(X,\pi)}=\pi$. In particular, 
\[
{\Kar{\C}}((X,\id),(Y,\id))={\C}(X,Y)
\]
and the functor $\C\to\Kar{\C}$, mapping $X\mapsto(X,\id)$, $f\mapsto f$, 
is fully faithful.

Every idempotent in $\Kar{\C}$ splits. Namely, if $q\in\Kar{\C}((X,\pi),(X,\pi))$
satisfies $q^2=q$, the maps
\[i=q:(X,q)\to(X,\pi)\aand p=q:(X,\pi)\to(X,q)\]
satisfy $i\circ p=q$ and $p\circ i=\id_{(X,q)}$. In particular, the Karoubi envelope
$\Kar{\C}$ of a $\PROP$ $\C$ contains the image of all the idempotents in $\sfk
\SS_n$. This allows to consider the objects $S^n[1]:=([n],\Sym_n)$ for any $n\geq
0$, where $\Sym_n$ is the symmetriser $\frac{1}{n!}\sum_{\sigma\in\SS_n}\sigma$.
If one further takes the closure $\cKar{\C}$ of $\Kar{\C}$ under possibly
infinite inductive limits, then $\cKar{\C}$ contains, in particular, the symmetric algebra 
$S[1]=\bigoplus_{n\geq0} S^n[1]$.
The construction of the universal enveloping algebra of a Lie algebra corresponds
to a functor $U: \preAlg\to \cKar{\LA}$ mapping $[1]_{\preAlg}$ to $S[1]_{\LA}$,
and $m$ to $\bigoplus_{k\leq i+j}m_{ij}^k$.

\subsection{Etingof--Kazhdan quantisation}\label{ss:EKpropq}
We consider the following categories
\begin{itemize}
\item the $\PROP$ $\LBA$ generated by morphisms $[\,,\,]:[2]\to[1]$ and
$\delta:[1]\to[2]$ such that $([1],[\,,\,],\delta)$ is a Lie bialgebra;
\item the $\hext{\sfk}$--linear category $\slbah$ with the
same objects as $\cKar{\LBA}$, and morphisms given by
\[\Hom_{\slbah}(V,W)=\Hom_{\cKar{\LBA}}(V,W)\fml;\]
\item the $\PROP$ $\BA$ generated by morphisms
$m:[2]\to[1]$, $\iota:[0]\to[1]$, $\Delta:[1]\to[2]$, and $\varepsilon:[1]\to[0]$,
such that $([1],m,\iota,\Delta,\varepsilon)$ is a bialgebra;
\item the $\PROP$ $\qcBA$ obtained by extending $\BA\fml$ with a morphism 
$\delta:[2]\to[1]$ satisfying the relation $\Delta-(1\,2)\circ\Delta=\hbar\delta$,
and modding out the torsion ideal (cf. \cite[\S4.1]{ee});
\item the topological $\PROP$ $\QUE$ obtained by completing $\qcBA$ with respect to the ideal generated by $(\id-\eta\circ\varepsilon)^{\ten n}\circ\Delta^{(n)}$.
\end{itemize}

The quantisation functor can be described in this context \cite
[Thm. 1.2]{ek-2}.

\begin{theorem}\label{thm:propEK}
There exists a universal quantisation functor 
\[Q:\QUE\longrightarrow\slbah\]
such that
\[
\begin{array}{cccl}
Q[B]&=&S[\b],&\\
Q(m)&=&m_0&\mod\hbar,\\
Q(\Delta)&=&\Delta_0&\mod\hbar,\\
\left.Q(\Delta-\Delta^{21})\right|_{[\b]}&=&\hbar\delta&\mod\hbar^2.
\end{array}
\]
where $[B]$ and $[\b]$ the generating objects in $\QUE$ and
$\LBA$, respectively, and $m_0, \Delta_0$ are the product and coproduct
of the enveloping algebra $U[\b]$.
\end{theorem}

The functor $Q$ depends upon the choice of an associator $\Phi$, and is
equivalent to the construction of a Hopf algebra structure on $U\pb=S\pb$, 
which lifts to the category $\slbah$ the construction
described in \S\ref{ss:eksum}. We outline this lift in \S\ref{ss:EK PROP start}--\ref
{sss:qpropmod}
below.

\subsection{\DYt modules}\label{ss:EK PROP start}

Let $\DrY{\pb}$ be the strict, symmetric tensor category of
\DYt modules over the Lie bialgebra $[\b]$ in $\slbah$. For
any $U,V\in\hDrY{[\b]}{}$, define $r_{UV}\in\End_{\cKar{\LBA}\fml}(U\otimes V)$
and $\Omega_{UV}\in\End_{\hDrY{[\b]}{}}(U\otimes V)$ by
\[r_{UV}=\pi_{U}\otimes\id\circ (1\,2)\circ \id\otimes\pi^*_V
\aand
\Omega_{UV}=r_{UV}+r^{21}_{V,U}.\]
Let $\Phi$ be a Lie associator, and $\DrY{\pb}^\Phi$ the category
$\DrY{\pb}$ with deformed commutativity and associativity
respectively given by
\[\beta_{UV}=(1\,2)\circ \exp{\left(\frac{\hbar}{2}\Omega_{UV}\right)}
\aand
\Phi_{UVW}=\Phi(\hbar\Omega_{UV},\hbar\Omega_{VW}).
\]

\subsection{The Verma module $\PMm$}\label{sss:PMm}

The Verma modules $\Mm,\Mpv$ may be lifted to \DYt modules
$\PMm,\PMpv$ over $\pb$ in $\cKar{\LBA}$. $\PMm$ is
obtained as follows.
\begin{itemize}
\item As an object, $\PMm=S\pb$.
\item The action $\pi$ of $\pb$ on $\PMm$ is given by the
multiplication map
\[\bigoplus_{0\leq j\leq i+1}m_{1,i}^j: \pb\ten S\pb\to S\pb\]
obtained from the \BCH series described in
\ref{ss:BCH}.
\item The coaction $\pi^*$ of $\pb$ is uniquely determined 
by requiring that it be
\begin{itemize}
\item trivial on $S^0\pb$;
\item compatible with the action of $\pb$ via the relation 
\[
\quad\qquad\pi^*\circ\pi=
\id\ten\pi\circ(1\,2)\circ \id\ten\pi^*-
\id\ten\pi\circ\delta\ten\id+
[\,,\,]\ten\id\circ\id\ten \pi^*.
\]
\end{itemize}
Specifically, the above requirements uniquely determine
a map $\pi^*:\PMm\to\pb\otimes\PMm$, by induction on
the $\IN$--grading of $S\pb$. A diagrammatic computation
then shows that $\pi_*$ is a coaction of $\pb$.
\end{itemize}

\subsection{Universal property of $\PMm$}\label{ss:univ-pr-PMm}

\begin{lemma}
For any $[V]\in\DrY{\pb}$, there is an isomorphism
\begin{equation}\label{eq:univ-pr-PMm}
\Hom_{\pb}^{\pb}(\PMm,[V])\simeq\Hom^{\pb}([0],[V])
\end{equation}
which is natural in $[V]$.
\end{lemma}
\begin{pf}
We proceed as in the case of the quantum Verma module $M_B$
(cf. \S\ref{ss:quantum-verma}). $\PMm$ is a unital algebra, with
multiplication given by the BCH series $m_{\b}$, and unit given
by the inclusion $\iota:[0]\to\PMm$. $\PMm$ is therefore a module
over itself, with $\ol{\pi}_{\PMm}=m_{\b}$, and the action of $\pb$
on $\PMm$ is recovered by restriction to $\pb$, \ie $\pi_{\PMm}=
m_{\b}\circ\iota_{\b}\ten\id_{\PMm}$ where $\iota_{\b}$ is the
canonical inclusion $\iota_{\b}:\pb\to\PMm$. 
Similarly, any action $\pi_{[V]}:\pb\ten[V]\to[V]$ extends to an
action $\ol{\pi}_{[V]}:\PMm\ten[V]\to[V]$ of $(\PMm,m_{\b},\iota)$
on $[V]$, which is defined on $S^n[\b]$ by $\pi^{(n)}_{[V]}\circ
\Sym_n$, where $\pi^{(n)}_{[V]}$ denotes the $n$th iterated action.

The description of \eqref{eq:univ-pr-PMm} is the following. 
In one direction, for any Drinfeld--Yetter morphism
$g:\PMm\to[V]$, one obtains a $\pb$--comodule map 
$\phi(g)=g\circ\iota:[0]\to[V]$. Conversely, for any morphism
of comodules $f:[0]\to[V]$, one defines a module map 
$\psi(f)=\ol{\pi}_{[V]}\circ\id_{\PMm}\ten f:\PMm\to[V]$.
One has $\psi\circ\phi(g)=g$, since $\PMm$ is
free of rank one on itself (\ie $\ol{\pi}_{\PMm}\circ\id_{\PMm}\ten\iota=\id_{\PMm}$),
and $\phi\circ\psi(f)=f$, since $\ol{\pi}_{[V]}$ is an action map
(\ie $\ol{\pi}_{[V]}\circ\iota\ten\id_{[V]}=\id_{[V]}$). 
\end{pf}

The same argument shows that, if $\b$ is a Lie bialgebra, and 
$\G_\b:\LBA\to\vect_\sfk$ a symmetric tensor functor such
 that $\G_\b\pb=\b$, then $\G_\b(\PMm)$ is a solution of the 
 universal property 
 \begin{equation*}
\Hom_{\b}^{\b}(\G_\b(\PMm),V)\simeq\Hom^{\b}(\sfk,V)
\end{equation*}
for any $V\in\DrY{\b}$.
In particular, $\G_\b(\PMm)$ is the \DYt module
$\Mm=\ind_{\b^*}^{\gb}\sfk$ over $\b$. Alternatively,
the result follows by identifying $\G_\b(\PMm)$ and
$\Mm$ as modules over $\pb$, and using the fact
that the coaction of $\pb$ on $\Mm$ is uniquely determined.

\subsection{Dualising Drinfeld--Yetter modules}

The $\PROP$ic description of the module $\Mpv$ given in \S
\ref{sss:Mpv-prop} below relies on the following considerations. 

Let $V$ be a \DYt module over a Lie bialgebra $(\b,[\,,\,],\delta)$
with left action and right coaction
\[\pi_V:\b\otimes V\to V\aand\pi^*_V:V\to\b\otimes V.\]
Then, the dual vector space $V^*$ is a \DYt module over the dual
(topological) Lie bialgebra $(\b^*,\delta^t,[\,,\,]^t)$, with action and
coaction given by
\begin{equation}\label{eq:flip dual DY}
\rho_{V^*}=-(\pi_V^*)^t:\b^*\otimes V^*\to V^*\aand
\rho^*_{V^*}=-\pi_V^t:V^*\to\b^*\otimes V^*.
\end{equation}

If $\b$ if finite--dimensional, the corresponding functor $\DY{\b}\to
\DY{\b^*}$ is readily seen to coincide with the duality enodfunctor
on $\Rep(\gb)$, via the identifications $\DY{\b}{}=\Rep(\gb)=\DY
{\b^*}$.

To lift this to the $\PROP$ic setting, notice the following.
\begin{itemize}
\item Let $\Theta:\LBA\to\LBA$ be the strict, contravariant
symmetric tensor functor mapping $[n]$ to $[n]$, a permutation
$\sigma\in\LBA([n],[n])$ to $\sigma^{-1}$, and permuting the bracket
$[\,,\,]:[2]\to[1]$ and the cobracket $\delta:[1]\to[2]$. Since
$\Theta\circ\Theta$ is the identity functor, $\Theta$ is an
isomorphism of $\LBA$ onto its opposite category.
\item $\Theta$ is a $\PROP$ic lift of the (symmetric, tensor)
duality functor $\sfD:\vect_\sfk\to\vect_\sfk$ in the following
sense. If $\G_\b:\LBA\to\vect_\sfk$ is a symmetric tensor
functor such that $\G_\b[1]=\b$, the functor $\G_{\b^*}=D\circ
\G_\b\circ\Theta$ maps $[1]$ to the Lie bialgebra $\b^*$
and, evidently, makes the following a commutative diagram
\begin{equation}\label{eq:Theta D}
\xymatrix{
\LBA \ar[r]^\Theta \ar[d]_{\G_\b} & \LBA \ar[d]^{\G_\b^*}\\
\vect_\sfk\ar[r]_\sfD & \vect_\sfk
}
\end{equation} 
\item If $(V,\pi_V,\pi^*_V)\in\DrY{\b}$ is $\PROP$ic, that is
of the form $\G_\b([V],\pi,\pi^*)$, where $[V]\in\DrY{\pb}$,
then so is $V^*$ as a \DYt module over $\b^*$. Specifically,
\begin{equation}\label{eq:V* PROP}
(V^*,\rho_{V^*},\rho_{V^*}^*)=
\G_{\b^*}\left(\Theta[V],\Theta(-\pi^*),\Theta(-\pi)\right)
\end{equation}
as follows from \eqref{eq:flip dual DY} and \eqref
{eq:Theta D}.
\end{itemize}

\subsection{The Verma module $\PMpv$}\label{sss:Mpv-prop}

Let $\b$ be a \fd Lie bialgebra. Since the action and coaction
of $\b^*$ on the Verma module $\Mp=\ind_{\b}^{\gb}\sfk$ are
described by the same formulae as those of $\b$ on $\Mm=
\ind_{\b^*}^{\gb}\sfk$, \eqref{eq:flip dual DY} may be used to
describe $\Mpv=\Mp^*$ as a \DYt module over $\b=(\b^*)^*$.
It then follows from \eqref{eq:V* PROP} that the \DYt module
$\Mpv$ possesses a lift to the $\PROP$ $\cKar{\LBA}$, given
by
\begin{equation}\label{eq:PMpv}
(\PMpv,\pi_+,\pi_+^*)=
\left(\Theta\PMm,\Theta(-\pi_-^*),\Theta(-\pi_-)\right),
\end{equation}
where $\pi_-,\pi_-^*$ are the action and coaction on $\PMm$.
In other words, $\PMpv$ is obtained from $\PMm$ by exchanging
action and coaction, permuting brackets and cobrackets, reversing
the order of operations, and applying a minus sign, as stated
in \cite[\S 1.4]{ek-2}. Note in particular that the application of
the functor $\Theta$ turns the ind--object $\PMm=S\pb$ into
the pro--object $\PMpv=\wh{S\pb}$.

Let now $\b$ an arbitrary Lie bialgebra, and $\G_\b:\LBA\to\vect
_\sfk$ a symmetric tensor functor such that $\G_\b\pb=\b$. We
show in \S\ref {ss:univ-pr-PMp} that $\G_\b\PMpv$ satisfies the
same universal property as $\Mpv$. It follows that $\PMpv$ is
$\PROP$ic lift of the module $\Mpv$ for any $\b$. Alternatively,
to show that $\G_{\b}\PMpv=\Mpv$ in $\DrY{\b}$, it suffices to
prove that their duals are equal in $\DrY{\b^*}$. By \eqref{eq:Theta D}
and Definition
\eqref{eq:PMpv},
\[\G_{\b}\PMpv^*=\G_{\b^*}\PMm=\wh{\Mp}\] 
where $\wh{\Mp}$ is the completion of the (algebraic) symmetric
algebra of $\b^*$ \wrt the weak topology. The action of $\b^*$
on $\wh{\Mp}$ is given by the CBH product, and its coaction is
uniquely determined by the fact that it kills the generating vector,
and is compatible with the action. Since $(\Mpv)^*$ also
satisfies these properties, the two coincide.

\subsection{Universal property of $\PMpv$}\label{ss:univ-pr-PMp}

\begin{lemma}\hfill
\begin{enumerate}
\item For any $[V]\in\DrY{\pb}$, there is an isomorphism
\[\Hom_{\pb}^{\pb}([V],\PMpv)\simeq\Hom^{\pb}([V],[0])\]
which is natural in $[V]$.
\item If $\b$ is a Lie bialgebra, and $\G_\b:\LBA\to\vect_\sfk$
a symmetric tensor functor such that $\G_\b\pb=\b$, then
$\G_\b(\PMpv)$ is a solution of the universal property
\[\Hom_{\b}^{\b}(V,\Mpv)\simeq\Hom^{\b}(V,\sfk)\]
for any $V\in\DrY{\b}$.
\end{enumerate}
\end{lemma}
\begin{pf}
We argue as in the case of the quantum Verma module $M_B^
{\vee}$ (cf. \S\ref{ss:quantum-verma-p}--\ref{ss:Mpv}). $\PMpv=
\wh{S}\pb$ is a counital coalgebra with comultiplication $\Delta_
{\b}=\Theta(m_{\b})$, and unit given by the canonical projection
$\varepsilon=\Theta(\iota):\PMpv\to[0]$. Therefore it is a comodule
over itself, with $\ol{\pi}^*_{\PMpv}=\Theta(m_{\b})^{\scs\operatorname
{op}}$, and the coaction of $\pb$ on $\PMm$ is recovered by projection
to $\pb$, \ie $\pi^*_{\PMm}=p_{\b}\ten\id_{\PMpv}\circ\Delta_{\b}
^{\scs\operatorname{op}}$ where $p_{\b}$ is the canonical projection
$p_{\b}=\Theta(\iota_{\b}):\PMpv\to\pb$. Similarly, any coaction
$\pi^*_{[V]}:[V]\to\pb\ten[V]$ extends to a coaction of $(\PMpv,
\Delta_{\b},\varepsilon)$ on $[V]$, $\ol{\pi}_{[V]}^*:[V]\to\PMpv\ten[V]$, 
given by $\sum_{n\geqslant0}\Sym_n\circ(\pi^*)^{(n)}_{[V]}$, where
$(\pi^*)^{(n)}_{[V]}$ denotes the $n$th iterated coaction.

(i) For any \DYt morphism $g:[V]\to\PMpv$, one obtains a $\pb$--module
map $\phi(g)=\varepsilon\circ g:[V]\to[0]$. Conversely, for any morphism
of modules $f:[V]\to[0]$, one defines a module map 
$\psi(f)=\circ\id_{\PMpv}\ten f\circ \ol{\pi}^*_{[V]}:[V]\to\PMpv$.
One has $\psi\circ\phi(g)=g$, since $\PMpv$ is
cofree of rank one on itself, \ie $\id_{\PMpv}\ten\varepsilon\circ\ol{\pi}_{\PMpv}^*
=\id_{\PMpv}$, and $\phi\circ\psi(f)=f$, since $\ol{\pi}^*_{[V]}$ is a coaction
map (\ie $\varepsilon\ten\id_{[V]}\circ\ol{\pi}^*_{[V]}=\id_{[V]}$). 
(ii) Follows by the same argument.
\end{pf}

\subsection{Fiber functor} \label{sss:prop-twist}

There is a tensor structure on the tautological forgetful functor $\sff_
{[\b]}:\hDrY{\pb}{\Phi}\to\slbah$. The latter is obtained by imitating the
construction of the fiber functor $\noEKff{\b}$ defined in \S\ref{ss:EK fiber},
once the latter has been identified with the forgetful functor $\sff_\b:
\hDrY{\b}{\Phi}\to\vect_{\sfk\fml}$.\footnote{The functor $\noEKff{\b}
(V)=\Hom_{\b}^{\b}(\Mm,\Mpv\wh{\ten}V)$ itself is not $\PROP$ic
since, for $V\in\hDrY{\pb}{\Phi}$, $\Hom_{\pb}^{\pb}(\PMm,\PMpv
\ten V)$ is a vector space, not an object in $\slbah$.}

Specifically, the modules $\PMm,\PMpv$ satisfy,
for any $V,W\in\hDrY{\pb}{}$ 
\begin{equation}\label{eq:thetaEK}
\Hom_{\pb}^{\pb}(\PMm\ten V^0, \PMpv\ten W)\simeq\Hom_{\cKar{\LBA}\fml}(V,W),
\end{equation}
where $V^0$ denotes the object $V$ with trivial action and coaction. Let
$\psi_V: \PMm\ten V^0\to\PMpv\ten V$ be the morphism corresponding to
$\id_V$ under \eqref{eq:thetaEK}, and let $\eta_V\in\Hom_{\cKar{\LBA}\fml}(V,
\PMpv\ten V)$ be defined by $\eta_V=\psi_V\circ 1_-\otimes\id_V$, where
$1_-:[0]\to\PMm$ is the canonical inclusion in $S\pb$.

The tensor structure on the forgetful functor from $\hDrY{\pb}{\Phi}$
to $\slbah$ is then given by
\[
J_{VW}=(1_+\ten\id)^{\ten 2}\circ A_{\Phi}\circ \eta_V\ten \eta_W
\in
\End_{\cKar{\LBA}\fml}(V\otimes W),
\]
where $V,W\in\hDrY{\pb}{\Phi}$, and $A_{\Phi}$ is the composition of
associativity and commutativity constraints defined in \S\ref{ss:EK fiber}.

\subsection{Quantisation of $U[\b]$}\label{sss:HA-struct}

The object $\PMm$ has a structure of algebra and coalgebra with unit
$\uh$ and counit $\couh$ given by the inclusion of and the projection
to $[0]$, respectively, and with product and coproduct given by the
maps
\[\mh=1_+\ten 1_+\ten\id\circ\id\ten\psi_{\PMm}\circ\Phi\circ\eta_{\PMm}\ten\id
\aand
\Dh=J^{-1}_{\PMm,\PMm}\circ\Delta_0,\]
where $1_+$ is the canonical projection of $\PMpv=\wh{S\pb}$ onto
$[0]$, and
$\Delta_0$ is the standard coproduct on $\PMm=S\pb$. Moreover,
$\PMm$ acts on any $V\in\hDrY{\pb}{\Phi}$, with action
\[\rho_V=1_+\ten 1_+\ten\id\circ\id\ten\psi_{V}\circ\Phi\circ\eta_{\PMm}\ten\id\]
satisfying $\rho_V\circ \mh\ten\id = \rho_V\circ (\id\ten \rho_V)$ and
\begin{align*}
\rho_{V\ten W}=J_{VW}\circ\rho_V\ten\rho_W\circ(2\,3)\circ\Dh\ten J_{VW}^{-1}.
\end{align*}
Since the bialgebra structure is a deformation of that on $U\pb$,
it follows that $(\PMm,\mh,\Dh)$ admits a Hopf algebra structure
(with antipode $\Sh$), and we set
\[Q([B], m, \iota, \Delta, \epsilon, S)=(\PMm, \mh, \uh, \Dh, \couh, \Sh).\]
It is easy to check that $Q[B]$ is a quantisation of the Lie bialgebra $\pb$
and gives rise to a functor $Q:\QUE\to\slbah$.

\subsection{Tannakian lift of $\sff_{\pb}$}\label{sss:qpropmod}

$\PMm$ coacts on any $V\in\hDrY{\pb}{\Phi}$, with coaction
\[\rho_V^*=R^J\circ\uh\ten\id_V,\]
where $R^J$ is the twisted $R$--matrix defined by $R^J=(1\,2)J_{V,W}^{-1}
\beta_{V,W}J_{V,W}$. The fiber functor $\sff_{[\b]}$ then lifts to
a braided monoidal functor $\wt{\sff}_{\pb}:\hDrY{\pb}{\Phi}\to\aDrY{Q[B]}$,
where the latter is the category of \DYt modules over the Hopf algebra
$Q[B]$ in $\cKar{\LBA}\fml$.

The functor $\wt{\sff}_{\pb}$ is a propic version of the Tannakian equivalence
$\noEKeq{\b}$. More precisely, any symmetric tensor functor $\G_{\b}:\LBA
\to\vectk$, which maps $\pb$ to a Lie bialgebra $\b$ gives rise to functors
\[\G_{\b}^\hbar:\hext{\cKar{\LBA}}\to\vect_{\sfk\fml},
\qquad
\G^{\DrY{}}_{\b}:\hDrY{\pb}{\Phi}\to\hDrY{\b}{\Phi}
\aand
\G^{\DrY{}}_{\EK\b}:\aDrY{Q[B]}\to\aDrY{\EK\b},
\]
where $\G_{\b}^{\hbar}(\pb)=
\hext{\b}$, $\G_{\b}^{\hbar}(\mu)=\mu$, $\G_{\b}^{\hbar}(\delta)=\delta$,
and $\aDrY{Q[B]}$ is the $\PROP$ defined in \S\ref{sss:col-prop-adm},
such that
\[
\xymatrix{
 & \aDrY{Q[B]} \ar[r]^{\G^{\DrY{}}_{\EK\b}} \ar@{-->}[dd] & \aDrY{\EK\b} \ar[dd]\\
\hDrY{\pb}{\Phi} \ar[dr]^{\sff_{\pb}} \ar[ur]^{\wt{\sff}_{\pb}} \ar[r]^{\G^{\DrY{}}_{\b}} 
& \hDrY{\b}{\Phi} \ar[dr]^{\sff_\b}\ar[ur]^{\wt{\sff}_\b} & \\
&\hext{\cKar{\LBA}}  \ar[r]^{\G_{\b}^{\hbar}} & \vectk
}
\] 
is a commutative diagram of tensor functors, where $\wt{\sff}_\b\cong
\noEKeq{\b}$ is obtained through the isomorphism $\sff_\b\cong F_\b$.

\subsection{The equivalence $\wt{\sff}_{\b}$}\label{sss:qpropmod-2}

We now sketch an alternative proof of Theorem \ref{thm:EK-equiv} that
the functor $\wt{\sff}_{\b}:\hDrY{\b}{\Phi}\to\aDrY{\EK\b}$ is a braided
tensor equivalence.

It was shown in \cite[Cor. 6.4]{ee} that the existence of a quantisation
functor $Q:\QUE\to\slbah$, together with Hensel's Lemma, imply that
$Q$ gives rise to an isomorphism of $\QUE$ and the $\PROP$ $\hext
{\cP}$ of co--Poisson universal enveloping algebras.
We shall adopt a similar strategy to show that $\sff_{\b}$ is an equivalence.

\subsubsection{}\label{sss:col-prop-adm}

Consider the \emph{colored} $\PROP$s (see \S\ref{ss:plba}) $\DrY{\LBA}$
and $\aDrY{\QUE}$ describing, respectively, a \DYt module over a Lie
bialgebra, and an admissible \DYt module over a quantised enveloping
algebra. 

The $\PROP$ $\DrY{\LBA}$ is generated by a Lie bialgebra object
$\pb$, together with an object $[V]$ endowed with the structure of
\DYt module over $\pb$. There is a canonical equivalence of categories 
\begin{equation}\label{eq:prop-eq-1}
\hDrY{\b}{\Phi}\simeq
\{\F\in\mathsf{Fun}_{\bf b}^{\ten}(\DrY{\LBA},\Kvect)\;|\; \F\pb=\b\}.
\end{equation}
The $\PROP$ $\aDrY{\QUE}$ is generated by a QUE object $[B]$,
together with a \DYt module $[V]$ over $[B]$, whose coaction factors
through the $\QQFSH$ subalgebra of $[B]$. This last condition is
encoded by requiring the existence of a map $\Xi_{[V]}:[V]\to[B]\ten[V]$
such that 
\[
(\id-\eta_{[B]}\circ\epsilon_{[B]})\ten\id_{[V]}\circ\pi_{[V]}^*=\hbar\,\Xi_{[V]}
\]
and modding out the torsion ideal.
In particular, it follows, as in the proof of Proposition \ref{prop:adm-cond},
\[
\left((\id-\eta_{[B]}\circ\epsilon_{[B]})^{\ten n}\circ\Delta^{(n)}\right)\ten\id_{[V]}\circ\pi_{[V]}^*
=\hbar^n\,\Xi_{[V]}^{(n)},
\]
where $\Xi_{[V]}^{(n)}:[V]\to[B]^{\ten n}\ten[V]$ is the $n$th iteration of $\Xi_{[V]}$.
Therefore, for any quantised universal
enveloping algebra $B$, there is an equivalence
\begin{align}\label{eq:prop-eq-2}
\aDrY{B}\simeq
&\;\{\F\in\mathsf{Fun}_{\bf b}^{\ten}(\aDrY{\QUE},\Kvect)\;|\; \F[B]=B\}.
\end{align}

\subsubsection{}

The formulae from \S\ref{sss:HA-struct} and \S\ref{sss:qpropmod} allow to
extend the quantisation functor $Q:\QUE\to\LBA\fml$ to a tensor functor $\wt{Q}:
\aDrY{\QUE}\to\hext{\DrY{\LBA}}$. 
Under the equivalences \eqref{eq:prop-eq-1}--\eqref{eq:prop-eq-2}, the braided 
tensor functor $\wt{\sff}_{\b}$ corresponds to the pullback of $\wt{Q}$.

Much like $Q$, the functor $\wt{Q}$ is not essentially surjective since
$\hext{\DrY{\LBA}}$ has many more objects. To remedy this, one can
introduce, by analogy with \cite[Cor. 6.4]{ee}, a $\PROP$ $\aDrY{\cP}$
describing the notion of \DYt module over a co--Poisson universal enveloping
algebra.

The presentation of $\aDrY{\cP}$ is tailored to describe the
subcategory of $\DrY{\LBA}$ generated by the $\cP$--module $S\pb$
and the object $[V]$. The latter is endowed with an action map $\mu
_{[V]}: S\pb\ten[V]\to[V]$ given by
\[
\mu_{[V]}=\bigoplus_{n\geq0}\pi_{[V]}^{(n)}\circ \Sym_{n}\ten\id_{[V]},
\]
where $\Sym_{n}:S^n\pb\to\pb^{\otimes n}$ is the inclusion, and the
maps $\pi_{[V]}^{(n)}:\pb^{\otimes n}\otimes[V]\to[V]$ are recursively
defined by $\pi_{[V]}^{(1)}=\pi_{[V]}$, with $\pi_{[V]}$ the action of
$\pb$ on $[V]$, and, for $n\geq 2$, $\pi_{[V]}^{(n)}=\pi_{[V]}\circ\id_{\pb}
\ten\pi_{[V]}^{(n-1)}$. It is also endowed with a family of \emph{coaction}
maps $\xi_{[V]}^{(n)}:[V]\to S\pb\ten[V]$
\[
\xi_{[V]}^{(n)}=i_n\circ\Sym_n\ten\id_{[V]}\circ(\pi_{[V]}^*)^{(n)},
\]
where $(\pi_{[V]}^*)^{(n)}:[V]\to\pb^{\otimes n}\otimes[V]$ is the
iterated coaction, $\Sym_n:\pb^{\otimes n}\to S^n\pb$ the projection,
and $i_n:S^n\pb\hookrightarrow S\pb$ the inclusion. The relations
satisfied by $\mu_{[V]}$ and $\xi_{[V]}^{(n)}$ are deduced from
those in $\DrY{\LBA}$.

\subsubsection{}

The functor $\wt{Q}$ restricts to a morphism of $\PROP$s $\aDrY
{\QUE}\stackrel{\sim}{\longrightarrow}\hext{\aDrY{\cP}}$ which is
essentially surjective by construction and fully faithful by Hensel's
Lemma. Finally, since the category of \DYt $\b$--modules is
equivalent to the category of $\aDrY{\cP}$--modules, \ie
\begin{align}
\nonumber
\hDrY{\b}{}&\simeq
\{\F\in\mathsf{Fun}_{\bf b}^{\ten}(\aDrY{\cP},\vect)\;|\; \F(S\pb)=S\b\},
\end{align}
one concludes that the pullback of $\wt{Q}$, and therefore the functor $\wt{\sff}_{\b}$
from $\hDrY{\b}{\Phi}$ to $\DrY{\EK\b}$, 
is an equivalence of braided tensor categories.


\section{Universal relative constructions}\label{se:rel prop}

In this section, we show that the quantisations of the Verma modules $\Lm,\Npv$
are isomorphic to their quantum counterparts $\Lmh,\Npvh$, thus proving Theorem
\ref{ss:relquantumverma2}. We also show that the constructions of Sections \S\ref
{s:Gamma}, \S\ref{s:Gammah} and \S\ref{s:gammahopf} can be realised in the
context of $\PROP$s, 
and are therefore functorial \wrt morphisms of split pairs of Lie bialgebras.

\subsection{Colored $\PROP$s}

The definition of $\PROP$ is easily generalised to allow a larger set of 
generating objects. A \emph{colored} $\PROP$ $\mathsf{P}$ is a
$\sfk$--linear, strict, symmetric monoidal category whose objects are
finite sequences over a set $\sfA$, \ie
\[\mathsf{Obj}(\mathsf{P})=\coprod_{n\geq0}\sfA^n\]
with tensor product given by concatenation of sequences, and tensor
unit given by the empty sequence. The notion of module over a 
colored $\PROP$ is easily generalised from \S\ref{ss:prop-intro}.

\subsection{$\PROP$ for split pairs of Lie bialgebras}\label{ss:plba}

Let $\PLBA$ be the $\PROP$ generated by two Lie bialgebra objects
$\pa,\pb$ and Lie bialgebra morphisms $i:\pa\to\pb$, $p:\pb\to\pa$
satisfying $p\circ i=\id_{\pa}$. 

The $\PROP$ $\PLBA$ is endowed with a pair of (strict) tensor functors
\[\funU, \funL:\LBA\to\PLBA\qquad\text{given by}
\qquad \funU[1]=\pb,\quad \funL[1]=\pa\]
and natural transformations $i:\funL\to\funU$, $p:\funU\to\funL$ such
that $p\circ i=\id_{\funL}$.
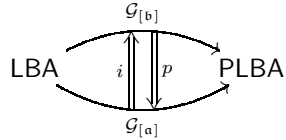
\[
\xymatrix{\LBA\ar@/^15pt/[rr]^{\funU}="U"\ar@/_15pt/[rr]_{\funL}="L"  & & 
\PLBA \ar@{=>}@<1ex> "U";"L"^p \ar@{=>}@<1ex> "L";"U"^i}
\]
Moreover, $\PLBA$ satisfies the following universal property: for any
symmetric tensor category $\C$ with the same property above as \PLBA,
there exists, up to a unique equivalence, a unique symmetric tensor
functor $\PLBA\to\C$ such that the following diagram commutes
\[
\xymatrix{\LBA\ar@/^15pt/[rr]^(.55){\funU}="U"\ar@/_15pt/[rr]_(.55){\funL}="L"  
\ar@/^30pt/[rrr] \ar@/_30pt/[rrr] & & 
\PLBA \ar@{=>}@<1ex> "U";"L"^p \ar@{=>}@<1ex> "L";"U"^i  \ar@{-->}[r]^{\exists!} & \C}
\]\\
\newcommand{\spLBA}{\G^{\PLBA}}

A module over $\PLBA$ in $\vectk$ is a split pair $(\b,\a)$ of Lie bialgebras
over $\sfk$, and a morphism $(\b,\a)\to(\b',\a')$ between two such pairs is a
pair of morphisms $(f,g)\in\LBA(\a,\a')\times\LBA(\b,\b')$ such that
the following diagrams commute
\[
\xymatrix{
\a'\ar[r]^{i'} & \b' \\
\a\ar[r]_{i}\ar[u]^{f} & \b\ar[u]_{g}
}
\qquad\qquad
\xymatrix{
\a' & \b'\ar[l]_{p'} \\
\a\ar[u]^{f} & \b\ar[l]^{p} \ar[u]_{g}
}\]

\subsection{$\PROP$ description of the Verma module $\Lm$}\label{ss:propL}

Let $\pi\in\End_{\PLBA}(\pb)$ be the idempotent $i\circ p$. The kernel
of $\pi$, $\ppm:=([\b],1-\pi)$, is an object in the Karoubi envelope of
$\PLBA$, which is a both a Lie ideal and coideal of $\pb$.

The module $\Lm$ introduced in \S \ref{ss:rel Verma} can be realised
in $\cKar{\PLBA}$. As an object, $\PLm$ is equal to $S\ppm$. The
structure of \DYt $\pb$--module is determined in the following way
\begin{itemize}
\item The action $\pi_{[\m]}$ of $\ppm$ is defined by the multiplication map
\[\bigoplus_{0\leq j\leq i+1}m_{1,i}^j: \ppm\ten S\ppm\to S\ppm\]
given by \BCH series.
\item The action $\pi_{[\a]}$ of $\pa$ is given by extending its action on
$[\m]$, which is an ideal in $[\b]$, to $T[\m]=\bigoplus_{n\geq 0}[\m]^
{\otimes n}$ and restricting it to $S[\m]$.
\item The action of $[\b]$ is determined by those of $[\a],[\m]$ since
$[\b]$ is the semi--direct product $[\a]\ltimes[\m]$, and $\pi_{[\a]},
\pi_{[\m]}$ satisfy the following identity 
\[\pi_{[\a]}\circ(\id_{[\a]}\ten\pi_{[\m]})-
\pi_{[\m]}\circ(\id_{[\m]}\ten\pi_{[\a]})\circ (1\,2)=
\pi_{[\m]}\circ([\,,\,]\ten\id_V)\]
as morphisms $[\a]\otimes[\m]\otimes\PLm\to\PLm$.
\item The coaction $\pi^*$ of $\pb$ on $\PLm$ is then uniquely
determined by requiring that it be
\begin{itemize}
\item trivial on $S^0\ppm$;
\item compatible with the action of the ideal $[\m]$ via the relation 
\[
\quad\qquad\pi^*\circ\pi_{[\m]}=
\id\ten\pi_{[\m]}\circ(1\,2)\circ \id\ten\pi^*-
\id\ten\pi_{[\m]}\circ\delta\ten\id+
[\,,\,]\ten \id\circ\id\ten \pi^*.
\]
\end{itemize}
\end{itemize}

When $\pa=0$, the description of $\PLm$ reproduces that of the
Verma module $\PMm$ given in \S\ref{sss:PMm}. More precisely,
since $0\hookrightarrow\pb$ is a split pair in $\cKar{\LBA}$, there
is a unique symmetric tensor functor $\G_{0,\pb}:\cKar{\PLBA}
\to\cKar{\LBA}$ mapping $\pa$ to $0$ and $\pb$ to $\pb$, and
$\G_{0,\pb}\PLm=\PMm$ as \DYt modules over $\pb$. Moreover,
the uniqueness of the coaction $\pi^*$ implies that if $(\a,\b)$ is
a split pair of Lie bialgebras, and $\G_{\a,\b}:\LBA\to \vectk$ a
corresponding realisation functor, $\G_{\a,\b}\PLm$ is the
\DYt module $\Lm=\ind_{\Lpp}^{\gb}\sfk$ introduced in \S\ref{ss:rel Verma}.

\subsection{Matched pairs of Lie bialgebras}\label{ss:majid}

In order to describe the \propic construction of the Verma 
module $\Npv$, we shall need the following notion from
\cite[Section 8.3]{majid-2}, which provides a generalisation 
of the notion of Drinfeld double.
\newcommand{\cond}{\coad}
\newcommand{\donc}{\coadd}

Two Lie algebras $(\c,[,]_{\c})$ and $(\d,[,]_{\d})$ form a
{\it matched pair} if there are maps
\[
\cond:\c\ten\d\to\d\aand\donc:\c\ten\d\to\c
\] 
such that
\begin{enumerate}
\item $\cond$ is a left action of $\c$ on $\d$, \ie
\[\cond\circ[,]_{\c}\ten\id=\cond\circ\id\ten\cond\circ(\id-(1\,2))\]
and $\donc$ is a right action of $\d$ on $\c$, \ie
\[\donc\circ\id\ten[,]_{\d}=\donc\circ\donc\ten\id\circ(\id-(2\,3));\]
\item $\donc, \cond$ satisfy the compatibility conditions
\begin{align*}
\donc\circ[,]_{\c}\ten\id
&=
[,]_{\c}\circ\donc\ten\id\circ(2\,3)+[,]_{\c}\circ\id\ten\donc+
\donc\circ\id\ten\cond\circ(\id-(1\,2))\\
\intertext{and}
\cond\circ\id\ten[,]_{\d}
&=
[,]_{\d}\circ\cond\ten\id+[,]_{\d}\circ\id\ten\cond\circ(1\,2)+
\cond\circ\donc\ten\id\circ(\id-(2\,3)).
\end{align*}
\end{enumerate}

The conditions (i)--(ii) are equivalent to the requirement that the
vector space $\c\oplus\d$ is endowed with a Lie bracket for which
$\c,\d$ are Lie subalgebras and, for $X\in\c$ and $Y\in\d$,
\[[X,Y]=X\cond Y+X\donc Y\]
With the above bracket, $\c\oplus\d$ is denoted by $\c\dcs\d$ and
called the \emph{double cross sum Lie algebra} of $\c,\d$.\\

\noindent\example
If $(\a,[,]_{\a},\delta_{\a})$ is a Lie bialgebra, the Lie algebras $(\a,[,]_{\a})$
and $(\a^*,\delta_{\a}^t)$ form a matched pair with respect to the coadjoint
action of $\a$ on $\a^*$ and the opposite coadjoint action of $\a^*$ on $\a$.
The corresponding double cross sum Lie algebra $\a\dcs\a^*$ is the Drinfeld
double of $\a$.

\subsection{Extended $\PLBA$}\label{ss:est-plba}

The definition of the Verma module $\Npv$ relies on the parabolic subalgebra
$\Lp_-=\gum\oplus\gdp=\Lmm\oplus\ga\subset\gb$, which cannot be realised
in $\PLBA$. We therefore introduce the colored $\PROP$ $\PLBAD$ obtained
by adding to the presentation of $\PLBA$ a Lie algebra object $\pad$
in matched pairing with the split pair $\pa\stackrel{i}{\longrightarrow}\pb\stackrel{p}{\longrightarrow}\pa$. 
Namely, we assume that $\pb$ and $\pad$
(and, respectively, $\pa$ and $\pad$) are endowed with mutual left and right
actions 
\[\coad_{\pb}:\pb\ten\pad\to\pad\aand\coadd_{\pb}:\pb\ten\pad\to\pb,\]
\[\coad_{\pa}:\pa\ten\pad\to\pad\aand\coadd_{\pa}:\pa\ten\pad\to\pa,\]
satisfying the compatibility conditions of \S\ref{ss:majid}. Further, we assume
that $i,p$ are morphisms of matched pairs, \ie
\[\coad_{\pb}\circ i\ten\id=\coad_{\pa}\aand
\coad_{\pb}=\coad_{\pa}\circ p\ten\id\]
as maps $\pa\otimes\pad\to\pad$ and $\pb\otimes\pad\to\pa$ respectively
(similarly for $\coadd_{\pb},\coadd_{\pa}$).
It follows that $(\pa\dcs\pad, \pb\dcs\pad)$ is a split pair of Lie
algebras. Additionally, we assume there is a morphism $r_{\pa}:[0]\to\pa\ten\pad$
satisfying the classical Yang--Baxter equation, together with the relations
\begin{eqnarray}
\delta_{\pa}\ten\id_{\pad}\circ r_{\pa}&=&[{r_{\pa}}_{13},{r_{\pa}}_{23}]\label{eq:r-mx},\\
\id\ten[,]_{\pad}\circ r_{\pa}\ten\id&=&\coadd_{\pa}\ten\id\circ\id\ten r_{\pa}\label{eq:r-mx-coadd},\\
\,[,]_{\pa}\ten\id\circ\id\ten r_{\pa}&=&\id\ten\coad_{\pa}\circ r_{\pa}\ten\id\label{eq:r-mx-coad},
\end{eqnarray}
as morphisms $[0]\to[\a]\otimes[\a]\otimes\pad$, $\pad\to\pa\ten\pad$, and 
$\pa\to\pa\ten\pad$, respectively. 

In particular, equation \eqref{eq:r-mx} encodes
the fact that the copairing $r_{[\a]}$ between $[\a]$ and $[\a^*]$
identifies the transpose of $\delta_{[\a]}$ with the bracket on $\pad$.
\footnote{The $\PROP$ $\PLBAD$ is a generalisation to the relative 
case of the Drinfeld double $\PROP$ $D_{\oplus}(\LBA)$ introduced
in \cite[Section 3]{EG}.}

Let $(\b,\a,\a^\sharp), (\b',\a',(\a^\sharp)')$ be two modules over $\PLBAD$.
Then a morphism of $\PLBAD$--modules $(\b,\a,\a^\sharp)\to(\b',\a',(\a^\sharp)')$ 
is a triple $(h,g,f)$ where $(h,g): (\b,\a)\to(\b',\a')$ is a morphism
of split pairs, $(g,f):(\a,(\a^\sharp))\to(\a',(\a^\sharp)')$ and $(h,f):(\b,\a^\sharp)\to
(\b',(\a^\sharp)')$ are morphisms of matched pairs, and $(h\otimes g)
(r_\a)=r_{\a'}$.

We shall call a module $(\b,\a,\a^\sharp)$ over $\PLBAD$ in $\vectk$ an 
{\it enhanced split pair} if $\a^\sharp=\a^*$ and $r_{\a}$ is the $r$--matrix of $\a$. 

\subsection{$\PROP$ description of the Verma module $\Npv$}\label{ss:propN}

The module $\pNpv$ can be realised on the object $\wh{S\pp}$, $\pp=\pb\oplus\pad$, 
formally added to $\cKarPLBAD$. 
In describing the structure of the \DYt module on $\pNpv$, we proceed
as in \S\ref{sss:Mpv-prop}. 
This is obtained as the propic solution of the universal property
\begin{equation}\label{eq:pNpv-univ-prop}
\xymatrix{
\Hom_{\b}^{\b}(V,\G\pNpv)
\ar@<2pt>[r]^(.6){\psi}
&\Hom_{{\m}}(V, \sfk)
\ar@<2pt>[l]^(.4){\phi}
},
\end{equation}
where $\psi$ denotes the composition with the projection $\varepsilon_{\pNpv}:\pNpv\to[0]$ and $\phi$ is the map
\[
\phi=\id\ten - \circ \sum_{n,m\geq 0}\Sym_n\ten\Sym_m\ten\id\circ\id^{\ten m}\ten\pi_V^{(m)}\circ\id^{\ten n}\ten r_{\a}^{(m)}\circ(\pi_V^*)^{(n)}.
\]
More precisely, $\pNpv$ is constructed on the object $\wh{S}\pp=\wh{S}\pb\ten\wh{S}\pad$
with the following \DYt structure.
\begin{itemize}
\item As in \S\ref{sss:Mpv-prop}, the coaction $\pi^*_{\pNpv}$ is obtained from
that of $\PMpv$, \ie $\pi^*_{\pNpv}=\pi^*_{\PMpv}\ten\id_{\wh{S}\pad}$.
In particular, the projection from $\pNpv$ to $\PMpv$ is a map of $\pb$--comodules. 
\item The formula for the action $\pi_{\pNpv}$ is obtained by imposing the following 
two conditions:
\begin{align}
\varepsilon_{\pNpv}\circ\pi_{\pNpv}=\varepsilon_{\pNpv}\circ\pi_{\pNpv}\circ p\ten\id
\end{align}
as maps from $\pb\ten\pNpv\to[0]$, and
\begin{align}
\phi(\varepsilon_{\pNpv})=\id_{\pNpv}
\end{align}
as maps from $\pNpv$ to $\wh{S}\pb\ten\wh{S}\pad=\pNpv$. 
One sees easily that the action $\pi^*_{\pNpv}$ exists and it is uniquely defined by these properties. 
Namely, one has
\begin{align*}
\pi_{\pNpv}&=\id_{\pNpv}\circ\pi_{\pNpv}\\
&=\left(\sum_{n,m\geq 0}
\id_{S^n\pb\ten S^m\pad}\ten\varepsilon\circ\Sym_n\ten\Sym_m\ten\id\circ\right.\\
&\left.\hspace{1cm}\circ\id^{\ten m}\ten\pi_{\pNpv}^{(m)}\circ\id^{\ten n}\ten r_{\a}^{(m)}
\circ(\pi_{\pNpv}^*)^{(n)}\right)
\circ\pi_{\pNpv}.
\end{align*}
By iterated application of the compatibility condition \eqref{eq:act-coact-lba} on 
$(\pi_{\pNpv}^*)^{(n)}\circ\pi_{\pNpv}$, relations \eqref{eq:action-lba} and \eqref{eq:r-mx-coad},
one obtains an explicit description of $\pi_{\pNpv}$, given exclusively in terms of morphisms in
$\PLBAD$. 
\item The description of the \DYt structure over $\pa\op$ is obtained with a similar argument.
\end{itemize}
When $\pa=0=\pad$, the description of $\pNpv$ reproduces that of the Verma modules $\PMpv$
given in \S\ref{sss:Mpv-prop}. More precisely, since $(\pb,0,0)$ is an enhanced split pair in $\LBA$,
there is a unique symmetric tensor functor $\G_{(\pb,0,0)}:\PLBAD\to\LBA$ mapping $\pa,\pad$ to $0$, 
$\pb$ to $\pb$, and $\G\pNpv=\PMpv$ as \DYt $\pb$--modules. Moreover, the uniqueness of
the action $\pi_{\Npv}$ implies that if $(\b,\a)$ is a split pair of Lie bialgebras and 
$\G_{(\b,\a,\a^*)}:\PLBAD\to\vectk$ a corresponding realisation functor, 
$\G_{(\b,\a,\a^*)}\pNpv$ is the \DYt $(\b,\a\op)$--module $\Npv$ introduced in \S\ref{ss:rel Verma}

\subsection{The propic construction of the twist}\label{ss:proptwist}

Let $\pDrY{\splbah}{\pa}{\Phi}$ and $\pDrY{\splbah}{\pb}{\Phi}$ be the
categories of deformed \DYt modules in the symmetric monoidal category
$\splbah$ over the Lie bialgebras $\pa$ and $\pb$ respectively.
By Frobenius reciprocity, for every $V,W\in\pDrY{\splbah}{\pb}{\Phi}$,
we get an isomorphism
\begin{align*}
\Hom_{\pb}^{\pb}(\PLm\ten V^0, \PNpv\ten W)\simeq\Hom_{\cKarPLBAD}(V,W).
\end{align*}
Let $\prelpsi{V}:\PLm\ten V^0\to\PNpv\ten V$ be the map corresponding to the identity
on $V$ and set $\preleta{V}=\prelpsi{V}\circ 1_-\ten\id_V$, where $1_-$ denotes the
inclusion of $[0]$ in $\PLm$.

We define a map $J\presped:V\ten W\to V\ten W$ by
\[
J\presped=1_+^{\ten 2}\ten\id^{\ten 2}\circ A_{\Phi}\circ\preleta{V}\ten\preleta{W},
\]
where $1_+$ is the projection from $\PNpv$ to $[0]$, and $A_{\Phi}$ is as in
\S\ref{s:tensor-structure}.

\begin{proposition}\hfill
\begin{enumerate}
\item The map $J\presped$ defines a tensor structure on the propic restriction
functor $\Res\presped:\hDrY{\pb}{\Phi}\to\hDrY{\pa}{\Phi}$ given by
\[\Res\presped(V,\pi,\pi^*)=
(V,\pi\circ
i\otimes\id_V,p\otimes\id_V\circ\pi^*)\]
\item Let $(\b,\a,\a^*)$ be an enhanced split pair with a realisation functor $\G=\G_{(\b,\a,\a^*)}:\PLBA^+\to\vectk$.
Then the tensor structure $\G(J_{[\a],[\b]})$ on the restriction
functor $\Res\resped:\hDrY{\b}{\Phi}\to\hDrY{\a}{\Phi}$ coincides, 
under the identification $\Res\resped\simeq F\resped$, with the
twist $J\resped$ constructed in Section \S\ref{s:Gamma}.
\end{enumerate}
\end{proposition}
\begin{pf}
(i) follows from a straighforward adaptation of \S\ref{ss:start tensor}--\S\ref{ss:end tensor}.
(ii) follows by direct inspection.
\end{pf}

It follows that the twist $J\resped$ is functorial with respect to morphisms
of enhanced split pairs. Namely, we have the following

\begin{corollary}
Let $[V_i]\in\hDrY{\pb}{\Phi}$, $i=1,2$, $(\b,\a,\a^*)$ and $(\b',\a',(\a^*)')$
two enhanced split pairs with realisation functors $\G,\G':\PLBAD\to\vect_\sfk$,
and set $V_i=\G[V_i], V_i'=\G'[V_i]$. 
Then for any $\rho\in\mathsf{Fun}^{\ten}(\G,\G')$ (or equivalently, for any 
morphism $(f,g,h):(\b,\a,\a^*)\stackrel{(f,g,h)}{\longrightarrow}(\b',\a',(\a')^*)$),
the following diagram is commutative
\[
\xymatrix@C=0.8in{
V'_1\ten V'_2\ar[r]^{J_{\a',\b'}} & 
V'_1\ten V'_2\\
V_1\ten V_2 \ar[u]^{\rho_{[V_1]\ten[V_2]}} \ar[r]_{J\resped} & 
V_1\ten V_2 \ar[u]_{\rho_{[V_1]\ten[V_2]}}
}
\]
\end{corollary}

\subsection{Quantisation of propic modules in $\PLBAD$}

The quantisation formulae given in \S\ref{sss:prop-twist}--\S\ref{sss:qpropmod}
have their analogues in $\cKarPLBAD$. Specifically, since the modules $\Mm^{\pb},{\Mpv}^{\pb}$
and $\Mm^{\pa},{\Mpv}^{\pa}$ have their own realisation in $\cKarPLBAD$, there are 
isomorphisms
\begin{equation}
\Hom_{\pb}^{\pb}(\Mm^{\pb}\ten V^0, {\Mpv}^{\pb}\ten W)\simeq\Hom_{\cKarPLBAD}(V,W),
\end{equation}
\begin{equation}
\Hom_{\pa}^{\pa}(\Mm^{\pa}\ten V^0, {\Mpv}^{\pa}\ten W)\simeq\Hom_{\cKarPLBAD}(V,W),
\end{equation}
for every \DYt $\pb$--modules $V,W$ in $\cKarPLBAD$, with 
analogous distinguished morphisms $\psi_V:\Mm^{\pb}\ten V\to{\Mpv}^{\pb}\ten V$
and $\eta_V=\psi_V\circ 1_-\ten\id_V$, and similarly for $\Mm^{\pa},{\Mpv}^{\pa}$. 
These allow to define an action
of $Q[B]=\Mm^{\pb}$ and $Q[A]=\Mm^{\pa}$ on $V$ 
\begin{equation}
\rho_V=1_+\ten 1_+\ten\id\circ\id\ten\psi_{V}\circ\Phi\circ\eta_M\ten\id
\end{equation}
and a coaction
\[
\rho_V^*=R^J\circ\uh\ten\id_V
\]
(similarly for $\pa$) providing a propic version $\wt{\sff}_{\pb}$, $\wt{\sff}_{\pa}$
of the equivalences $\noEKeq{\b},\noEKeq{\a}$ over the symmetric category $\hext{\cKarPLBAD}$.
In particular, the quantisation of \DYt $\pb$--modules in $\hext{\cKarPLBAD}$ is functorial
with respect to morphisms of enhanced split pairs. Namely we have the following 

\begin{corollary}\label{co:functoriality of quantisation}
Let $(\Vm,\pi,\pi^*)$ be a \DYt $\pb$--module in $\PLBA$ (resp. $\PLBAD$), 
$(\b,\a)$ and $(\b',\a')$ two split pairs with realisation functors $\G,\G':\PLBA\to\vectk$
(resp. $(\b,\a,\a^*)$ and $(\b',\a',(\a^*)')$
two enhanced split pairs with realisation functors $\G,\G':\PLBAD\to\vect_\sfk$),
and set
\[(V,\pi_V,\pi^*_V)=\G([V],\pi,\pi^*)\in\DrY{\b}
\quad\text{and}\quad
(V',\pi_{V'},\pi^*_{V'})=\G'([V],\pi,\pi^*)\in\DrY{\b'}.\]
Then, for any $\rho\in\mathsf{Fun}^{\ten}(\G,\G')$ (or equivalently,
for any morphism $(f,g):(\b,\a)\to(\b',\a')$  
(resp. $(f,g,h):(\b,\a,\a^*)\stackrel{(f,g,h)}{\longrightarrow}(\b',\a',(\a')^*)$)),
the following diagrams are commutative
\[\xymatrix@C=0.8in{
U_\hbar{\b'}\ten \sff_{\b'}(V')\ar[r]^(.55){\pi_V} & \sff_{\b'}(V')\\
\Uhb\ten \sff_{\b}(V)\ar[u]^{\rho_{S\pb}\ten\rho_{[V]}}\ar[r]_(.55){\pi_{V'}} & 
\sff_{\b}(V)\ar[u]_{\rho_{[V]}}
}
\]
and
\[
\xymatrix@C=0.8in{
\sff_{\b'}(V')\ar[r]^(.45){\pi^*_{V'}} &U_\hbar{\b'}\ten\sff_{\b'}(V')\\
\sff_{\b}(V)\ar[u]^{\rho_{[V]}}\ar[r]_(.45){\pi^*_V} & 
\Uhb\ten\sff_{\b}(V)\ar[u]_{\rho_{S\pb}\ten\rho_{[V]}}
}\]
\end{corollary}

\subsection{The isomorphism $\noEKeq{\b}(\Lm)\simeq\Lmh$}\label{ss:Liso}

We now prove part (i) of Theorem \ref{thm:quantization-LN}. Part
(ii) is proved in \S\ref{ss:iso N}.

By Lemma \ref{le:Fb admissible} and Proposition \ref{le:q rel verma},
the semiclassical limits of $\noEKeq{\b}(\Lm)$ and $\Lmh$ are both
equal to $\Lm$. It therefore suffices to construct an intertwiner
$\Lmh\to\noEKeq{\b}(\Lm)$ in $\DrY{\Uhb}$ whose reduction mod
$\hbar$ is the identity.
By the universal property of $\Lmh$ \eqref{eq:L-univ-p}, this amounts to constructing
a linear map $\ellh:\sfk\fml\to\noEKeq{\b}(\Lm)$ which intertwines
the action of $\Uha$ and the coaction of $\Uhb$, and whose reduction
mod $\hbar$ is the inclusion $\ell$ of the generating vector $1_-\in
\Lm$.

To this end, it is instructive to note that the intertwining properties
of $\ell$ correspond to the commutativity of the diagrams
\[
\xymatrix{
\b\otimes\Lm \ar[r]^(.55){\pi} &  \Lm\\
\a\otimes\sfk \ar[u]^{i\otimes\ell}\ar[r]_(.55){0\otimes\id} &  \sfk\ar[u]_{\ell}
}
\qquad\qquad
\xymatrix{
\Lm  \ar[r]^{\pi^*} &  \b\otimes\Lm\\
\sfk \ar[u]^{\ell}\ar[r]_{0\otimes\id} &  \a\otimes\sfk\ar[u]_{i\otimes\ell}
}
\]
Moreover, if the trivial module $\sfk\in\DrY{\a}$ is thought of as the
Verma module $\Lm$ corresponding to the split pair $\a\hookrightarrow
\a$, these diagrams arise from the functoriality of $\Lm$ \wrt the
morphism of split pairs $(i,\id):(\a,\a)\to(\b,\a)$.

Similarly, regarding $\sfk\fml$ as the trivial \DYt module over $\Uha$,
the required interwining properties of the map $\ellh$ correspond to
the diagrams
\[
\xymatrix{
\Uhb\otimes\noEKeq{\b}(\Lm)  \ar[r]^(.55){\pi} &  \noEKeq{\b}(\Lm)\\
\Uha\otimes\sfk\fml  \ar[u]^{i_\hbar\otimes\ellh}\ar[r]_(.55){\epsilon\otimes\id} &  \sfk\fml\ar[u]_{\ellh}
}
\qquad\qquad
\xymatrix{
\noEKeq{\b}(\Lm)  \ar[r]^(.4){\pi^*} &  \noEKeq{\b}(\Lm)\otimes\Uhb\\
\sfk\fml \ar[u]^{\ellh}\ar[r]_(.45){1\otimes\id} &  \Uha\otimes\sfk\fml\ar[u]_{i_\hbar\otimes\ellh}
}
\]
Since $\DrY{\Uha}\ni\sfk\fml$ is equal to $\noEKeq{\a}(\sfk\fml)$,
the existence of $\ellh$, the fact that it reduces to $\ell$ mod
$\hbar$, and the commutativity of the above diagrams follow
from the functoriality of the Tannakian lift of $\Lm$ \wrt the morphism
of split pairs $(i,\id):(\a,\a)\to(\b,\a)$
(Corollary \ref{co:functoriality of quantisation}). Moreover, $\ellh$
is a morphism of coalgebras since it maps the group--like generating
vector of $\Lmh$ to a group--like element of $\noEKeq{\b}(\Lm)$.
\qed

\subsection{The isomorphism $\noEKeq{\b}\circ\noEKeq{\a}(\Npv)\simeq\Npvh$}\label{ss:iso N}
\newcommand{\enh}{n_{\hbar}}
\newcommand{\aNpv}{(\Npv)_{\a}}
We adopt the same strategy of \S\ref{ss:Liso} to prove 
part (ii) of Theorem \ref{thm:quantization-LN}.

By  Proposition \ref{le:q rel verma},
the semiclassical limits of $\Npvh$ is
equal to $\Npv$ as \DYt $(\b,\a)$--bimodules. 
Similarly, combining Lemma \ref{le:Fb admissible} and Proposition \ref{ss:quant-bimod},
one concludes that the semiclassical limit of $\noEKeq{\b}\circ\noEKeq{\a}(\Npv)\simeq\noEKeq{\b\oplus\a^{\scs\op}}(\Npv)$ is
equal to $\Npv$ as \DYt $(\b,\a)$--bimodules.
It therefore suffices to construct an intertwiner
$\noEKeq{\b}\circ\noEKeq{\a}(\Npv)\to\Npvh$ in $\DrY{\Uhb}$ whose reduction mod
$\hbar$ is the identity.
By the universal property of $\Npvh$ \eqref{eq:N-univ-prop-int}, 
this amounts to constructing a linear map 
$\enh:\noEKeq{\b}\circ\noEKeq{\a}(\Npv)\to p_{\hbar}^*\aNpvh$ which intertwines
the action of $\Uhb$ and coaction of $\Uha$, and whose reduction
mod $\hbar$ is the projection $n$ of $\Npv$ onto $\aNpv$ (induced by 
the projection of (topological) Lie bialgebras $\Lp=\b\oplus\a^*\to\a\oplus\a^*$).

We observe as in \S\ref{ss:Liso} that the intertwining properties of $n$ corresponds to 
the commutativity of the diagrams 
\[
\xymatrix{
\b\otimes\Npv \ar[r]^(.55){\pi} \ar[d]_{p\otimes n}&  \Npv \ar[d]^{n}\\
\a\otimes\aNpv \ar[r]_(.55){\pi_{\a}} &  \aNpv
}
\qquad\qquad
\xymatrix{
\Npv  \ar[r]^{\pi^*} \ar[d]_{n}&  \b\otimes\Npv\ar[d]^{p\otimes n}\\
\aNpv \ar[r]_(.45){\pi_{\a}^*} &  \a\otimes\aNpv
}
\]
and analogue diagrams for the right action and coaction of $\a$.

If the module $\aNpv\in\DrY{\a}$ is thought of as the Verma module $\Npv$
corresponding to the split pair $\a\hookrightarrow\a$, there diagrams arise
from the functoriality of $\Npv$ with respect to the morphism of enhanced 
split pairs $(p,\id,\id):(\b,\a,\a^*)\to(\a,\a,\a^*)$.

The required intertwining properties of $\enh$ correspond to the commutativity of
the diagrams
\[
\xymatrix{
\EK\b\otimes\noEKeq{\b}\circ\noEKeq{\a}(\Npv) \ar[r]^(.55){\pi} \ar[d]_{p_{\hbar}\otimes \enh}&  
\noEKeq{\b}\circ\noEKeq{\a}(\Npv) \ar[d]^{\enh}\\
\EK\a\otimes\aNpvh \ar[r]_(.55){\pi_{\a}} & \aNpvh
}
\quad
\xymatrix{
\noEKeq{\b}\circ\noEKeq{\a}(\Npv)  \ar[r]^(.45){\pi^*} \ar[d]_{\enh}&  
\EK\b\otimes\noEKeq{\b}\circ\noEKeq{\a}(\Npv)\ar[d]^{p_{\hbar}\otimes\enh}\\
\aNpvh \ar[r]_(.4){\pi_{\a}^*} &  
\EK\a\otimes\aNpvh
}
\]
and their analogues for the right action and coaction of $\EK\a$.
The existence of $\enh$, the fact that it reduces to $n \mod\hbar$, and the commutativity
of the above diagrams follows from the functoriality of the Tannakian lift of $\Npv$
with respect to the morphism of enhanced split pairs $(p,\id,\id):(\b,\a,\a^*)\to(\a,\a,\a^*)$.
Finally, $\enh$ is a morphism of algebras, since the algebra structure is uniquely determined 
by the universal property.
\qed

\subsection{Propicity}

The results from Sections \S\ref{s:Gammah} and \S\ref{s:gammahopf}
have analogous counterparts in the category $\cKarPLBAD$, since they rely
exclusively on the propic realisations of the relative Verma modules $\Lm$, $\Npv$ 
in $\cKarPLBAD$. In particular, we obtain a propic version of Theorem \ref{thm:transf}. 

\begin{theorem}\label{thm:L-prop}\hfill
\begin{enumerate}
\item The object $[\Lm]$ has a natural structure of Hopf algebra
object in $\hDrY{\pa}{\Phi}$ with product
\[
m_{[L]}=1_+\ten1_+\ten\id\circ\id\ten\phi_{[L]}\circ\Phi\circ\eta_{[L]}\ten\id
\]
and coproduct
\[
\Delta_{[L]}=J\presped^{-1}\circ\Delta_0,
\]
where $\Delta_0$ is the standard coproduct on $[\Lm]=S\ppm$ 
and $J\resped$ is the twist constructed in \S\ref{ss:proptwist}.
\item The quantised module $\wt{\sff}_{\pa}[\Lm]$ is a Hopf algebra
object in $\DrY{Q[A]}$ and the Radford biproduct $\HBp=\wt{\sff}_{\a}[\Lm]\star Q[A]$
is a quantisation of the Lie bialgebra $\pb$.
\item $\HBp$ splits over $\EK\pa$ and there is an isomorphism of split 
pairs of Hopf algebras in $\splbah$
\[
u\presped: (\EK\pb, \EK\pa)\to(\HBp,\EK\pa).
\]
\end{enumerate}
\end{theorem}

The proof follows the same steps of Section \S\ref{s:gammahopf}
and it is therefore omitted.


\section{Alternative constructions}\label{se:Severa}

In this section we discuss the relation between our construction 
and the quantisation functor described by P. \v{S}evera in \cite{sev},
whose construction and general principle we briefly review in \S\ref{ss:Sev-1}.

In \S\ref{ss:Sev-2}, we rephrase the case of a coisotropic subalgebra \cite[Example 3]{sev}
in the language of Section \S\ref{s:Gamma}, and we show that this allows to construct a relative twist in the
$\PROP$ $\hext{\cKar{\PLBA}}$. From the identification of the two constructions, it
follows that the universal twist $J\resped$, constructed in Sections \S\ref{s:Gamma} and \S\ref{se:rel prop}
in the $\PROP$ $\hext{\cKar{\PLBA}^+}$,
is gauge equivalent to a universal twist in $\hext{\cKar{\PLBA}}$,
a result which is needed in \cite{ATL2}.

\subsection{Quantisation of Lie bialgebras revisited}\label{ss:Sev-1}

In \cite{sev}, \v{S}evera provides an alternative construction of a
universal quantisation functor, based on the following observations.

\subsubsection{}
Let $(\C,\ten_{\C}, \bf{1}_{\C}, \Phi_{\C}, \beta_{\C})$ be a braided monoidal category.
For any cocommutative coalgebra object $(M,\Delta_M,\epsilon_M)\in\mathsf{CoCoAlg}(\C)$, the 
functor $M\ten_{\C}\bullet:\C\to\C$ is endowed with a canonical (lax)
\footnote{\ie $J_{VW}$ is not necessarily invertible.}
tensor structure
\[
J_{VW}:M\ten_{\C}(V\ten_{\C} W)\to (M\ten_{\C} V)\ten_{\C}(M\ten_{\C} W),
\qquad
J_{VW}=A_{\C}\circ\id\ten\Delta,
\]
where $A_{\C}$ is the canonical isomorphism in $\C$ between
$(M\ten_{\C} M)\ten_{\C}(V\ten_{\C} W)$ and $(M\ten_{\C} V)\ten_{\C}(M\ten_{\C} W)$.

\subsubsection{}
Let $(\D,\ten_{\D},\bf{1}_{\D}, \Phi_{\D}, \beta_{\D})$ be a braided monoidal category
and $G:\C\to\D$ a braided monoidal functor.
Then it is possible to describe necessary and sufficient conditions for the composition 
functor $G_M=G\circ M\ten_{\C}\bullet:\C\to\D$ to be a tensor functor. These conditions 
motivate the notion of \emph{$M$--adapted functor} \cite[Definition 1]{sev}. Specifically,
$G$ is $M$--adapted if
\begin{enumerate}
\item the composition $u_M$
\begin{equation}\label{eq:adapt-1}
G(M)\stackrel{G(\epsilon_M)}{\longrightarrow} G(\bf{1}_{\C})\simeq\bf{1}_{\D}
\end{equation}
is an isomorphism;
\item for every $V,W\in\C$, the composition $\tau_{VW}$
\begin{equation}\label{eq:adapt-2}
\xymatrix@C=.7in{
G((V\ten_{\C}M)\ten_{\C}W)\ar[r]^(.45){G(\id\ten\Delta_M\ten\id)}\ar[dr]_{\tau_{VW}}&
G((V\ten_{\C} (M\ten_{\C} M)\ten_{\C} W)\ar[d]^{\psi_{VW}} \\
& G(V\ten_{\C} M)\ten_{\D}G(M\ten_{\C}W)
}
\end{equation}
where $\psi_{VW}$ is the composition of the tensor structure on 
$G$ with the canonical isomorphism in $\C$ between the bracketing 
$\bullet (\bullet\bullet)\bullet$ and $(\bullet\bullet)(\bullet\bullet)$,
is an isomorphism.
\end{enumerate}
It follows that $G$ is $M$--adapted if and only if $G_M$ is a tensor functor.

\subsubsection{}\label{sss:sev-hopf}
For any $M$--adapted functor $G:\C\to\D$, $G_M(M)$ has a natural structure of
Hopf algebra in $\D$. Specifically, 
\begin{enumerate}
\item the coalgebra structure is induced by that of $M$, with coproduct
$G_M(M)\to G_M(M)\ten G_M(M)$
\begin{equation}\label{eq-sev-coprod}
J_G\circ G(A_{\C}\circ\Delta_M\ten\Delta_M)
\end{equation}
and counit $G(\epsilon_M\ten\epsilon_M)$;
\item the algebra structure is defined by the product $G_M(M)\ten G_M(M)\to G_M(M)$
\begin{equation}\label{eq-sev-prod}
G(\id\ten\epsilon_M\ten\id)\circ\tau_{MM}^{-1}
\end{equation}
and unit $G(\Delta_M)\circ u_M^{-1}$;
\item finally, the antipode is given by $G((\beta_{\C})_{MM}^{-1})$.
\end{enumerate}

Moreover $G_M(M)$ acts and coacts on any $G_M(V)$, $V\in\C$, with action
\begin{equation}\label{eq:sev-act}
G(\id\ten\epsilon_M\ten\id)\circ\tau_{MV}^{-1}
\end{equation}
and coaction
\begin{equation}\label{eq:sev-coact}
R_{MV}\circ 1_{G_M(M)}\ten \id
\end{equation}
where $R_{MV}$ is defined by
\[
\xymatrix{
G_M(M)\ten G_M(V)\ar[r]^{\beta_{\D}^{-1}R_{MV}} \ar[d]_{J_{G_M}} & 
G_M(V)\ten G_M(M)\ar[d]^{J_{G_M}}\\
G_M(M\ten V) \ar[r]^{G_M(\beta_{MV})} & G_M(V\ten M)
}
\]
One verifies that \eqref{eq:sev-act} and \eqref{eq:sev-coact} are compatible and 
therefore $G_M$ factors through the category of \DYt $G_M(M)$--modules.

\subsubsection{}
Finally, this construction applies to the case of Lie bialgebras. Set $\D=\Kvect$, $\C=\hDrY{\b}{\Phi}$, and consider 
the cocommutative coalgebra object in $\hDrY{\b}{\Phi}$ given by the Verma module $M=\Mm$.
One can easily check that the functor of coinvariants $G:\hDrY{\b}{\Phi}\to\Kvect$, $G(V)=V_{\b}$ 
is a $M$--adapted lax braided tensor functor. Moreover, there is a natural identification
\begin{equation}\label{eq:sev-isom}
\phi_V: G_M(V)=(M\ten V)_{\b}\to V,
\end{equation}
where $\phi_V([x\ten v])=S_0(x)v$ and $\phi_V^{-1}(v)=[1\ten v]$.
It is easy to check that the Hopf algebra $G_{M}(M)$ is a quantisation of the Lie bialgebra $\b$.\\

This is easily reproduced in the $\PROP$ $\slbah$, giving rise
to another universal quantisation functor.
The advantage with respect to the Etingof--Kazhdan
quantisation is that the construction of $G_M(M)$ does not require to consider \emph{topological}
\DYt modules. On the other hand, one can easily see that this is 
equivalent to the Etingof--Kazhdan functor on the category of discrete 
\DYt $\b$--modules, as for any discrete $V\in\hDrY{\b}{\Phi}$, 
there is a canonical isomorphism of tensor functors
\[
G_M(V)\simeq F_{\b}(V^*)^*,
\]
given by the identifications
\begin{align*}
(M\ten V)_{\b}\simeq\Hom_{\b}^{\b}(\Mm\ten V, \Mpv)^*\simeq\Hom_{\b}^{\b}(\Mm, \Mpv\ten V^*)^*.
\end{align*}

\subsection{Tensor structure on the restriction functor}\label{ss:Sev-2}
The construction of the fiber functor $G_M:\hDrY{\b}{\Phi}\to\vectk$ generalises
to the relative case obtaining analogous results to Theorems \ref{th:Gamma} and \ref{thm:braidhopf}.

\subsubsection{}
For any split pair $\a\to\b\to\a$, one can consider
the \DYt $\b$--module $L=S\Lmms$. This is endowed with a canonical
structure of cocommutative coalgebra in $\hDrY{\b}{\Phi}$, so that the functor 
$L\ten\bullet:\hDrY{\b}{\Phi}\to\hDrY{\b}{\Phi}$
is naturally endowed with a lax tensor structure. 
Then one replaces the functor $G_{\b}$ of $\b$--coinvariants with
\[
G_{\Lmms}:\hDrY{\b}{\Phi}\to\hDrY{\a}{\Phi},
\qquad
G_{\Lmms}(V)=V_{\Lmms}.
\]
One sees immediately that $G_{\Lmms}$ is a lax braided tensor functor, which
is adapted to $L$. Moreover, there is a canonical isomorphism
\begin{equation}\label{eq:sev-rel-isom}
\phi_V: G_{\Lmms}(L\ten V)\to V,
\end{equation}
where $\phi_V([x\ten v])=S_0(x)v$ and $\phi_V^{-1}(v)=[1\ten v]$.
It follows that
\begin{enumerate}
\item  the functor $G_L(V)=(L\ten V)_{\Lmms}$ from $\hDrY{\b}{\Phi}$
to $\hDrY{\a}{\Phi}$ is isomorphic to $\Res\resped$ and is naturally 
endowed with a tensor structure;
\item $G_L(L)$ has a Hopf algebra structure defined by the formulae
\eqref{eq-sev-prod}, \eqref{eq-sev-coprod}, corresponding to $L$ and $G_{\Lmms}$;
\item $G_L(L)$ acts and coacts by \eqref{eq:sev-act}, \eqref{eq:sev-coact}, 
on any $G_L(V)$, $V\in\hDrY{\b}{\Phi}$ and the functor $G_L$
factors through the category of \DYt $G_L(L)$--modules in $\hDrY{\a}{\Phi}$.
\end{enumerate}

As in the case $\a=0$, restricted to the category
of discrete \DYt $\b$--modules, there is an isomorphism
of tensor functors
\begin{equation}\label{eq:sev-rel-ident}
G_L(V)\simeq F\resped(V^*)^*,
\end{equation}
given by the identifications
\begin{align*}
(L\ten V)_{\m}\simeq\Hom_{\b}^{\b}(\Lm\ten V, \Npv)^*\simeq\Hom_{\b}^{\b}(\Lm, \Npv\ten V^*)^*.
\end{align*}

\subsubsection{}
The tensor functor $G_L$ has an obvious propic realisation in $\hext{\cKar{\PLBA}}$, which
relies on the identification \eqref{eq:sev-rel-isom} and leads to an analogue of Theorem 
\ref{thm:L-prop} $(i)$--$(iii)$. 
Namely, for every $V,W\in\hDrY{\pb}{\Phi}$, one defines the 
twist $J_{G_L}$ as the composition
\begin{equation}
\xymatrix@C=0.5in{
V\ten W \ar[r]^(.45){\phi^{-1}_{V\ten W}} &
L\ten V\ten W \ar[r]^{A_{\Phi}\circ\Delta\ten\id^{\ten 2}} &
L\ten V\ten L\ten W \ar[r]^(.6){\phi_V\ten\phi_W} &
V\ten W,
}
\end{equation}
where $A_{\Phi}$ denotes the usual isomorphism from $(L\ten L)\ten (V\ten W)$
to $(L\ten V)\ten (L\ten W)$.
The Hopf algebra structure on $G_L(L)$, its action and coaction on any \DYt
$\pb$--module have their analogues in $\hext{\cKar{\PLBA}}$. 

In particular, from \eqref{eq:sev-rel-ident} we get the following

\begin{proposition}
The tensor structure $J\resped$, constructed in Sections \S\ref{s:Gamma} and \S\ref{se:rel prop}, is 
gauge equivalent to a universal relative tensor structure from $\hext{\cKar{\PLBA}}$.
\end{proposition}

Finally, we point out that the proof of Theorems \ref{th:first main} and \ref{th:second main} 
does not get any simpler in this context.
Namely, in order to obtain an alternative proof, it would be necessary to construct a natural 
transformation of tensor functors
\[
\xymatrix{
\hDrY{\b}{\Phi}\ar[r]^{\noEKeq{\b}} \ar[d]_{\Res\resped} & 
\DrY{\EK\b} \ar[d]^{\Res_{\EK\a,\EK\b}}\ar@{=>}[dl]_{v\resped}\\
\hDrY{\a}{\Phi}\ar[r]_{\noEKeq{\a}} & \DrY{\EK\a}
}
\]
where the tensor structure on the restriction functors $\Res\resped$ and $\Res_{\EK\a,\EK\b}$ are
now induced by the identifications
\[
\Res\resped(V)\simeq(L\ten V)_{\m}
\aand
\Res_{\EK\a,\EK\b}(\V)\simeq(L^{\hbar}\ten\V)_{L^{\hbar}}.
\]
Adapting the proof of Theorem \ref{ss:nat-trans}, it would be necessary to show that 
the spaces of coinvariants are preserved by the quantisation as \DYt modules. 
The most obvious strategy to approach the problem relies on the equivalences $\noEKeq{\b}$
and $\noEKeq{\a}$ and therefore on the identifications
\begin{eqnarray*}
(L\ten V)_{\m}&\simeq&\Hom_{\b}^{\b}(\Lm,\Npv\ten V^*)^*,\\
(L^{\hbar}\ten\V)_{L^{\hbar}}&\simeq&\Hom_{\EK\b}^{\EK\b}(\Lmh,\Npvh\ten\V^*)^*.
\end{eqnarray*}
Ultimately, it becomes necessary, exactly as in \S\ref{ss:nat-trans}, to identify 
the quantisation of the relative Verma modules with their quantum analogues, 
producing the same proof we followed in Section \S\ref{s:Gammah}.

\appendix

\section{Quantum double and Drinfeld--Yetter modules}\label{s:app}

Let $B$ be a \fd Hopf algebra or a quantised enveloping algebra.
We review in this section the construction of the quantum double 
of $B$ as a double cross product of Hopf algebras, following \cite
[\S7.2]{majid-2}. We also describe the equivalence of braided tensor
categories between (admissible) \DYt modules over $B$ and
modules over its quantum double.

\subsection{Double cross product Hopf algebras \cite{majid-2}}\label{ss:app-1}

Two Hopf algebras $(B, m_B,$ $\iota_B, \Delta_B, \varepsilon_B, S_B)$
and $(C, m_C, \iota_C, \Delta_C, \varepsilon_C, S_C)$ form a \emph
{matched pair} if they are endowed with maps
\[
\BonC :B\ten C\to C\aand \ConB :B\ten C\to B
\]
such that
\begin{enumerate}
\item $(C, \BonC)$ is a left $B$--module coalgebra, \ie
\begin{equation*}\label{eq:maj-mp-1}
\Delta_C\circ \BonC = 
\BonC\ten\BonC\circ(2\,3)\circ\Delta_B\ten\Delta_C;
\end{equation*}
\item $(B, \ConB)$ is a right $C$--module coalgebra, \ie
\begin{equation*}\label{eq:maj-mp-2}
\Delta_B\circ \ConB = 
\ConB\ten\ConB\circ(2\,3)\circ\Delta_B\ten\Delta_C;
\end{equation*}
\item $\ConB$ and $\BonC$ are compatible, respectively, with $m_B$ and $m_C$, \ie 
\begin{eqnarray*}
\label{eq:maj-mp-3}
\ConB\circ m_B\ten\id&=&m_B\circ\ConB\ten\id\circ\id\ten\BonC\ten\ConB\circ(3\,4)\circ\id\ten\Delta_B\ten\Delta_C,\\
\label{eq:maj-mp-4}
\BonC\circ \id\ten m_C&=&m_C\circ\id\ten\BonC\circ\ConB\ten\BonC\ten\id\circ(2\,3)\circ\Delta_B\ten\Delta_C\ten\id;
\end{eqnarray*}
\item the units $\iota_B$ and $\iota_C$ are module maps, \ie
\begin{eqnarray*}\label{eq:maj-mp-5}
\ConB\circ\iota_B\ten\id&=&\iota_B\ten\varepsilon_C,\\
\label{eq:maj-mp-6}
\BonC\circ\id\ten\iota_C&=&\varepsilon_B\ten\iota_C;
\end{eqnarray*}
\item finally, $\ConB$ and $\BonC$ satisfy the following compatibility condition:
\begin{equation*}\label{eq:maj-mp-7}
\ConB\ten\BonC\circ(2\,3)\circ\Delta_B\ten\Delta_C =
\ConB\ten\BonC\circ(2\,3)\circ\Delta^{21}_B\ten\Delta^{21}_C.
\end{equation*}
\end{enumerate}

Any matched pair $(B,C,\BonC,\ConB)$ gives rise to a
Hopf algebra structure on $B\ten C$, called the \emph{double cross
product of $B$ and $C$}, and denoted $\dcp{B}{C}$. The product is
defined by requiring that $\dcp{B}{C}$ contains $B,C$ as subalgebras
under the natural inclusions $i_B:B\to B\ten C$ and $i_C:C\to B\ten
C$, and
\begin{eqnarray*}
\label{eq:maj-dcp-1}
m_{B,C}\circ i_B\ten i_C&=& id_{B\ten C},\\
\label{eq:maj-dcp-2}
m_{B,C}\circ i_C\ten i_B&=&\ConB\ten\BonC\circ(2\,3)\circ\Delta_B\ten\Delta_C\circ(1\,2).
\end{eqnarray*}
The remaining structure is defined as follows: 
\begin{eqnarray*}
\iota_{B,C}&=&\iota_B\ten\iota_C \label{eq:maj-dcp-4},\\
\Delta_{B,C}&=&(2\,3)\circ\Delta_B\ten\Delta_C \label{eq:maj-dcp-3},\\
\varepsilon_{B,C}&=&\varepsilon_B\ten\varepsilon_C\label{eq:maj-dcp-5},\\
S_{B,C}&=&m_{B,C}\circ i_C\ten i_B \circ S_C\ten S_B\circ (1\,2) \label{eq:maj-dcp-6}.
\end{eqnarray*}
In particular, $\dcp{B}{C}$ contains $B$ and $C$ as Hopf subalgebras.


\subsection{Quantum double}\label{ss:app-2}
Let $B$ be a finite--dimensional Hopf algebra. The quantum double $DB$  
(cf.  \cite{drin-2} and \S\ref{ss:DY-qD}) has a natural description as a double cross 
product Hopf algebra. 
Namely, let $\BonC_B:B\ten B^*\to B^*$ be the left coadjoint action of $B$ on $B^*$, \ie
\begin{equation*}
\BonC_B=\iip{}{}\ten\id\circ\id\ten m_{B^*}\ten\id\circ (2\,3\,4)\circ
\id\ten S^{-1}_{B^*}\ten \id^{\ten 2}\circ\id\ten\Delta_{B^*}^{(3)},
\end{equation*}
and $\ConB_B:B\ten B^*\to B$ the right coadjoint action of $B^*$ on $B$, \ie
\begin{equation*}
\ConB_B=\id\ten\iip{}{}\circ\id\ten m_{B}\ten\id\circ (3\,2\,1)\circ
\id^{\ten 2}\ten S^{-1}_{B}\ten \id\circ\Delta_{B}^{(3)}\ten\id.
\end{equation*}
Then, the product on $DB$ given by \eqref{eq:q-double-mult}  reads
\begin{equation}\label{eq:DB-prod-2}
m_{DB}\circ i_{B^*}\ten i_B=
\ConB_B\ten\BonC_B\circ(2\,3)\circ\Delta_B\ten\Delta^{21}_{B^*}\circ(1\,2).
\end{equation}
In particular, it follows that $(B, B^\circ, \BonC_B,\ConB_B)$, where $B^{\circ}
={B^*}^{\scs\operatorname{cop}}$,
is a matched pair of Hopf algebras and $DB=\dcp{B}{B^{\circ}}$ is the 
associated double cross product Hopf algebra.


\subsection{Drinfeld--Yetter modules revisited}\label{ss:app-3}
\newcommand{\btd}{\blacktriangledown}
\newcommand{\btu}{\blacktriangle}
Let $\btd_B: B\ten B\to B$ and $\btu_B:B\to B\ten B$ be, respectively, the adjoint action and 
the adjoint coaction of $B$ on itself (cf. \eqref{eq:adj-action}, \eqref{eq:adj-coact}), \ie
\begin{eqnarray*}
\btu_B&=&m_B\ten\id\circ(1\,2\,3)\circ S^{-1}_B\ten\id^{\ten 2}\circ\Delta_B^{(3)},\\
\btd_B&=&m_B^{(3)}\circ(3\,2\,1)\circ S^{-1}_B\ten\id^{\ten 2}\circ\Delta_B\ten\id.
\end{eqnarray*}
Then, for every Drinfeld--Yetter $B$--module $(V, \pi_V,\pi_V^*)$, the 
compatibility condition \eqref{eq:DY Hopf} between the action and the coaction reads
\begin{equation}\label{eq:DY-comp-2}
\pi_V^*\circ\pi_V=m_B\ten\pi_V\circ(3\,2\,1)\circ\btu_B\ten\btd_B\ten\id\circ\Delta_B\ten\pi_V^*.
\end{equation}


\subsection{Drinfeld--Yetter modules and quantum double}\label{ss:app-4}
We mentioned in \S\ref{ss:DY-qD} that the category of $DB$--modules is 
equivalent to the category of Drinfeld--Yetter $B$--modules. 
This follows, in particular, from the duality between the maps $\ConB_B$ (resp. $\BonC_B$) and 
$\btu_B$ (resp. $\btd_B$), appearing in the formula \eqref{eq:DB-prod-2} for the product of $DB$
and in the compatibility condition \eqref{eq:DY-comp-2} for Drinfeld--Yetter $B$--modules.
Indeed, the coadjoint action of $B^*$ on $B$ corresponds to the adjoint coaction of $B$ on $B$, \ie
\begin{eqnarray}
\ConB_B&=&\id_B\ten\iip{}{}\circ(2\,3)\circ\btu_B\ten\id_{B^*}\label{eq:coadj-dual-1},\\
\btu_B&=&(1\,2)\circ\ConB_B\ten\id_{B}\circ\id_B\ten R_B^{21}\label{eq:coadj-dual-2},
\end{eqnarray}
where $\iip{}{}$ denotes the pairing between $B$ and $B^*$ and $R_B\in B\ten B^*$ is the $R$--matrix. 
Similarly, the coadjoint action of $B$ on $B^*$ corresponds to the adjoint
action of $B$ on $B$, \ie
\begin{eqnarray}
\iip{}{}^{21}\circ\id_{B^*}\ten\btd_B&=&\iip{}{}^{21}\circ\BonC_B\ten\id_B\circ(1\,2)\label{eq:coadj-dual-3},\\
\btd_B\ten\id_{B^*}\circ\id_B\ten R_B&=&\id_B\ten\BonC_B\circ(2\,3)\circ R_B\ten\id_B\label{eq:coadj-dual-4}.
\end{eqnarray}

\begin{theorem}\label{thm:DY-qd-eq}
There is a canonical equivalence of braided tensor categories
\[
\xymatrix{
\DrY{B} \ar@<2pt>[r]^(.4){\Xi} \ar@<-2pt>@{<-}[r]_(.4){\Theta}
&\Rep DB
}
\]
where
\begin{enumerate}
\item  for any Drinfeld--Yetter $B$--module $(V,\pi_V,\pi^*_V)$, 
$\Xi(V,\pi_V,\pi^*_V)=(V,\xi_V)$ with
\begin{equation*}\label{eq:DY-funct-1}
\xi_V=\pi_V\circ\id_B\ten\iip{}{}^{21}\ten\id_V\circ\id_B\ten\id_{B^*}\ten\pi_V^*;
\end{equation*}
\item for any $DB$--module $(V,\xi_V)$, $\Theta(V,\xi_V)=(V, \pi_V, \pi_V^*)$ with
\begin{eqnarray*}
\pi_V&=&\xi_V\circ i_B\ten\id_V\label{eq:DY-funct-2},\\
\pi_V^*&=&\id_B\ten\xi_V\circ \id_B\ten i_{B^*}\ten\id_V\circ R_B\ten\id_V\label{eq:DY-funct-3}.
\end{eqnarray*}
\end{enumerate}
\end{theorem}

\begin{proof}
The only non trivial part consists in showing that the functors $\Theta$ and $\Xi$
are well--defined. It is then clear that they are inverse of each other and preserve 
the braided tensor structures on $\DrY{B}$ and $\Rep DB$.\\

$(i)$ We have to show that $\xi_V$ defines an action of $DB$ on $V$. 
Namely, we need to check that,
\begin{equation*}
\xi_V\circ\id_{DB}\ten\xi_V\circ i_{B^*}\ten i_{B}\ten\id_V=
\xi_V\circ m_{DB}\ten\id_V\circ i_{B^*}\ten i_{B}\ten\id_V.
\end{equation*}
By definition of $\xi_V$, the LHS reads
\begin{align*}
\xi_V\circ&\id_{DB}\ten\xi_V\circ i_{B^*}\ten i_{B}\ten\id_V\\
&=\iip{}{}^{21}\ten\id_V\circ\pi_V^*\circ\id_{B^*}\ten\pi_V\\
&=\iip{}{}^{21}\ten\id_V\circ\id_{B^*}\ten(\pi^*_V\circ\pi_V)\\
&=\iip{}{}^{21}\ten\id_V\circ\id_{B^*}\ten(m_B\ten\pi_V\circ(3\,2\,1)\circ\btu_B\ten\btd_B\ten\id\circ\Delta_B\ten\pi_V^*)\\
&=(\iip{}{}^{21}\circ\id_{B^*}\ten\iip{}{}^{21}\ten\id_B)\ten\pi_V\circ(4\,5)\circ\id_{B^*}^{\ten 2}\ten\btu_B\ten\btd_B\circ\Delta_{B^*}\ten\Delta_B\ten\pi_V^*.
\end{align*}
Applying \eqref{eq:coadj-dual-1} and \eqref{eq:coadj-dual-3}, we get 
\begin{align*}
\xi_V\circ&\id_{DB}\ten\xi_V\circ i_{B^*}\ten i_{B}\ten\id_V\\
&=\pi_V\circ\id_B\ten\iip{}{}^{21}\ten\id_V\circ\ConB_B\ten\BonC_B\ten\id_B\ten\id_V
\circ(4\,3\,1)\circ\Delta_{B^*}\ten\Delta_B\ten\pi_V^*\\
&=\pi_V\circ\id_B\ten\iip{}{}^{21}\ten\id_V\circ\id_B\ten\id_{B^*}\ten\pi_V^*
\circ\ConB_B\ten\BonC_B\ten\id_V\circ(2\,3)\circ\\
&\qquad\qquad\qquad\qquad\qquad\qquad\qquad\qquad\qquad\qquad
\circ\Delta_{B}\ten\Delta_{B^*}^{21}\ten\id_V\circ(1\,2)\\
&=\xi_V\circ m_{DB}\ten\id_V\circ i_{B^*}\ten i_{B}\ten\id_V
\end{align*}
as required.

$(ii)$ It is clear that $\pi_V$ and $\pi_V^*$ define an action and a coaction of $B$ on $V$. We 
have to show that they satisfy the compatibility condition \eqref{eq:DY-comp-2}. One has
\begin{align*}
\pi_V^*\circ\pi_V&=\id_B\ten\xi_V\circ\id_B\ten\id_{B^*}\ten\xi_V\circ R_B\ten\id_B\ten\id_V\\
&=\id_B\ten\xi_V\circ\id_B\ten m_{DB}\ten\id_V\circ R_B\ten\id_B\ten\id_V\\
&=
m_B\ten\ConB_B\ten\BonC_B\ten\id_V
\circ
(6\,5\,3\,4)
\circ\id^{\ten 4}\ten\Delta_B\ten\id_V\circ \\
&\qquad\qquad\qquad\qquad
\id_{B}\ten R_B\ten\id_{B^*}\ten\id_B\ten\id_V\circ R_B\ten\id_B\ten\id_V.
\end{align*}
Applying \eqref{eq:coadj-dual-2} and $\eqref{eq:coadj-dual-4}$, we get
\begin{align*}
\pi_V^*\circ\pi_V=m_B\ten\pi_V\circ(3\,2\,1)\circ\btu_B\ten\btd_B\ten\id\circ\Delta_B\ten\pi_V^*
\end{align*}
as required.
\end{proof}

\subsection{QUE algebras}\label{ss:app-5}

Let now $B$ be a QUE algebra, and $B'\subset B$ the corresponding
QFSH algebra (see \S\ref{ss:dualityQUE}). In \S\ref{ss:adm-Mm} and
\S\ref{ss:Mpv}, we proved the following

\begin{proposition}\hfill
\begin{enumerate}
\item The adjoint action of $B$ on itself preserves $B'$. In particular, 
$(B',\btd_B)$ is a left $B$--module.
\item The adjoint coaction of $B$ on itself factors through $B'$. In particular,
$(B,\btu_B)$ is an admissible right $B$--comodule.
\end{enumerate}
\end{proposition}

Set $B^{\vee}={(B')^*}^{\scs\operatorname{cop}}$. The adjoint action of $B$ on $B'$
induces a coadjoint action $\BonC_B: B\ten B^{\vee}\to B^{\vee}$ of $B$ on $B^{\vee}$.
Similarly, the adjoint coaction of $B'$ on $B$ induces a coadjoint action 
$\ConB_B:B\ten B^{\vee}\to B$ of $B^{\vee}$ on $B$. 
One checks easitly that the tuple $(B, B^{\vee}, \BonC_B, \ConB_B)$ is a matched 
pair of Hopf algebras and the QUE quantum double of $B$ is precisely the double
cross product $DB=\dcp{B}{B^{\vee}}$.

The computations carried out in \S\ref{ss:app-2}--\S\ref{ss:app-4} apply verbatim
to this case and yield the following.

\begin{theorem}
The formulae from Theorem \ref{thm:DY-qd-eq} give an equivalence of braided tensor categories
between the category $\aDrY{B}$ of admissible Drinfeld--Yetter $B$--modules and the category
$\Rep DB$ of modules over the quantum double.
\end{theorem}


\end{document}